\documentclass[11pt,a4paper]{article}
\usepackage[margin=28mm,headsep=7mm]{geometry}
\usepackage[utf8]{inputenc} 
\usepackage[T1]{fontenc}    
\usepackage{lmodern}
\usepackage{booktabs}       
\usepackage{amsfonts}       
\usepackage{mathrsfs}         
\usepackage{amsmath,amssymb,bbm}
\usepackage{amsthm}
\usepackage{microtype}      
\usepackage{tikz-cd}
\usepackage{graphicx,adjustbox}
\usepackage{hyperref}
\hypersetup{hidelinks,pdfauthor={Keyao Peng},
  pdftitle={Defect TQFTs, the Generalized Cobordism Hypothesis, and Algebraic K-Theory}}
\newtheorem{theorem}{Theorem}[section]

\newtheorem{proposition}[theorem]{Proposition}
\newtheorem{lemma}[theorem]{Lemma}

\newtheorem{conjecture}[theorem]{Conjecture}
\newtheorem{definition}[theorem]{Definition}

\newtheorem{example}[theorem]{Example}
\newtheorem{remark}[theorem]{Remark}

\newcommand{\cA}{\mathcal{A}}
\newcommand{\cB}{\mathcal{B}}
\newcommand{\cC}{\mathcal{C}}
\newcommand{\cD}{\mathcal{D}}
\newcommand{\cE}{\mathcal{E}}
\newcommand{\bZ}{\mathbb{Z}}
\newcommand{\bR}{\mathbb{R}}
\newcommand{\bO}{\mathbb{O}}

\newcommand{\one}{\mathbf{1}}

\newcommand{\GL}{\mathrm{GL}}

\newcommand{\St}{\mathrm{St}}

\newcommand{\rFun}{\mathrm{Fun}}
\newcommand{\rMap}{\mathrm{Map}}
\newcommand{\End}{\mathrm{End}}
\newcommand{\Obj}{\mathrm{Obj}}
\newcommand{\Mor}{\mathrm{Mor}}
\newcommand{\colim}{\operatorname*{colim}}
\newcommand{\hofib}{\operatorname{hofib}}
\newcommand{\Ker}{\operatorname{Ker}}
\newcommand{\Spc}{\mathcal{S}pc}

\newcommand{\Bord}{\mathrm{Bord}}
\newcommand{\Lasa}{\mathrm{Lasa}}
\newcommand{\Loc}{\mathrm{Loc}}
\newcommand{\Cat}{\mathrm{Cat}}
\newcommand{\SMCat}{\mathrm{SMCat}}
\newcommand{\Vect}{\mathrm{Vect}}
\newcommand{\fd}{\mathrm{fd}}
\newcommand{\fr}{\mathrm{fr}}
\newcommand{\op}{\mathrm{op}}

\newcommand{\xr}[1]{\xrightarrow{#1}}

\title{Defect TQFTs, the Generalized Cobordism Hypothesis,
and Algebraic $K$-Theory}
\author{Keyao Peng}
\date{}
\begin{document}
\maketitle

\begin{abstract}
By constructing a fully dualizable category $\mathrm{Bord}^{\mathcal{C}}_n$ as the left adjoint of the forgetful functor, we state a generalized cobordism hypothesis. As a corollary, we establish a generalized Pontryagin–Thom isomorphism that gives a geometric interpretation of higher algebraic $K$-theory using cobordism classes of framed manifolds and stratified spaces decorated by categorical string diagrams, which we call lasagnas. For finite-dimensional vector spaces over a field, these decorations are branched local systems.

We also examine low-dimensional representatives, including $K_1$, Steinberg-group presentations of $K_2$, and fillings related to the Steinberg relation, providing a more explicit construction of Ghys' manifold.
\end{abstract}

\tableofcontents

\section{Introduction}\label{sec:intro}

What is an element of higher algebraic $K$-theory? This is the
question that motivated the present work. An element of $K_0$ has a familiar
interpretation in terms of objects and the relations defining their
Grothendieck group. For the higher groups, it is natural to look for
geometric witnesses of those relations, then witnesses relating those
witnesses. The path-based viewpoint of homotopy type theory helped suggest
this approach \cite{HoTTBook}. The discovering of elementary picture
\(0=1-1=-1+1=0\) for Hopf fibration \cite{0-eq-0}, together with the relation between framed cobordism,
the sphere spectrum, and the $K$-theory of finite sets, led me to look for a
similar description for symmetric monoidal categories
\cite{GGN,WeibelKbook}.

The aim is to represent such witnesses by manifolds and stratified spaces
carrying string diagrams of coefficient objects and morphisms. A composition
of morphisms is represented by a local pasting; duals and adjoints appear
through framed bends, cups, and caps. This does not make the computation of
$K$-theory easy. It translates the problem into the construction and
classification of geometric objects with specified local structure. No new
computation of the algebraic $K$-groups considered here is claimed.

\paragraph{Defect bordisms and the generalized cobordism hypothesis.}
For an $(\infty,n)$-category $\cC$, the proposed category $\Bord_n^{\cC}$
consists of framed bordisms carrying compatible $\cC$-defects. The central
universal property is
\[
 \rFun^\otimes(\Bord_n^{\cC},\cD)
 \simeq \rFun(\cC,U\cD^{\fd}),
\]
where $\cD^{\fd}$ is the maximal symmetric monoidal subcategory with all
$n$-duals. Thus the construction is intended to be left adjoint to the
forgetful functor from categories with duals. For $\cC=*$, this is the
ordinary framed cobordism hypothesis \cite{Lurie}.

We describe the coefficient diagrams using opetopes and their equipped
Poincar\'e-dual fans. The cobordism hypothesis with singularities gives the
local cell-attachment calculation used in the unmarked opetopic case
(Proposition~\ref{prop:Er} and Theorem~\ref{thm:opetope}). Universal markings
are intended to encode invertible comparison cells
(Theorem~\ref{thm:u-opetope}). The passage from local shapes to a general
coefficient category is expressed by the colimit in
Theorem~\ref{thm:gch}. This approach builds on the opetopic description of
composition \cite{BaezDolan,ChengUniversality} and on geometric locality for
extended field theories \cite{GPGeometric,GPLocality}.

\paragraph{Branefolds, lasagna, and generalized Pontryagin--Thom isomorphism.}
For pointed coefficients, the distinguished object is identified with the
vacuum. For monoidal coefficients, delooping places the tensor operations
inside the categorical diagrams. Removing the resulting vacuum strata
leaves branching supports, which we call \emph{branefolds}; their decorated
bordisms are \emph{lasagnas}. Under the compatibility and stable-realization
conditions stated in Theorem~\ref{thm:lch}, they satisfy
\[
 \rFun^\otimes(\Lasa_n^{\cC,\otimes},\cD)
 \simeq \rFun^\otimes(\cC,\cD^{\fd}).
\]
Applying this universal property to additive Picard $\infty$-groupoids gives
the spectrum comparison of Theorem~\ref{thm:gPT} as a generalized Pontryagin--Thom isomorphism:
\[
 K(\cC)\simeq
 \operatorname{sp}\bigl|\Lasa_n^{\cC^{\cong},\otimes}\bigr|.
\]
Here $K(\cC)$ is defined by group completion of the symmetric monoidal
core. The additional geometric representation property in that theorem
identifies the indicated low-dimensional homotopy classes with lasagna
bordism classes.

\paragraph{Examples and the scope of the construction.}
For $\cC=(\Vect_F,\oplus)$, lasagna decorations are local systems on the
smooth pieces, with direct-sum identifications at branches
(Proposition~\ref{prop:ex-local-systems}). The examples recover the signed
rank description of $K_0(F)$ and describe $K_1(F)$ by monodromy, including
the stable framing correction. For $K_2(F)$, we study punctured surfaces,
the Steinberg group, and local surgery moves. Proposition~\ref{prop:ex-OCT}
gives an explicit filling for a pair of commuting elementary matrices.
The subsequent Ghys-type pictures and the five-term conjecture
(Conjecture~\ref{conj:ex-five-term}) describe further geometric constructions
suggested by the algebra.

The construction proceeds from local defect types to general coefficient
categories. The singularity cobordism hypothesis supplies the input for
the opetopic induction; the pointed and monoidal constructions then use
colimits and delooping. The lasagna and Pontryagin--Thom comparisons carry
the compatibility and realization conditions stated in
Section~\ref{sec:lasagna}. The framed links, collars, and coefficient
relations remain part of the geometric data throughout.

This viewpoint is closely related to the decorated foams of
Gepner--Im--Khovanov--Kitchloo, who connect flat projective-module
decorations and foam cobordisms with algebraic $K$-theory
\cite{GepnerImKhovanovKitchloo}. The present project emphasizes a
higher-categorical universal property and its pointed and monoidal
versions. Its terminology is also motivated by skein lasagna constructions
\cite{MorrisonWalkerWedrich}, although the intrinsic framed supports used
here are different from embedded surfaces in a fixed four-manifold.

\paragraph{Organization and conventions.}
Section~\ref{sec:bg} recalls the background. Section~\ref{sec:gch} develops
the opetopic and general defect constructions; Section~\ref{sec:lasagna}
passes to lasagnas and the spectrum comparison. Section~\ref{sec:examples}
studies vector-space coefficients, and Section~\ref{sec:further} lists
further directions. We write $(-)^{\cong}$ for the core and $(-)^{\fd}$ for
the maximal subcategory with all required duals. Throughout, $\rFun$ and
$\rMap$ denote spaces, rather than full functor higher-categories.

\paragraph{History and acknowledgments.}
The idea for this work dates to March 2024. I discussed it with Nils
Carqueville and received helpful feedback. The project was then postponed
because I could not construct the desired geometric presentations of the
relations in $K_2$. At the SMF Congress in June 2025, I learned from Julien
March\'e about the related problem of constructing Ghys manifolds
\cite{overcommuting}. Progress on these constructions in summer 2026 led me
to return to the project.

My original plan used factorization homology
\cite{AyalaFrancisRozenblyum} and the geometric cobordism hypothesis
\cite{GPGeometric}. John Francis pointed out difficulties with this use of
factorization homology and offered useful suggestions. While completing
the manuscript, I became aware of \cite{GepnerImKhovanovKitchloo}, whose
geometric viewpoint and predictions are closely related to this project.
Although the work was developed independently, the proposed universal
properties here provide a way to approach some of the questions raised
there.

The difficulties of finding collaborators and sustaining an academic
career affected the development of this project. I used large language
model tools to help organize the arguments and develop details of the
proof ideas. The manuscript is offered for further examination and
improvement; corrections and counterexamples are welcome. I especially
thank my friend Qian Tang for his help.

This work was supported by the Institut de Math\'ematiques de Bourgogne
(IMB), UMR 5584 CNRS and Universit\'e Bourgogne Europe. This material is also
based upon work supported by the National Science Foundation under Grant
No.~DMS-1928930, while the author was in residence at the Simons Laufer
Mathematical Sciences Institute (formerly MSRI) in Berkeley, California,
during the fall semester of 2026.

\section{Background}\label{sec:bg}
We recall the conventions and background used in the construction. The main
inputs are the framed cobordism hypothesis, the description of composition by
universal opetopic cells, and the relation between symmetric monoidal
groupoids and connective spectra.

\subsection{Higher categories and duals}
Fix $n\geq1$. Write $\Cat_n=\Cat_{(\infty,n)}$ and let $\SMCat_n$ be the category of small
symmetric monoidal $(\infty,n)$-categories and strong symmetric monoidal
functors. We write $U$ for the forgetful functor. The core $\cC^{\cong}$
is the maximal $\infty$-groupoid in $\cC$. Throughout, $\rFun$ and $\rMap$
denote spaces of functors or maps and their equivalences; thus
$\rFun(\cC,\cD)=\operatorname{Fun}(\cC,\cD)^{\cong}$. All limits, colimits,
fibers, and mapping spaces are derived.

A symmetric monoidal $(\infty,n)$-category has \emph{all $n$-duals} if its
objects have tensor duals and every $r$-morphism, $1\leq r<n$, has left and
right adjoints. These conditions hold internally, so the units and counits
have the adjoints required in their own degrees. Let $\SMCat_n^{\fd}$ denote
this subcategory. For a general $\cD$, write $\cD^{\fd}$ for its maximal
replete symmetric monoidal subcategory with all $n$-duals. This restricts
morphisms as well as objects; it is not merely the full subcategory on fully
dualizable objects \cite[Section 2.3]{Lurie}.

An adjoint is not an inverse: its unit and counit need not be equivalences.
Moreover, although tensor symmetry identifies the left and right duals of an
object, it does not generally identify the left and right adjoints of a
morphism. We retain the framing data of bends and their higher homotopies.

\subsection{The cobordism hypothesis and framed defects}
An $n$-framing of a $k$-manifold, $k\leq n$, is a trivialization
\[
 \tau:TM\oplus\underline{\bR}^{\,n-k}
 \xrightarrow{\cong}\underline{\bR}^{\,n}.
\]
The fully extended category $\Bord_n^{\fr}$ has framed points as objects,
collared framed $k$-bordisms as geometric $k$-morphisms, and diffeomorphisms,
isotopies, and their higher families above degree $n$. Its tensor product is
disjoint union. Evaluation at the positively framed point gives the framed
cobordism hypothesis
\begin{equation}\label{eq:bg-ch}
 \rFun^\otimes(\Bord_n^{\fr},\cD)
 \simeq (\cD^{\fd})^{\cong}.
\end{equation}
Thus $\Bord_n^{\fr}$ is free with all $n$-duals on one object
\cite[Theorem 2.4.6]{Lurie}. The geometric construction and the locality
of field theories are developed by Grady--Pavlov
\cite{GPGeometric,GPLocality}.

For defects, a codimension-$r$ stratum $W$ has a specified normal model
$W\times C(\Lambda)$, where $\Lambda$ is a closed structured
$(r-1)$-dimensional link. Its residual framing has rank $n-r$. The link,
radial direction, tubular identifications, and compatibility on incident
strata are part of the singularity datum. For a discrete collection of new
framed defect labels $e$, the relative singularity theorem identifies the
space of extensions of a fixed lower-depth theory $Z$ with
\begin{equation}\label{eq:bg-relative-ch}
 \operatorname{hofib}_{Z}\!\left[
 \rFun^\otimes(\mathcal B_{\leq r},\cD)
 \longrightarrow\rFun^\otimes(\mathcal B_{<r},\cD)\right]
 \simeq
 \prod_e\rMap_{\Omega^{r-1}\cD}(\one,Z(\Lambda_e)),
 \qquad \cD\in\SMCat_n^{\fd}.
\end{equation}
Here $\Omega^{r-1}\cD$ is the iterated endomorphism category of the unit,
and the displayed hom space is the core of its hom higher-category.
We use this resolved-cone form of the cobordism hypothesis with singularities
as a theorem input \cite[Theorem 4.3.11]{Lurie}. Applying it to a
normal-map or fan model requires a comparison with the specified singularity
data. 

\subsection{Opetopic sets, spaces, and universal cells}
An opetope is a higher-dimensional shape with a many-input, one-output
boundary. Its input faces are organized by a pasting tree, while its output
is one face of the same degree. For example, a binary $2$-opetope has the
form
\[
 x\xrightarrow{f}y\xrightarrow{g}z,
 \qquad h:x\longrightarrow z,
 \qquad \alpha:(f,g)\Longrightarrow h.
\]
It specifies a comparison from the input pasting to $h$, rather than the
strict equation $h=gf$. These shapes and their face maps form a category
$\bO$ \cite{BaezDolan,ChengOpetopes}.

\begin{definition}[Opetopic sets and spaces]\label{def:bg-opetopic}
An \emph{opetopic set} is a presheaf
$X:\bO^{\op}\to\mathrm{Set}$. An \emph{opetopic space} is a presheaf
$X:\bO^{\op}\to\Spc$. A point of $X(O)$ is a cell of shape $O$; its faces
are obtained by restriction along the face maps. For a fixed boundary
$b\in\rMap(\partial y(O),X)$, the space of fillers is
\[
 \operatorname{Fill}_X(b)
 :=\operatorname{hofib}_b\!\left[
 X(O)\longrightarrow\rMap(\partial y(O),X)\right],
\]
where $\partial y(O)$ is generated by the proper faces of $O$.
\end{definition}

A \emph{niche} specifies the complete input pasting and its lower boundary,
but leaves the output cell and its filler unspecified. A universal filler
selects a composite of the input. Nullary input pastings provide identities;
their lower boundary data must still be retained. Merely having fillers is
not enough: substitution of universal fillers must again be universal, with
coherent units and associativity \cite{BaezDolan}.

We use universality in a coherent representable model. If
$u:A\Rightarrow b$, let $\mathcal H(A;c)$ denote the higher category of cells
from the input pasting $A$ to a compatible output $c$. The factorization
criterion for universality is that substitution induce equivalences
\begin{equation}\label{eq:bg-universal}
 u^*:\mathcal H(b;c)\xrightarrow{\simeq}\mathcal H(A;c)
\end{equation}
for all compatible outputs and surrounding contexts. These are equivalences
of the full hom higher-categories, not only of their sets of components.
Without chosen composites, substitution is expressed by a composition span.
Cheng develops this formulation and compares it with the recursive
Baez--Dolan definition in dimensions at most two
\cite{ChengUniversality}.

\begin{lemma}[Universal cells and comparison equivalences]\label{lem:bg-universal}
In a coherent representable model, choose a universal compositor
$c_A:A\Rightarrow\langle A\rangle$. Write
\[
 \alpha:A\Rightarrow b,
 \qquad
 \alpha\simeq\bar\alpha\circ c_A,
 \qquad
 \bar\alpha:\langle A\rangle\Rightarrow b.
\]
Then $\alpha$ is universal if and only if $\bar\alpha$ is an equivalence.
\end{lemma}
\begin{proof}
Substitution factors as $\alpha^*=c_A^*\bar\alpha^*$. Since $c_A^*$ is an
equivalence, the universality criterion for $\alpha$ is equivalent to
$\bar\alpha^*$ being an equivalence in every context. The representable
criterion for equivalences, applied in the relevant hom higher-category,
gives the assertion.
\end{proof}

For the binary example, universality means that
$\bar\alpha:gf\Rightarrow h$ is invertible. It does not require
$f$, $g$, or $h$ to be invertible. A collection of universal cells is
therefore not generally a subpresheaf: their input faces need not be universal.

\subsection{Higher universality and categorical realization}
We fix a coherent, complete, representable opetopic model: compositions are
represented by universal fillers, substitution preserves universality, and
equivalences are saturated. Completeness means that the homotopy data in the
model correctly identify categorical equivalences. A comparison with a
standard model of higher categories, such as complete $\Theta_n$-spaces
\cite{Rezk}, is part of the setup.

\begin{definition}[Opetopic $(\infty,n)$ convention]\label{def:bg-infty-n}
Within this fixed model, an \emph{opetopic $(\infty,n)$-category} is a
coherent representable opetopic space in which every cell of dimension
$r>n$ is universal. All higher cells and their homotopies are retained.
\end{definition}

By Lemma~\ref{lem:bg-universal}, this means that the corresponding unary
morphisms above dimension $n$ are equivalences. Conversely, higher
invertibility makes these higher comparison cells universal. This is the
usual meaning of $(\infty,n)$ in the chosen realization. It is not an
axiomatization of arbitrary marked presheaves by this condition alone.
In particular, we do not impose uniqueness of higher fillers: the finite
opetopic $n$-category definitions use a uniqueness condition above $n$, whereas
here those degrees carry higher homotopy \cite{ChengUniversality}.

For an equipped shape $O$, write $[O]$ for its coherent categorical
realization. We use the notation $\bO$ also for the chosen equipped shape
category, including the substitutions required for the realization. Assume
that the associated nerve is fully faithful and dense:
\[
 N_{\bO}:\Cat_n\longrightarrow\operatorname{PSh}(\bO),
 \qquad
 N_{\bO}(\cC)(O)=\rFun([O],\cC).
\]
These are assumptions on the model, not consequences of
Definition~\ref{def:bg-opetopic}.

For an opetopic space $X$, its category of elements is
$I_X=\int_{\bO}X$. An object is $(O,x)$ with $x\in X(O)$; an arrow
$(P,p)\to(O,x)$ consists of $\beta:P\to O$ and a path
$p\simeq\beta^*x$, with all higher coherence. Presheaf density, and then the
chosen realization, give
\begin{equation}\label{eq:bg-density}
 X\simeq\colim_{(O,x)\in I_X}y(O),
 \qquad
 \cC\simeq\colim_{(O,c)\in I_{N_{\bO}\cC}}[O].
\end{equation}
This reduces a colimit-preserving construction to its values on shapes.
For a geometric construction, the corresponding descent theorem must still
be proved: the existence of one mesh decomposition does not identify all
its refinements and parameter families.

If $u$ is a collection of positive-dimensional marked cells, define
$[O][u^{-1}]$ by inverting their unary comparison cells. Its universal
property is
\begin{equation}\label{eq:bg-marking}
 \rFun([O][u^{-1}],\cD)
 \simeq
 \{F\in\rFun([O],\cD):F(e)\text{ is universal for }e\in u\}.
\end{equation}
Coherent inverse data have contractible choice space over an equivalence.
Identifying a proposed geometric passage rule with this localization is a
separate local statement.

\subsection{Pointing, delooping, and group completion}
A pointed category $(\cC,c_*)$ has a specified object. A pointed functor
comes with an equivalence $F(c_*)\simeq d_*$, so
\[
 \rFun_*((\cC,c_*),(\cD,d_*))
 \simeq
 \rFun(\cC,\cD)
 \times^h_{\cD^{\cong}}\{d_*\}.
\]
In the bordism construction, the distinguished coefficient object is sent
to the tensor unit. Its pure identity decoration represents vacuum;
nonidentity endomorphisms of that object must remain as scalar coupons.

For an $E_m$-monoidal $(\infty,n)$-category $\cC$, its $m$-fold delooping
$B^m\cC$ has a distinguished tower of units below degree $m$, and
\[
 \operatorname{End}^m(B^m\cC)\simeq\cC
\]
as $E_m$-monoidal categories. We use connected deloopings; recovering an
arbitrary pointed category from its iterated endomorphisms requires a
connectivity hypothesis. In the symmetric monoidal case this can be
iterated compatibly \cite[Remark 4.3.18]{Lurie}. The shift explains the
geometry of lasagna: in $B\cC$, regions carry the unique object, while walls
and junctions carry objects and morphisms of $\cC$. Erasing vacuum regions
leaves the visible branefold support and its retained normal data.

For a symmetric monoidal category $\cC$, the core $\cC^{\cong}$ is an
$E_\infty$-space. Its group completion is grouplike and hence determines a
connective spectrum. Our convention is
\begin{equation}\label{eq:bg-K}
 \Omega^\infty K(\cC)=(\cC^{\cong})^{\mathrm{gp}}.
\end{equation}
For a connective spectrum $E$, let $\mathbb P(E)=\Omega^\infty E$ with its
additive Picard structure. Group completion is characterized by
\[
 \rMap_{\mathrm{Sp}_{\geq0}}(K(\cC),E)
 \simeq
 \rMap_{E_\infty}(\cC^{\cong},\Omega^\infty E)
\]
\cite{GGN}.
For $\cC=(\operatorname{Proj}_R,\oplus)$ this is the usual connective
$K(R)$, since exact sequences of finitely generated projectives split
\cite[Theorem IV.7.1]{WeibelKbook}. For a general exact category, group
completion of the core with direct sum need not recover its nonsplit
exact-category $K$-theory.

Finally, $|\mathcal B|$ denotes groupoid completion of the \emph{entire}
category $\mathcal B$, whereas $\mathcal B^{\cong}$ discards its
noninvertible morphisms. If $\mathcal B$ is symmetric monoidal and its
objects have duals, evaluation makes $[b^\vee]$ an additive inverse of
$[b]$ in $\pi_0|\mathcal B|$. Thus $|\mathcal B|$ is grouplike. We write
$\operatorname{sp}(|\mathcal B|)$ for its associated connective spectrum;
it is not the suspension spectrum of its underlying space. This convention
will be used in the Pontryagin--Thom and $K$-theory statements below.

\section{Generalized Cobordism Hypothesis}\label{sec:gch}
Let $\Cat_n$ denote the category of small $(\infty,n)$-categories, and let
$\Cat_n^{\fd}$ denote its subcategory of categories whose morphisms in degrees
$1,\ldots,n-1$ have left and right adjoints. Similarly, write
$\SMCat_n$ for symmetric monoidal $(\infty,n)$-categories and
$\SMCat_n^{\fd}$ for those with all $n$-duals: objects have tensor duals, and
morphisms below degree $n$ have left and right adjoints. Throughout,
$\rFun$ and $\rMap$ denote spaces of functors or maps and their equivalences;
all limits, colimits, and fibers are derived.

We seek a geometric construction of the left adjoint
\[
  \Cat_n
  \substack{\overset{L_n=\Bord_n^{(-)}}{\longrightarrow}\\
  \underset{U=\mathrm{forget}}{\longleftarrow}}
  \SMCat_n^{\fd}.
\]
Informally, $\Bord_n^{\cC}$ consists of framed manifolds carrying string
diagrams, or defects, compatibly labeled by $\cC$. For a symmetric monoidal
$(\infty,n)$-category $\cD^\otimes$, define the space of fully extended field
theories with $\cC$-defects by
\[
  \mathrm{TFT}^{\cC}(\cD^\otimes)
  :=\rFun^\otimes(\Bord_n^{\cC},\cD^\otimes).
\]
Once the construction and its universal property are established, we obtain
\[
\begin{aligned}
  \rFun^\otimes(\Bord_n^{\cC},\cD^\otimes)
  &\simeq\rFun^\otimes(\Bord_n^{\cC},\cD^{\otimes,\fd})\\
  &\simeq\rFun(\cC,U\cD^{\otimes,\fd}).
\end{aligned}
\]
The first equivalence follows because $\Bord_n^{\cC}$ has all $n$-duals;
the second is the desired adjunction. Here $\cD^{\otimes,\fd}$ is the maximal
replete symmetric monoidal subcategory with all $n$-duals, restricting
morphisms as well as objects. Taking $\cC=*$ recovers the framed cobordism
hypothesis (CH) \cite{Lurie}. We henceforth suppress $U$ on underlying targets
when no ambiguity arises.

We work with a coherent opetopic presentation of $\Cat_n$, motivated by
\cite{BaezDolan,ChengUniversality}. Fix a category of equipped shapes $\bO$ and coherent
realizations $[O]$, with fully faithful, dense nerve
\[
  N_{\bO}:\Cat_n\longrightarrow\mathrm{PSh}(\bO),
  \qquad
  \cC\longmapsto\bigl(O\longmapsto\rFun([O],\cC)\bigr).
\]
The realization and density properties are assumptions on the chosen model.

For an opetopic space $C:\bO^{\op}\to\Spc$, let
\[
 I_C=\int_{\bO}C.
\]
An object of $I_C$ is a pair $(O,c)$ with $c\in C(O)$. An arrow
$(P,p)\to(O,c)$ consists of a shape map $\alpha:P\to O$ and an identification
$p\simeq\alpha^*c$, with all higher coherence retained. Thus $I_C\to\bO$ is
the associated right fibration.

Density gives
\[
 C\simeq\colim_{(O,c)\in I_C}y(O),
 \qquad
 \cC\simeq\colim_{(O,c)\in I_{N_{\bO}\cC}}[O].
\]
A left adjoint must therefore satisfy
\[
 \Bord_n^{\cC}
 \simeq\colim_{(O,c)\in I_{N_{\bO}\cC}}\Bord_n^{[O]}.
\]
We first construct $\Bord_n^O:=\Bord_n^{[O]}$ for an equipped opetope $O$,
with the intended universal property
\[
 \rFun^\otimes(\Bord_n^O,\cD^{\fd})
 \simeq\rFun([O],\cD^{\fd})
 =\{\text{coherent $O$-shaped diagrams in }\cD^{\fd}\}.
\]

\subsection{Opetopic Cobordism Hypothesis}\label{subsec:opetopic-ch}

Let $O$ be an $m$-dimensional opetope, initially with $m\leq n$. Choose an
equipped Poincar\'e-dual fan $PD(O)$: a conical stratification of $\bR^m$
with specified compatible normal framings $(e_1,\ldots,e_r)$ on its
codimension-$r$ strata. Write $PD^{(r)}(O)$ for the set of these strata.
The compatible framed link and collar data are part of this choice.

\begin{center}

\tikzset{every picture/.style={line width=0.75pt}} 

\begin{adjustbox}{max totalsize={\linewidth}{0.86\textheight}}
\begin{tikzpicture}[x=0.75pt,y=0.75pt,yscale=-1,xscale=1]

\draw    (80,120) -- (99.26,71.86) ;
\draw [shift={(100,70)}, rotate = 111.8] [color={rgb, 255:red, 0; green, 0; blue, 0 }  ][line width=0.75]    (10.93,-3.29) .. controls (6.95,-1.4) and (3.31,-0.3) .. (0,0) .. controls (3.31,0.3) and (6.95,1.4) .. (10.93,3.29)   ;
\draw    (100,70) -- (148,70) ;
\draw [shift={(150,70)}, rotate = 180] [color={rgb, 255:red, 0; green, 0; blue, 0 }  ][line width=0.75]    (10.93,-3.29) .. controls (6.95,-1.4) and (3.31,-0.3) .. (0,0) .. controls (3.31,0.3) and (6.95,1.4) .. (10.93,3.29)   ;
\draw    (150,70) -- (169.26,118.14) ;
\draw [shift={(170,120)}, rotate = 248.2] [color={rgb, 255:red, 0; green, 0; blue, 0 }  ][line width=0.75]    (10.93,-3.29) .. controls (6.95,-1.4) and (3.31,-0.3) .. (0,0) .. controls (3.31,0.3) and (6.95,1.4) .. (10.93,3.29)   ;
\draw    (80,120) -- (168,120) ;
\draw [shift={(170,120)}, rotate = 180] [color={rgb, 255:red, 0; green, 0; blue, 0 }  ][line width=0.75]    (10.93,-3.29) .. controls (6.95,-1.4) and (3.31,-0.3) .. (0,0) .. controls (3.31,0.3) and (6.95,1.4) .. (10.93,3.29)   ;
\draw    (121.5,89.99) -- (121.58,102.66)(118.5,90.01) -- (118.58,102.68) ;
\draw [shift={(120.13,110.67)}, rotate = 269.63] [color={rgb, 255:red, 0; green, 0; blue, 0 }  ][line width=0.75]    (10.93,-3.29) .. controls (6.95,-1.4) and (3.31,-0.3) .. (0,0) .. controls (3.31,0.3) and (6.95,1.4) .. (10.93,3.29)   ;
\draw  [draw opacity=0][fill={rgb, 255:red, 248; green, 231; blue, 28 }  ,fill opacity=1 ] (240,50) -- (320,50) -- (320,130) -- (240,130) -- cycle ;
\draw    (240,50) -- (280,90) ;
\draw    (280,90) -- (280,130) ;
\draw    (280,50) -- (280,90) ;
\draw    (280,90) -- (320,50) ;
\draw [color={rgb, 255:red, 74; green, 144; blue, 226 }  ,draw opacity=1 ]   (280,90) -- (298,90) ;
\draw [shift={(300,90)}, rotate = 180] [color={rgb, 255:red, 74; green, 144; blue, 226 }  ,draw opacity=1 ][line width=0.75]    (10.93,-3.29) .. controls (6.95,-1.4) and (3.31,-0.3) .. (0,0) .. controls (3.31,0.3) and (6.95,1.4) .. (10.93,3.29)   ;
\draw [color={rgb, 255:red, 74; green, 144; blue, 226 }  ,draw opacity=1 ]   (280,90) -- (280,110) ;
\draw [shift={(280,110)}, rotate = 270] [color={rgb, 255:red, 74; green, 144; blue, 226 }  ,draw opacity=1 ][line width=0.75]    (17.64,-3.29) .. controls (13.66,-1.4) and (10.02,-0.3) .. (6.71,0) .. controls (10.02,0.3) and (13.66,1.4) .. (17.64,3.29)(10.93,-3.29) .. controls (6.95,-1.4) and (3.31,-0.3) .. (0,0) .. controls (3.31,0.3) and (6.95,1.4) .. (10.93,3.29)   ;

\draw [color={rgb, 255:red, 74; green, 144; blue, 226 }  ,draw opacity=1 ]   (280,120) -- (298,120) ;
\draw [shift={(300,120)}, rotate = 180] [color={rgb, 255:red, 74; green, 144; blue, 226 }  ,draw opacity=1 ][line width=0.75]    (10.93,-3.29) .. controls (6.95,-1.4) and (3.31,-0.3) .. (0,0) .. controls (3.31,0.3) and (6.95,1.4) .. (10.93,3.29)   ;
\draw [color={rgb, 255:red, 74; green, 144; blue, 226 }  ,draw opacity=1 ]   (250,60) -- (268,60) ;
\draw [shift={(270,60)}, rotate = 180] [color={rgb, 255:red, 74; green, 144; blue, 226 }  ,draw opacity=1 ][line width=0.75]    (10.93,-3.29) .. controls (6.95,-1.4) and (3.31,-0.3) .. (0,0) .. controls (3.31,0.3) and (6.95,1.4) .. (10.93,3.29)   ;
\draw [color={rgb, 255:red, 74; green, 144; blue, 226 }  ,draw opacity=1 ]   (280,60) -- (298,60) ;
\draw [shift={(300,60)}, rotate = 180] [color={rgb, 255:red, 74; green, 144; blue, 226 }  ,draw opacity=1 ][line width=0.75]    (10.93,-3.29) .. controls (6.95,-1.4) and (3.31,-0.3) .. (0,0) .. controls (3.31,0.3) and (6.95,1.4) .. (10.93,3.29)   ;
\draw [color={rgb, 255:red, 74; green, 144; blue, 226 }  ,draw opacity=1 ]   (300,70) -- (318,70) ;
\draw [shift={(320,70)}, rotate = 180] [color={rgb, 255:red, 74; green, 144; blue, 226 }  ,draw opacity=1 ][line width=0.75]    (10.93,-3.29) .. controls (6.95,-1.4) and (3.31,-0.3) .. (0,0) .. controls (3.31,0.3) and (6.95,1.4) .. (10.93,3.29)   ;

\draw (121,123.4) node [anchor=north west][inner sep=0.75pt]    {$O$};
\draw (260,142.4) node [anchor=north west][inner sep=0.75pt]    {$PD( O)$};

\end{tikzpicture}
\end{adjustbox}

\end{center}

Intersecting the fan with a small unit sphere gives its framed stratified
link $\Lambda(O)$. We write $\Lambda(O)=\partial PD(O)$ for this link,
although the open fan itself has no boundary. Then
\[
 PD(O)\cong\mathring C(\Lambda(O)).
\]
For a stratum $S$, let $O_S$ denote its prescribed local normal shape.
The chosen normal germs identify a neighborhood of $S$ with
\[
 S\times PD(O_S)
 \cong S\times\mathring C(\Lambda(O_S)).
\]
These identifications are required to agree on incident strata and on
punctured tubes. We use the resolved-cone model specified by these data.
For a manifold with corners, the chosen core faces $Q$ include the interior,
boundary faces, and their corner intersections.

\begin{definition}[Manifold with $O$-stratification]\label{def:Ostr}
Let $0\leq k\leq n$, and let $(M,\tau)$ be a $k$-manifold with corners and
an $n$-framing $\tau=(v_1,\ldots,v_n)$ of
$TM\oplus\underline{\bR}^{\,n-k}$.
An \emph{resolved $O$-stratification} is a locally finite stratification
$\mathcal W=\{W_t\}$ with label maps
\[
 l:\mathcal W^{(r)}\longrightarrow PD^{(r)}(O)
\]
and the following compatible data for every codimension-$r$ stratum $W_t$:
\begin{enumerate}
  \item a residual $(n-r)$-framing
  \[
  \tau_t:TW_t\oplus\underline{\bR}^{\,n-k}
  \xrightarrow{\cong}\underline{\bR}^{\,n-r};
  \]
  \item a specified stratified tubular germ
  \[
  N(W_t)\cong PD(O_{l(W_t)})\times W_t,
  \]
  compatible with the ambient framing through the normal splitting
  $(e_1,\ldots,e_r,\tau_t)$; these germs agree on punctured tubes and
  iterated incidence charts;
  \item neat intersections with every chosen boundary or corner face $Q$,
  of dimension $\dim Q-r$, with empty intersection if this is negative.
  All structures are products on the corresponding collars.
\end{enumerate}

An associated \emph{$O$-normal map} is a germ
\[
 F:\widehat M:=\operatorname{germ}_{M\times0}
 (M\times\bR^{n-k})\longrightarrow PD(O),
\]
using the boundary collar extensions. For every chosen core face $Q$,
every stratum $S\in PD^{(r)}(O)$, and every $x\in Q\cap F^{-1}(S)$,
require:
\begin{enumerate}
  \item \emph{Transversality:} the normal derivative
  \[
  T_xQ\xr{d(F|_Q)_x}T_{F(x)}\bR^m
  \xr{q_S}N_{F(x)}S
  \]
  is surjective;
  \item \emph{Frame compatibility:}
  \[
    q_S\bigl(dF(v_a)\bigr)=e_a,
    \qquad 1\leq a\leq r.
  \]
\end{enumerate}
Normal bundles are identified with coordinate spaces using their chosen
frames. The inverse images $W_t$ recover the labeled strata. The residual
frame is obtained by projection, not by merely restricting the remaining
ambient vectors. If $\lambda_a(v_j)$ are the coordinates of
$q_S(dF(v_j))$, it is represented by
\[
 w_j=v_j-\sum_{a=1}^r\lambda_a(v_j)v_a,
 \qquad r<j\leq n.
\]
Identifying this map presentation with a resolved stratification includes
the compatible tubular data specified above. Every slice of an admissible
parameter family satisfies the same core and face conditions.

Define $\Bord_n^O$ to be the completed symmetric monoidal
$(\infty,n)$-category of compact collared $n$-framed bordisms with these
resolved $O$-stratifications. Composition is collar gluing and tensor
product is disjoint union; the construction retains the spaces of
isotopies and higher families.
\end{definition}

Normal maps are useful for studying isotopies and automorphisms of these
structures. If $O_{\leq r}$ denotes the cells of dimension at most $r$,
we similarly obtain the lower-depth category $\Bord_n^{O_{\leq r}}$.

\begin{center}

\tikzset{every picture/.style={line width=0.75pt}} 

\begin{adjustbox}{max totalsize={\linewidth}{0.86\textheight}}
\begin{tikzpicture}[x=0.75pt,y=0.75pt,yscale=-1,xscale=1]

\draw  [draw opacity=0][fill={rgb, 255:red, 248; green, 231; blue, 28 }  ,fill opacity=1 ] (260,70) -- (340,70) -- (340,150) -- (260,150) -- cycle ;
\draw    (260,70) -- (300,110) ;
\draw    (300,110) -- (300,150) ;
\draw    (300,70) -- (300,110) ;
\draw    (300,110) -- (340,70) ;
\draw [color={rgb, 255:red, 74; green, 144; blue, 226 }  ,draw opacity=1 ]   (300,110) -- (318,110) ;
\draw [shift={(320,110)}, rotate = 180] [color={rgb, 255:red, 74; green, 144; blue, 226 }  ,draw opacity=1 ][line width=0.75]    (10.93,-3.29) .. controls (6.95,-1.4) and (3.31,-0.3) .. (0,0) .. controls (3.31,0.3) and (6.95,1.4) .. (10.93,3.29)   ;
\draw [color={rgb, 255:red, 74; green, 144; blue, 226 }  ,draw opacity=1 ]   (300,110) -- (300,130) ;
\draw [shift={(300,130)}, rotate = 270] [color={rgb, 255:red, 74; green, 144; blue, 226 }  ,draw opacity=1 ][line width=0.75]    (17.64,-3.29) .. controls (13.66,-1.4) and (10.02,-0.3) .. (6.71,0) .. controls (10.02,0.3) and (13.66,1.4) .. (17.64,3.29)(10.93,-3.29) .. controls (6.95,-1.4) and (3.31,-0.3) .. (0,0) .. controls (3.31,0.3) and (6.95,1.4) .. (10.93,3.29)   ;

\draw [color={rgb, 255:red, 74; green, 144; blue, 226 }  ,draw opacity=1 ]   (300,140) -- (318,140) ;
\draw [shift={(320,140)}, rotate = 180] [color={rgb, 255:red, 74; green, 144; blue, 226 }  ,draw opacity=1 ][line width=0.75]    (10.93,-3.29) .. controls (6.95,-1.4) and (3.31,-0.3) .. (0,0) .. controls (3.31,0.3) and (6.95,1.4) .. (10.93,3.29)   ;
\draw [color={rgb, 255:red, 74; green, 144; blue, 226 }  ,draw opacity=1 ]   (270,80) -- (288,80) ;
\draw [shift={(290,80)}, rotate = 180] [color={rgb, 255:red, 74; green, 144; blue, 226 }  ,draw opacity=1 ][line width=0.75]    (10.93,-3.29) .. controls (6.95,-1.4) and (3.31,-0.3) .. (0,0) .. controls (3.31,0.3) and (6.95,1.4) .. (10.93,3.29)   ;
\draw [color={rgb, 255:red, 74; green, 144; blue, 226 }  ,draw opacity=1 ]   (300,80) -- (318,80) ;
\draw [shift={(320,80)}, rotate = 180] [color={rgb, 255:red, 74; green, 144; blue, 226 }  ,draw opacity=1 ][line width=0.75]    (10.93,-3.29) .. controls (6.95,-1.4) and (3.31,-0.3) .. (0,0) .. controls (3.31,0.3) and (6.95,1.4) .. (10.93,3.29)   ;
\draw [color={rgb, 255:red, 74; green, 144; blue, 226 }  ,draw opacity=1 ]   (320,90) -- (338,90) ;
\draw [shift={(340,90)}, rotate = 180] [color={rgb, 255:red, 74; green, 144; blue, 226 }  ,draw opacity=1 ][line width=0.75]    (10.93,-3.29) .. controls (6.95,-1.4) and (3.31,-0.3) .. (0,0) .. controls (3.31,0.3) and (6.95,1.4) .. (10.93,3.29)   ;
\draw  [dash pattern={on 4.5pt off 4.5pt}]  (259,97) .. controls (299,67) and (299.64,127.42) .. (339.64,97.42) ;
\draw  [dash pattern={on 4.5pt off 4.5pt}]  (259,123) .. controls (290.64,121.42) and (299.64,153.42) .. (339.64,123.42) ;
\draw    (120.67,91.67) .. controls (160.67,61.67) and (161.31,122.09) .. (201.31,92.09) ;
\draw    (120.67,117.67) .. controls (152.31,116.09) and (161.31,148.09) .. (201.31,118.09) ;
\draw  [fill={rgb, 255:red, 0; green, 0; blue, 0 }  ,fill opacity=1 ] (137.25,82.82) .. controls (137.25,81.38) and (138.41,80.22) .. (139.85,80.22) .. controls (141.28,80.22) and (142.45,81.38) .. (142.45,82.82) .. controls (142.45,84.26) and (141.28,85.42) .. (139.85,85.42) .. controls (138.41,85.42) and (137.25,84.26) .. (137.25,82.82) -- cycle ;
\draw  [fill={rgb, 255:red, 0; green, 0; blue, 0 }  ,fill opacity=1 ] (158.58,92.15) .. controls (158.58,90.72) and (159.74,89.55) .. (161.18,89.55) .. controls (162.62,89.55) and (163.78,90.72) .. (163.78,92.15) .. controls (163.78,93.59) and (162.62,94.75) .. (161.18,94.75) .. controls (159.74,94.75) and (158.58,93.59) .. (158.58,92.15) -- cycle ;
\draw  [fill={rgb, 255:red, 0; green, 0; blue, 0 }  ,fill opacity=1 ] (158.91,129.82) .. controls (158.91,128.38) and (160.08,127.22) .. (161.51,127.22) .. controls (162.95,127.22) and (164.11,128.38) .. (164.11,129.82) .. controls (164.11,131.26) and (162.95,132.42) .. (161.51,132.42) .. controls (160.08,132.42) and (158.91,131.26) .. (158.91,129.82) -- cycle ;
\draw  [fill={rgb, 255:red, 0; green, 0; blue, 0 }  ,fill opacity=1 ] (167.25,96.82) .. controls (167.25,95.38) and (168.41,94.22) .. (169.85,94.22) .. controls (171.28,94.22) and (172.45,95.38) .. (172.45,96.82) .. controls (172.45,98.26) and (171.28,99.42) .. (169.85,99.42) .. controls (168.41,99.42) and (167.25,98.26) .. (167.25,96.82) -- cycle ;
\draw [color={rgb, 255:red, 74; green, 144; blue, 226 }  ,draw opacity=1 ]   (139.85,82.82) -- (157.85,82.82) ;
\draw [shift={(159.85,82.82)}, rotate = 180] [color={rgb, 255:red, 74; green, 144; blue, 226 }  ,draw opacity=1 ][line width=0.75]    (10.93,-3.29) .. controls (6.95,-1.4) and (3.31,-0.3) .. (0,0) .. controls (3.31,0.3) and (6.95,1.4) .. (10.93,3.29)   ;
\draw [color={rgb, 255:red, 74; green, 144; blue, 226 }  ,draw opacity=1 ]   (161.18,92.15) -- (173.7,103.23) ;
\draw [shift={(175.2,104.55)}, rotate = 221.49] [color={rgb, 255:red, 74; green, 144; blue, 226 }  ,draw opacity=1 ][line width=0.75]    (10.93,-3.29) .. controls (6.95,-1.4) and (3.31,-0.3) .. (0,0) .. controls (3.31,0.3) and (6.95,1.4) .. (10.93,3.29)   ;
\draw [color={rgb, 255:red, 74; green, 144; blue, 226 }  ,draw opacity=1 ]   (169.85,96.82) -- (186.05,104.38) ;
\draw [shift={(187.87,105.22)}, rotate = 204.99] [color={rgb, 255:red, 74; green, 144; blue, 226 }  ,draw opacity=1 ][line width=0.75]    (10.93,-3.29) .. controls (6.95,-1.4) and (3.31,-0.3) .. (0,0) .. controls (3.31,0.3) and (6.95,1.4) .. (10.93,3.29)   ;
\draw [color={rgb, 255:red, 74; green, 144; blue, 226 }  ,draw opacity=1 ]   (161.51,129.82) -- (179.58,133.79) ;
\draw [shift={(181.53,134.22)}, rotate = 192.4] [color={rgb, 255:red, 74; green, 144; blue, 226 }  ,draw opacity=1 ][line width=0.75]    (10.93,-3.29) .. controls (6.95,-1.4) and (3.31,-0.3) .. (0,0) .. controls (3.31,0.3) and (6.95,1.4) .. (10.93,3.29)   ;

\end{tikzpicture}
\end{adjustbox}

\tikzset{every picture/.style={line width=0.75pt}} 

\begin{adjustbox}{max totalsize={\linewidth}{0.86\textheight}}
\begin{tikzpicture}[x=0.75pt,y=0.75pt,yscale=-1,xscale=1]

\draw  [draw opacity=0][fill={rgb, 255:red, 248; green, 231; blue, 28 }  ,fill opacity=1 ] (280,81.32) -- (360,81.32) -- (360,161.32) -- (280,161.32) -- cycle ;
\draw  [fill={rgb, 255:red, 245; green, 166; blue, 35 }  ,fill opacity=1 ][dash pattern={on 4.5pt off 4.5pt}] (305.71,87.03) .. controls (306.13,87.01) and (306.57,87) .. (307,87) .. controls (323.57,87) and (337,102.67) .. (337,122) .. controls (337,141.33) and (323.57,157) .. (307,157) .. controls (306.51,157) and (306.02,156.99) .. (305.53,156.96) -- cycle ;
\draw    (280,81.32) -- (320,121.32) ;
\draw    (320,121.32) -- (320,161.32) ;
\draw    (320,81.32) -- (320,121.32) ;
\draw    (320,121.32) -- (360,81.32) ;
\draw [color={rgb, 255:red, 74; green, 144; blue, 226 }  ,draw opacity=1 ]   (320,121.32) -- (338,121.32) ;
\draw [shift={(340,121.32)}, rotate = 180] [color={rgb, 255:red, 74; green, 144; blue, 226 }  ,draw opacity=1 ][line width=0.75]    (10.93,-3.29) .. controls (6.95,-1.4) and (3.31,-0.3) .. (0,0) .. controls (3.31,0.3) and (6.95,1.4) .. (10.93,3.29)   ;
\draw [color={rgb, 255:red, 74; green, 144; blue, 226 }  ,draw opacity=1 ]   (320,121.32) -- (320,141.32) ;
\draw [shift={(320,141.32)}, rotate = 270] [color={rgb, 255:red, 74; green, 144; blue, 226 }  ,draw opacity=1 ][line width=0.75]    (17.64,-3.29) .. controls (13.66,-1.4) and (10.02,-0.3) .. (6.71,0) .. controls (10.02,0.3) and (13.66,1.4) .. (17.64,3.29)(10.93,-3.29) .. controls (6.95,-1.4) and (3.31,-0.3) .. (0,0) .. controls (3.31,0.3) and (6.95,1.4) .. (10.93,3.29)   ;

\draw [color={rgb, 255:red, 74; green, 144; blue, 226 }  ,draw opacity=1 ]   (320,151.32) -- (338,151.32) ;
\draw [shift={(340,151.32)}, rotate = 180] [color={rgb, 255:red, 74; green, 144; blue, 226 }  ,draw opacity=1 ][line width=0.75]    (10.93,-3.29) .. controls (6.95,-1.4) and (3.31,-0.3) .. (0,0) .. controls (3.31,0.3) and (6.95,1.4) .. (10.93,3.29)   ;
\draw [color={rgb, 255:red, 74; green, 144; blue, 226 }  ,draw opacity=1 ]   (290,91.32) -- (308,91.32) ;
\draw [shift={(310,91.32)}, rotate = 180] [color={rgb, 255:red, 74; green, 144; blue, 226 }  ,draw opacity=1 ][line width=0.75]    (10.93,-3.29) .. controls (6.95,-1.4) and (3.31,-0.3) .. (0,0) .. controls (3.31,0.3) and (6.95,1.4) .. (10.93,3.29)   ;
\draw [color={rgb, 255:red, 74; green, 144; blue, 226 }  ,draw opacity=1 ]   (320,91.32) -- (338,91.32) ;
\draw [shift={(340,91.32)}, rotate = 180] [color={rgb, 255:red, 74; green, 144; blue, 226 }  ,draw opacity=1 ][line width=0.75]    (10.93,-3.29) .. controls (6.95,-1.4) and (3.31,-0.3) .. (0,0) .. controls (3.31,0.3) and (6.95,1.4) .. (10.93,3.29)   ;
\draw [color={rgb, 255:red, 74; green, 144; blue, 226 }  ,draw opacity=1 ]   (340,101.32) -- (358,101.32) ;
\draw [shift={(360,101.32)}, rotate = 180] [color={rgb, 255:red, 74; green, 144; blue, 226 }  ,draw opacity=1 ][line width=0.75]    (10.93,-3.29) .. controls (6.95,-1.4) and (3.31,-0.3) .. (0,0) .. controls (3.31,0.3) and (6.95,1.4) .. (10.93,3.29)   ;
\draw  [fill={rgb, 255:red, 245; green, 166; blue, 35 }  ,fill opacity=1 ] (111.55,158.42) .. controls (111.51,157.69) and (111.49,156.96) .. (111.48,156.22) .. controls (111.3,128.03) and (137.82,105) .. (170.71,104.79) .. controls (203.6,104.57) and (230.41,127.26) .. (230.59,155.45) .. controls (230.6,156.29) and (230.58,157.12) .. (230.54,157.95) -- cycle ;
\draw  [fill={rgb, 255:red, 0; green, 0; blue, 0 }  ,fill opacity=1 ] (144.14,114.18) .. controls (141.7,114.2) and (139.71,112.23) .. (139.69,109.78) .. controls (139.67,107.34) and (141.64,105.35) .. (144.09,105.33) .. controls (146.53,105.32) and (148.52,107.28) .. (148.54,109.73) .. controls (148.55,112.17) and (146.59,114.16) .. (144.14,114.18) -- cycle ;
\draw  [fill={rgb, 255:red, 0; green, 0; blue, 0 }  ,fill opacity=1 ] (117.7,138.82) .. controls (115.26,138.84) and (113.26,136.87) .. (113.25,134.43) .. controls (113.23,131.98) and (115.2,129.99) .. (117.64,129.97) .. controls (120.09,129.96) and (122.08,131.92) .. (122.09,134.37) .. controls (122.11,136.81) and (120.14,138.81) .. (117.7,138.82) -- cycle ;
\draw  [fill={rgb, 255:red, 0; green, 0; blue, 0 }  ,fill opacity=1 ] (145.76,162.46) .. controls (143.32,162.48) and (141.32,160.51) .. (141.31,158.07) .. controls (141.29,155.62) and (143.26,153.63) .. (145.7,153.62) .. controls (148.15,153.6) and (150.14,155.57) .. (150.16,158.01) .. controls (150.17,160.45) and (148.2,162.45) .. (145.76,162.46) -- cycle ;
\draw  [fill={rgb, 255:red, 0; green, 0; blue, 0 }  ,fill opacity=1 ] (224.22,138.13) .. controls (221.78,138.15) and (219.78,136.18) .. (219.77,133.74) .. controls (219.75,131.29) and (221.72,129.3) .. (224.16,129.28) .. controls (226.61,129.27) and (228.6,131.24) .. (228.62,133.68) .. controls (228.63,136.12) and (226.66,138.12) .. (224.22,138.13) -- cycle ;
\draw    (117.67,134.4) -- (224.19,133.71) ;
\draw    (145.73,158.04) -- (170.93,134.05) ;
\draw    (144.11,109.76) -- (170.93,134.05) ;
\draw [color={rgb, 255:red, 74; green, 144; blue, 226 }  ,draw opacity=1 ]   (117.83,131.86) -- (128.84,117.63) ;
\draw [shift={(130.07,116.05)}, rotate = 127.74] [color={rgb, 255:red, 74; green, 144; blue, 226 }  ,draw opacity=1 ][line width=0.75]    (10.93,-3.29) .. controls (6.95,-1.4) and (3.31,-0.3) .. (0,0) .. controls (3.31,0.3) and (6.95,1.4) .. (10.93,3.29)   ;
\draw [color={rgb, 255:red, 74; green, 144; blue, 226 }  ,draw opacity=1 ]   (117.83,131.86) -- (133.64,144.11) ;
\draw [shift={(133.64,144.11)}, rotate = 217.74] [color={rgb, 255:red, 74; green, 144; blue, 226 }  ,draw opacity=1 ][line width=0.75]    (17.64,-3.29) .. controls (13.66,-1.4) and (10.02,-0.3) .. (6.71,0) .. controls (10.02,0.3) and (13.66,1.4) .. (17.64,3.29)(10.93,-3.29) .. controls (6.95,-1.4) and (3.31,-0.3) .. (0,0) .. controls (3.31,0.3) and (6.95,1.4) .. (10.93,3.29)   ;

\draw [color={rgb, 255:red, 74; green, 144; blue, 226 }  ,draw opacity=1 ]   (225.14,133.06) -- (214.46,118.57) ;
\draw [shift={(213.27,116.96)}, rotate = 53.61] [color={rgb, 255:red, 74; green, 144; blue, 226 }  ,draw opacity=1 ][line width=0.75]    (10.93,-3.29) .. controls (6.95,-1.4) and (3.31,-0.3) .. (0,0) .. controls (3.31,0.3) and (6.95,1.4) .. (10.93,3.29)   ;
\draw [color={rgb, 255:red, 74; green, 144; blue, 226 }  ,draw opacity=1 ]   (225.14,133.06) -- (241.24,121.19) ;
\draw [shift={(241.24,121.19)}, rotate = 143.61] [color={rgb, 255:red, 74; green, 144; blue, 226 }  ,draw opacity=1 ][line width=0.75]    (17.64,-3.29) .. controls (13.66,-1.4) and (10.02,-0.3) .. (6.71,0) .. controls (10.02,0.3) and (13.66,1.4) .. (17.64,3.29)(10.93,-3.29) .. controls (6.95,-1.4) and (3.31,-0.3) .. (0,0) .. controls (3.31,0.3) and (6.95,1.4) .. (10.93,3.29)   ;

\draw [color={rgb, 255:red, 74; green, 144; blue, 226 }  ,draw opacity=1 ]   (143.37,109.63) -- (159.69,102.03) ;
\draw [shift={(161.5,101.19)}, rotate = 155.03] [color={rgb, 255:red, 74; green, 144; blue, 226 }  ,draw opacity=1 ][line width=0.75]    (10.93,-3.29) .. controls (6.95,-1.4) and (3.31,-0.3) .. (0,0) .. controls (3.31,0.3) and (6.95,1.4) .. (10.93,3.29)   ;
\draw [color={rgb, 255:red, 74; green, 144; blue, 226 }  ,draw opacity=1 ]   (143.37,109.63) -- (151.81,127.76) ;
\draw [shift={(151.81,127.76)}, rotate = 245.03] [color={rgb, 255:red, 74; green, 144; blue, 226 }  ,draw opacity=1 ][line width=0.75]    (17.64,-3.29) .. controls (13.66,-1.4) and (10.02,-0.3) .. (6.71,0) .. controls (10.02,0.3) and (13.66,1.4) .. (17.64,3.29)(10.93,-3.29) .. controls (6.95,-1.4) and (3.31,-0.3) .. (0,0) .. controls (3.31,0.3) and (6.95,1.4) .. (10.93,3.29)   ;

\draw [color={rgb, 255:red, 74; green, 144; blue, 226 }  ,draw opacity=1 ]   (169.17,133.79) -- (169.33,115.79) ;
\draw [shift={(169.35,113.79)}, rotate = 90.52] [color={rgb, 255:red, 74; green, 144; blue, 226 }  ,draw opacity=1 ][line width=0.75]    (10.93,-3.29) .. controls (6.95,-1.4) and (3.31,-0.3) .. (0,0) .. controls (3.31,0.3) and (6.95,1.4) .. (10.93,3.29)   ;
\draw [color={rgb, 255:red, 74; green, 144; blue, 226 }  ,draw opacity=1 ]   (169.17,133.79) -- (189.16,133.97) ;
\draw [shift={(189.16,133.97)}, rotate = 180.52] [color={rgb, 255:red, 74; green, 144; blue, 226 }  ,draw opacity=1 ][line width=0.75]    (17.64,-3.29) .. controls (13.66,-1.4) and (10.02,-0.3) .. (6.71,0) .. controls (10.02,0.3) and (13.66,1.4) .. (17.64,3.29)(10.93,-3.29) .. controls (6.95,-1.4) and (3.31,-0.3) .. (0,0) .. controls (3.31,0.3) and (6.95,1.4) .. (10.93,3.29)   ;

\draw [color={rgb, 255:red, 74; green, 144; blue, 226 }  ,draw opacity=1 ]   (146.54,158.59) -- (128.55,158.75) ;
\draw [shift={(126.55,158.77)}, rotate = 359.49] [color={rgb, 255:red, 74; green, 144; blue, 226 }  ,draw opacity=1 ][line width=0.75]    (10.93,-3.29) .. controls (6.95,-1.4) and (3.31,-0.3) .. (0,0) .. controls (3.31,0.3) and (6.95,1.4) .. (10.93,3.29)   ;
\draw [color={rgb, 255:red, 74; green, 144; blue, 226 }  ,draw opacity=1 ]   (146.54,158.59) -- (146.37,138.59) ;
\draw [shift={(146.37,138.59)}, rotate = 89.49] [color={rgb, 255:red, 74; green, 144; blue, 226 }  ,draw opacity=1 ][line width=0.75]    (17.64,-3.29) .. controls (13.66,-1.4) and (10.02,-0.3) .. (6.71,0) .. controls (10.02,0.3) and (13.66,1.4) .. (17.64,3.29)(10.93,-3.29) .. controls (6.95,-1.4) and (3.31,-0.3) .. (0,0) .. controls (3.31,0.3) and (6.95,1.4) .. (10.93,3.29)   ;

\end{tikzpicture}
\end{adjustbox}

\end{center}

The proof begins with the following folding lemma. We use $\Lambda_e$ for
the framed link of the coupon labeled by $e$, with the polarization
specified by its complete source and target configurations.

\begin{lemma}[Framed folding]\label{lem:mates}
Let $\cB_{<r}$ be the lower-depth structured bordism category, with coupon
functor $j_{<r}:P_{<r}\to U\cB_{<r}$. Let
$Z:\cB_{<r}\to\cD$ be symmetric monoidal, where
$\cD\in\SMCat_n^{\fd}$. Then the chosen folded link gives a natural
equivalence
\[
 \rMap_{\Omega^{r-1}\cD}\bigl(\one,Z(\Lambda_e)\bigr)
 \simeq
 \rMap_{\partial Zs_e\to\cD}(Zs_e,Zt_e).
\]
Here $Zs_e$ abbreviates $Z(j_{<r}s_e)$, and similarly for $t_e$.
The right-hand side denotes the space of $r$-cells with the specified
entire boundary. Mapping spaces in iterated loop categories mean the
cores of the corresponding hom-categories.
\end{lemma}

\begin{proof}
For $r=1$, tensoring with coevaluation gives
\[
 \rMap_{\cD}(x,y)
 \simeq\rMap_{\cD}(\one,x^\vee\otimes y),
\]
and evaluation gives its inverse. The triangle identities supply the
inverse homotopies. In higher dimensions, bend the boundary faces in the
order specified by the polarized coupon. At each step, the adjunction
equivalence in the relevant hom-category moves one boundary face across
the mapping expression. After all faces are bent, the source is the
iterated tensor unit and the target is the folded link. Each step is
natural in the lower labeling, so their composite proves the assertion.
\end{proof}

We use the following resolved-cone form of the cobordism hypothesis with
singularities as a theorem input.

\begin{theorem}[{\cite[Theorem 4.3.11]{Lurie}}]\label{thm:chs}
Let $\cB_{\leq r-1}$ be a framed singular bordism category. Adjoin finitely
many labeled cones on closed structured links $\Lambda_e$, with stratum
tangential bundle $\underline{\bR}^{\,n-r}$, and let $\cB_{\leq r}$ be the
resulting resolved bordism category. Both categories have all $n$-duals.
For $\cD\in\SMCat_n^{\fd}$ and $Z:\cB_{\leq r-1}\to\cD$,
\[
 \hofib_Z\!\left[
 \rFun^\otimes(\cB_{\leq r},\cD)
 \longrightarrow\rFun^\otimes(\cB_{\leq r-1},\cD)
 \right]
 \simeq
 \prod_e\rMap_{\Omega^{r-1}\cD}
       \bigl(\one,Z(\Lambda_e)\bigr).
\]
\end{theorem}

\begin{proposition}[Geometric cell attachment]\label{prop:Er}
For the resolved-cone model of Definition~\ref{def:Ostr}, let
$\cD\in\SMCat_n^{\fd}$ and fix an evaluation $Z$ of the cells of dimension
less than $r$. Then
\begin{equation}\label{eq:Er}
\begin{aligned}
 &\hofib_Z\!\left[
 \rFun^\otimes(\Bord_n^{O_{\leq r}},\cD)
 \longrightarrow
 \rFun^\otimes(\Bord_n^{O_{\leq r-1}},\cD)
 \right]\\
 &\hspace{20mm}\simeq
 \prod_{e\in O_{(r)}}
 \rMap_{\partial Zs_e\to\cD}(Zs_e,Zt_e),
 \qquad 1\leq r\leq n.
\end{aligned}
\end{equation}
\end{proposition}

\begin{proof}
The chosen normal models are the cones on $\Lambda_e$ appearing in
Theorem~\ref{thm:chs}. Apply that theorem to the specified resolved-cone
data, then apply Lemma~\ref{lem:mates} to each factor. Both equivalences
are natural in $Z$.
\end{proof}

\begin{theorem}[Opetopic Cobordism Hypothesis]\label{thm:opetope}
Let $O$ be an equipped opetope of dimension at most $n$. Under the
resolved-cone hypotheses above, restriction induces a natural equivalence
\[
 \rFun^\otimes(\Bord_n^O,\cD)
 \simeq
 \rFun_{\Cat_n}([O],\cD^{\fd})
\]
for every $\cD\in\SMCat_n$. The generating coupons define
\[
 j_O:[O]\longrightarrow U\Bord_n^O.
\]
\end{theorem}

\begin{proof}
At degree zero,
\[
 \Bord_n^{O_{\leq0}}
 \simeq\coprod_{e\in O_{(0)}}\Bord_n^{\fr},
\]
so the assertion is the ordinary framed CH
\cite[Theorem 2.4.6]{Lurie}. Inductively, Proposition~\ref{prop:Er}
identifies the space of geometric extensions over a fixed lower labeling
with the space of choices for its generating $r$-cells. This is also the
fiber of restriction from $O$-diagrams to their $(r-1)$-skeleta, by the
coherent cellular realization of $[O]$. The base and every fiber are
therefore equivalent, proving the induction. For arbitrary $\cD$, every
functor from the source factors through $\cD^{\fd}$. The universal
generating diagram gives $j_O$.
\end{proof}

For general coefficients, we also need opetopes with universal markings.
In the representable coherent model, a many-input cell is universal when
its unary comparison from a chosen composite is invertible. This does not
require its input cells to be invertible
\cite{BaezDolan,ChengUniversality}.

Let $u$ be a set of positive-dimensional cells of $O$. Write
$[O][u^{-1}]$ for the localization at their unary comparison cells. Thus
\[
\begin{aligned}
 \rFun_{\Cat_n}([O][u^{-1}],\cC)
 \simeq\bigl\{D\in\rFun_{\Cat_n}([O],\cC):
 D(e)\text{ is universal for every }e\in u\bigr\}.
\end{aligned}
\]

To describe a proposed geometric marking, equip $PD_u(O)$ with partial
normal frames. Along a chosen resolved incidence flag, write
$NF_u(t(S))$ for the normal directions inherited from the target $t(S)$,
and let $e_S$ be
the cell corresponding to a codimension-$r$ stratum $S$. Set
\[
 NF_u(S)=
 \begin{cases}
 NF_u(t(S)),&e_S\in u,\\
 NF_u(t(S))\sqcup\{e_r\},&e_S\notin u.
 \end{cases}
\]
Require agreement between the chosen incidence charts. Here $\partial S$
denotes the inherited categorical boundary data, not merely the topological
boundary of $S$. Let
\[
 N_uS=\operatorname{Span}(NF_u(S)),
\]
and write $q_{u,S}$ for the projection to this retained part of the normal
bundle, using the original normal framing.

\begin{definition}[Manifold with $(O,u)$-stratification]\label{def:Oustr}
Let $(M,\tau)$ be a $k$-dimensional manifold as in
Definition~\ref{def:Ostr}. An
\emph{$(O,u)$-normal map} is a smooth germ
\[
 F:\widehat M\longrightarrow PD_u(O)
\]
satisfying, for every chosen core face $Q$ and every $x\in Q\cap F^{-1}(S)$:
\begin{enumerate}
  \item \emph{Partial transversality:}
  \[
  T_xQ\xr{d(F|_Q)_x}T_{F(x)}\bR^m
  \xr{q_{u,S}}N_{u,F(x)}S
  \quad\text{is surjective};
  \]
  \item \emph{Partial frame compatibility:}
  \[
  q_{u,S}\bigl(dF(v_a)\bigr)=e_a
  \qquad\text{for every retained }e_a\in NF_u(S).
  \]
\end{enumerate}
All boundary data are compatible on product collars.

Partial transversality does not imply that the full inverse images of
marked strata are smooth. We therefore use \emph{$(O,u)$-stratification}
for the resulting normal-map field; a resolved stratified presentation,
when available, includes compatible normal germs and their corner data.
If $d=|NF_u(S)|$, the retained normal derivative has a kernel of rank $n-d$.

An \emph{resolved $(O,u)$-stratification} is a locally finite stratification
$\mathcal W=\{W_t\}$ with label maps
\[
 l:\mathcal W^{(r)}\longrightarrow PD^{(r)}(O)
\]
and the following compatible data for every codimension-$r$ stratum $W_t$:
\begin{enumerate}
  \item a residual $(n-d)$-framing
  \[
  \tau_{u,t}:TW_t\oplus\underline{\bR}^{\,n-k}\oplus\underline{\bR}^{\,r-d}
  \xrightarrow{\cong}\underline{\bR}^{\,n-d};
  \]
  \item a specified stratified tubular germ
  \[
  N(W_t)\cong PD(O_{l(W_t)})\times W_t,
  \]
  compatible with the ambient framing through the normal splitting
  $(NF_u(l(W_t)),\tau_{u,t})$; these germs agree on punctured tubes and
  iterated incidence charts;
  \item neat intersections with every chosen boundary or corner face $Q$,
  of dimension $\dim Q-r$, with empty intersection if this is negative.
  All structures are products on the corresponding collars.
\end{enumerate}

The proposed $\Bord_n^{(O,u)}$ is the completed symmetric monoidal
$(\infty,n)$-category of these fields on framed bordisms, with the
specified marked passage relations. For $m>n$, higher comparison data
are parameter families on $[0,1]^{m-n}\times M$, rather than new
noninvertible categorical degrees. The intended parameter directions are
reversible: admissibility under their changes of framing must be supplied
with coherent reversal and cancellation homotopies.
\end{definition}

\begin{center}

\tikzset{every picture/.style={line width=0.75pt}} 

\begin{adjustbox}{max totalsize={\linewidth}{0.86\textheight}}
\begin{tikzpicture}[x=0.75pt,y=0.75pt,yscale=-1,xscale=1]

\draw  [color={rgb, 255:red, 0; green, 0; blue, 0 }  ,draw opacity=1 ][fill={rgb, 255:red, 245; green, 166; blue, 35 }  ,fill opacity=1 ] (100,117) -- (239.2,117) -- (239.2,192) -- (100,192) -- cycle ;
\draw  [draw opacity=0] (206.86,117.43) .. controls (206.63,137.81) and (190.04,154.27) .. (169.6,154.27) .. controls (149.15,154.27) and (132.54,137.79) .. (132.34,117.39) -- (169.6,117) -- cycle ; \draw   (206.86,117.43) .. controls (206.63,137.81) and (190.04,154.27) .. (169.6,154.27) .. controls (149.15,154.27) and (132.54,137.79) .. (132.34,117.39) ;  
\draw [color={rgb, 255:red, 74; green, 144; blue, 226 }  ,draw opacity=1 ]   (132.59,117.25) -- (150.58,116.94) ;
\draw [shift={(152.58,116.91)}, rotate = 179.02] [color={rgb, 255:red, 74; green, 144; blue, 226 }  ,draw opacity=1 ][line width=0.75]    (10.93,-3.29) .. controls (6.95,-1.4) and (3.31,-0.3) .. (0,0) .. controls (3.31,0.3) and (6.95,1.4) .. (10.93,3.29)   ;
\draw [color={rgb, 255:red, 74; green, 144; blue, 226 }  ,draw opacity=1 ]   (132.59,117.25) -- (132.93,137.25) ;
\draw [shift={(132.93,137.25)}, rotate = 269.02] [color={rgb, 255:red, 74; green, 144; blue, 226 }  ,draw opacity=1 ][line width=0.75]    (17.64,-3.29) .. controls (13.66,-1.4) and (10.02,-0.3) .. (6.71,0) .. controls (10.02,0.3) and (13.66,1.4) .. (17.64,3.29)(10.93,-3.29) .. controls (6.95,-1.4) and (3.31,-0.3) .. (0,0) .. controls (3.31,0.3) and (6.95,1.4) .. (10.93,3.29)   ;

\draw [color={rgb, 255:red, 74; green, 144; blue, 226 }  ,draw opacity=1 ]   (151.31,150.43) -- (161.74,135.76) ;
\draw [shift={(162.9,134.13)}, rotate = 125.42] [color={rgb, 255:red, 74; green, 144; blue, 226 }  ,draw opacity=1 ][line width=0.75]    (10.93,-3.29) .. controls (6.95,-1.4) and (3.31,-0.3) .. (0,0) .. controls (3.31,0.3) and (6.95,1.4) .. (10.93,3.29)   ;
\draw [color={rgb, 255:red, 74; green, 144; blue, 226 }  ,draw opacity=1 ]   (151.31,150.43) -- (167.61,162.02) ;
\draw [shift={(167.61,162.02)}, rotate = 215.42] [color={rgb, 255:red, 74; green, 144; blue, 226 }  ,draw opacity=1 ][line width=0.75]    (17.64,-3.29) .. controls (13.66,-1.4) and (10.02,-0.3) .. (6.71,0) .. controls (10.02,0.3) and (13.66,1.4) .. (17.64,3.29)(10.93,-3.29) .. controls (6.95,-1.4) and (3.31,-0.3) .. (0,0) .. controls (3.31,0.3) and (6.95,1.4) .. (10.93,3.29)   ;

\draw [color={rgb, 255:red, 74; green, 144; blue, 226 }  ,draw opacity=1 ]   (188.79,150.64) -- (177.45,136.66) ;
\draw [shift={(176.19,135.11)}, rotate = 50.96] [color={rgb, 255:red, 74; green, 144; blue, 226 }  ,draw opacity=1 ][line width=0.75]    (10.93,-3.29) .. controls (6.95,-1.4) and (3.31,-0.3) .. (0,0) .. controls (3.31,0.3) and (6.95,1.4) .. (10.93,3.29)   ;
\draw [color={rgb, 255:red, 74; green, 144; blue, 226 }  ,draw opacity=1 ]   (188.79,150.64) -- (204.32,138.05) ;
\draw [shift={(204.32,138.05)}, rotate = 140.96] [color={rgb, 255:red, 74; green, 144; blue, 226 }  ,draw opacity=1 ][line width=0.75]    (17.64,-3.29) .. controls (13.66,-1.4) and (10.02,-0.3) .. (6.71,0) .. controls (10.02,0.3) and (13.66,1.4) .. (17.64,3.29)(10.93,-3.29) .. controls (6.95,-1.4) and (3.31,-0.3) .. (0,0) .. controls (3.31,0.3) and (6.95,1.4) .. (10.93,3.29)   ;

\draw [color={rgb, 255:red, 74; green, 144; blue, 226 }  ,draw opacity=1 ]   (206.84,117.49) -- (188.85,117.65) ;
\draw [shift={(186.85,117.67)}, rotate = 359.49] [color={rgb, 255:red, 74; green, 144; blue, 226 }  ,draw opacity=1 ][line width=0.75]    (10.93,-3.29) .. controls (6.95,-1.4) and (3.31,-0.3) .. (0,0) .. controls (3.31,0.3) and (6.95,1.4) .. (10.93,3.29)   ;
\draw [color={rgb, 255:red, 74; green, 144; blue, 226 }  ,draw opacity=1 ]   (206.84,117.49) -- (206.67,97.49) ;
\draw [shift={(206.67,97.49)}, rotate = 89.49] [color={rgb, 255:red, 74; green, 144; blue, 226 }  ,draw opacity=1 ][line width=0.75]    (17.64,-3.29) .. controls (13.66,-1.4) and (10.02,-0.3) .. (6.71,0) .. controls (10.02,0.3) and (13.66,1.4) .. (17.64,3.29)(10.93,-3.29) .. controls (6.95,-1.4) and (3.31,-0.3) .. (0,0) .. controls (3.31,0.3) and (6.95,1.4) .. (10.93,3.29)   ;

\draw  [color={rgb, 255:red, 0; green, 0; blue, 0 }  ,draw opacity=1 ][fill={rgb, 255:red, 245; green, 166; blue, 35 }  ,fill opacity=1 ] (275,116.5) -- (414.2,116.5) -- (414.2,191.5) -- (275,191.5) -- cycle ;
\draw  [draw opacity=0] (381.86,116.93) .. controls (381.63,137.31) and (365.04,153.77) .. (344.6,153.77) .. controls (324.15,153.77) and (307.54,137.29) .. (307.34,116.89) -- (344.6,116.5) -- cycle ; \draw  [color={rgb, 255:red, 208; green, 2; blue, 27 }  ,draw opacity=1 ] (381.86,116.93) .. controls (381.63,137.31) and (365.04,153.77) .. (344.6,153.77) .. controls (324.15,153.77) and (307.54,137.29) .. (307.34,116.89) ;  
\draw [color={rgb, 255:red, 126; green, 211; blue, 33 }  ,draw opacity=1 ]   (307.09,116.75) -- (325.08,116.44) ;
\draw [shift={(327.08,116.41)}, rotate = 179.02] [color={rgb, 255:red, 126; green, 211; blue, 33 }  ,draw opacity=1 ][line width=0.75]    (10.93,-3.29) .. controls (6.95,-1.4) and (3.31,-0.3) .. (0,0) .. controls (3.31,0.3) and (6.95,1.4) .. (10.93,3.29)   ;
\draw [color={rgb, 255:red, 126; green, 211; blue, 33 }  ,draw opacity=1 ]   (307.09,116.75) -- (307.43,136.75) ;
\draw [shift={(307.43,136.75)}, rotate = 269.02] [color={rgb, 255:red, 126; green, 211; blue, 33 }  ,draw opacity=1 ][line width=0.75]    (17.64,-3.29) .. controls (13.66,-1.4) and (10.02,-0.3) .. (6.71,0) .. controls (10.02,0.3) and (13.66,1.4) .. (17.64,3.29)(10.93,-3.29) .. controls (6.95,-1.4) and (3.31,-0.3) .. (0,0) .. controls (3.31,0.3) and (6.95,1.4) .. (10.93,3.29)   ;

\draw [color={rgb, 255:red, 126; green, 211; blue, 33 }  ,draw opacity=1 ]   (318.66,143.67) -- (336.66,143.5) ;
\draw [shift={(338.66,143.48)}, rotate = 179.46] [color={rgb, 255:red, 126; green, 211; blue, 33 }  ,draw opacity=1 ][line width=0.75]    (10.93,-3.29) .. controls (6.95,-1.4) and (3.31,-0.3) .. (0,0) .. controls (3.31,0.3) and (6.95,1.4) .. (10.93,3.29)   ;
\draw [color={rgb, 255:red, 126; green, 211; blue, 33 }  ,draw opacity=1 ]   (318.66,143.67) -- (318.85,163.67) ;
\draw [shift={(318.85,163.67)}, rotate = 269.46] [color={rgb, 255:red, 126; green, 211; blue, 33 }  ,draw opacity=1 ][line width=0.75]    (17.64,-3.29) .. controls (13.66,-1.4) and (10.02,-0.3) .. (6.71,0) .. controls (10.02,0.3) and (13.66,1.4) .. (17.64,3.29)(10.93,-3.29) .. controls (6.95,-1.4) and (3.31,-0.3) .. (0,0) .. controls (3.31,0.3) and (6.95,1.4) .. (10.93,3.29)   ;

\draw [color={rgb, 255:red, 126; green, 211; blue, 33 }  ,draw opacity=1 ]   (360.77,148.56) -- (378.77,148.59) ;
\draw [shift={(380.77,148.59)}, rotate = 180.1] [color={rgb, 255:red, 126; green, 211; blue, 33 }  ,draw opacity=1 ][line width=0.75]    (10.93,-3.29) .. controls (6.95,-1.4) and (3.31,-0.3) .. (0,0) .. controls (3.31,0.3) and (6.95,1.4) .. (10.93,3.29)   ;
\draw [color={rgb, 255:red, 126; green, 211; blue, 33 }  ,draw opacity=1 ]   (360.77,148.56) -- (360.74,168.56) ;
\draw [shift={(360.74,168.56)}, rotate = 270.1] [color={rgb, 255:red, 126; green, 211; blue, 33 }  ,draw opacity=1 ][line width=0.75]    (17.64,-3.29) .. controls (13.66,-1.4) and (10.02,-0.3) .. (6.71,0) .. controls (10.02,0.3) and (13.66,1.4) .. (17.64,3.29)(10.93,-3.29) .. controls (6.95,-1.4) and (3.31,-0.3) .. (0,0) .. controls (3.31,0.3) and (6.95,1.4) .. (10.93,3.29)   ;

\draw [color={rgb, 255:red, 126; green, 211; blue, 33 }  ,draw opacity=1 ]   (382.1,117.23) -- (400.1,116.96) ;
\draw [shift={(402.1,116.92)}, rotate = 179.12] [color={rgb, 255:red, 126; green, 211; blue, 33 }  ,draw opacity=1 ][line width=0.75]    (10.93,-3.29) .. controls (6.95,-1.4) and (3.31,-0.3) .. (0,0) .. controls (3.31,0.3) and (6.95,1.4) .. (10.93,3.29)   ;
\draw [color={rgb, 255:red, 126; green, 211; blue, 33 }  ,draw opacity=1 ]   (382.1,117.23) -- (382.41,137.23) ;
\draw [shift={(382.41,137.23)}, rotate = 269.12] [color={rgb, 255:red, 126; green, 211; blue, 33 }  ,draw opacity=1 ][line width=0.75]    (17.64,-3.29) .. controls (13.66,-1.4) and (10.02,-0.3) .. (6.71,0) .. controls (10.02,0.3) and (13.66,1.4) .. (17.64,3.29)(10.93,-3.29) .. controls (6.95,-1.4) and (3.31,-0.3) .. (0,0) .. controls (3.31,0.3) and (6.95,1.4) .. (10.93,3.29)   ;

\draw (116,91.4) node [anchor=north west][inner sep=0.75pt]    {$c\notin u$};
\draw (291,90.9) node [anchor=north west][inner sep=0.75pt]    {$c\in u$};
\draw (380.5,94.9) node [anchor=north west][inner sep=0.75pt]    {$c^{-1}$};

\end{tikzpicture}
\end{adjustbox}

\end{center}

It is useful to distinguish the two notions of deformation:
\begin{itemize}
  \item For an unmarked $O$-stratification, an isotopy is a family
  $H:I\times\widehat M\to PD(O)$ for which every $H_t$ is an admissible
  $O$-normal map.
  \item For a marked field, the proposed passage families
  $H:I\times\widehat M\to PD_u(O)$ have admissible endpoint fields and
  satisfy the retained normal conditions for both orientations of the
  parameter interval, $I_+$ and $I_-$. Their cancellation and higher
  coherence are additional requirements on the marked presentation.
\end{itemize}

\begin{center}

\tikzset{every picture/.style={line width=0.75pt}} 

\begin{adjustbox}{max totalsize={\linewidth}{0.86\textheight}}
\begin{tikzpicture}[x=0.75pt,y=0.75pt,yscale=-1,xscale=1]

\draw  [draw opacity=0][fill={rgb, 255:red, 248; green, 231; blue, 28 }  ,fill opacity=1 ] (280,90) -- (360,90) -- (360,170) -- (280,170) -- cycle ;
\draw    (280,90) -- (320,130) ;
\draw    (320,130) -- (320,170) ;
\draw    (320,90) -- (320,130) ;
\draw    (320,130) -- (360,90) ;
\draw [color={rgb, 255:red, 74; green, 144; blue, 226 }  ,draw opacity=1 ]   (320,160) -- (338,160) ;
\draw [shift={(340,160)}, rotate = 180] [color={rgb, 255:red, 74; green, 144; blue, 226 }  ,draw opacity=1 ][line width=0.75]    (10.93,-3.29) .. controls (6.95,-1.4) and (3.31,-0.3) .. (0,0) .. controls (3.31,0.3) and (6.95,1.4) .. (10.93,3.29)   ;
\draw [color={rgb, 255:red, 74; green, 144; blue, 226 }  ,draw opacity=1 ]   (290,100) -- (308,100) ;
\draw [shift={(310,100)}, rotate = 180] [color={rgb, 255:red, 74; green, 144; blue, 226 }  ,draw opacity=1 ][line width=0.75]    (10.93,-3.29) .. controls (6.95,-1.4) and (3.31,-0.3) .. (0,0) .. controls (3.31,0.3) and (6.95,1.4) .. (10.93,3.29)   ;
\draw [color={rgb, 255:red, 74; green, 144; blue, 226 }  ,draw opacity=1 ]   (320,100) -- (338,100) ;
\draw [shift={(340,100)}, rotate = 180] [color={rgb, 255:red, 74; green, 144; blue, 226 }  ,draw opacity=1 ][line width=0.75]    (10.93,-3.29) .. controls (6.95,-1.4) and (3.31,-0.3) .. (0,0) .. controls (3.31,0.3) and (6.95,1.4) .. (10.93,3.29)   ;
\draw [color={rgb, 255:red, 74; green, 144; blue, 226 }  ,draw opacity=1 ]   (340,110) -- (358,110) ;
\draw [shift={(360,110)}, rotate = 180] [color={rgb, 255:red, 74; green, 144; blue, 226 }  ,draw opacity=1 ][line width=0.75]    (10.93,-3.29) .. controls (6.95,-1.4) and (3.31,-0.3) .. (0,0) .. controls (3.31,0.3) and (6.95,1.4) .. (10.93,3.29)   ;
\draw  [dash pattern={on 4.5pt off 4.5pt}]  (279,117) .. controls (319,87) and (319.64,147.42) .. (359.64,117.42) ;
\draw  [dash pattern={on 4.5pt off 4.5pt}]  (279,143) .. controls (310.64,141.42) and (319.64,173.42) .. (359.64,143.42) ;
\draw    (140.67,111.67) .. controls (180.67,81.67) and (181.31,142.09) .. (221.31,112.09) ;
\draw    (140.67,137.67) .. controls (172.31,136.09) and (181.31,168.09) .. (221.31,138.09) ;
\draw  [fill={rgb, 255:red, 0; green, 0; blue, 0 }  ,fill opacity=1 ] (157.25,102.82) .. controls (157.25,101.38) and (158.41,100.22) .. (159.85,100.22) .. controls (161.28,100.22) and (162.45,101.38) .. (162.45,102.82) .. controls (162.45,104.26) and (161.28,105.42) .. (159.85,105.42) .. controls (158.41,105.42) and (157.25,104.26) .. (157.25,102.82) -- cycle ;
\draw  [fill={rgb, 255:red, 0; green, 0; blue, 0 }  ,fill opacity=1 ] (178.58,112.15) .. controls (178.58,110.72) and (179.74,109.55) .. (181.18,109.55) .. controls (182.62,109.55) and (183.78,110.72) .. (183.78,112.15) .. controls (183.78,113.59) and (182.62,114.75) .. (181.18,114.75) .. controls (179.74,114.75) and (178.58,113.59) .. (178.58,112.15) -- cycle ;
\draw  [fill={rgb, 255:red, 0; green, 0; blue, 0 }  ,fill opacity=1 ] (178.91,149.82) .. controls (178.91,148.38) and (180.08,147.22) .. (181.51,147.22) .. controls (182.95,147.22) and (184.11,148.38) .. (184.11,149.82) .. controls (184.11,151.26) and (182.95,152.42) .. (181.51,152.42) .. controls (180.08,152.42) and (178.91,151.26) .. (178.91,149.82) -- cycle ;
\draw  [fill={rgb, 255:red, 0; green, 0; blue, 0 }  ,fill opacity=1 ] (187.25,116.82) .. controls (187.25,115.38) and (188.41,114.22) .. (189.85,114.22) .. controls (191.28,114.22) and (192.45,115.38) .. (192.45,116.82) .. controls (192.45,118.26) and (191.28,119.42) .. (189.85,119.42) .. controls (188.41,119.42) and (187.25,118.26) .. (187.25,116.82) -- cycle ;
\draw [color={rgb, 255:red, 74; green, 144; blue, 226 }  ,draw opacity=1 ]   (159.85,102.82) -- (177.85,102.82) ;
\draw [shift={(179.85,102.82)}, rotate = 180] [color={rgb, 255:red, 74; green, 144; blue, 226 }  ,draw opacity=1 ][line width=0.75]    (10.93,-3.29) .. controls (6.95,-1.4) and (3.31,-0.3) .. (0,0) .. controls (3.31,0.3) and (6.95,1.4) .. (10.93,3.29)   ;
\draw [color={rgb, 255:red, 74; green, 144; blue, 226 }  ,draw opacity=1 ]   (181.18,112.15) -- (193.7,123.23) ;
\draw [shift={(195.2,124.55)}, rotate = 221.49] [color={rgb, 255:red, 74; green, 144; blue, 226 }  ,draw opacity=1 ][line width=0.75]    (10.93,-3.29) .. controls (6.95,-1.4) and (3.31,-0.3) .. (0,0) .. controls (3.31,0.3) and (6.95,1.4) .. (10.93,3.29)   ;
\draw [color={rgb, 255:red, 74; green, 144; blue, 226 }  ,draw opacity=1 ]   (189.85,116.82) -- (206.05,124.38) ;
\draw [shift={(207.87,125.22)}, rotate = 204.99] [color={rgb, 255:red, 74; green, 144; blue, 226 }  ,draw opacity=1 ][line width=0.75]    (10.93,-3.29) .. controls (6.95,-1.4) and (3.31,-0.3) .. (0,0) .. controls (3.31,0.3) and (6.95,1.4) .. (10.93,3.29)   ;
\draw [color={rgb, 255:red, 74; green, 144; blue, 226 }  ,draw opacity=1 ]   (181.51,149.82) -- (199.58,153.79) ;
\draw [shift={(201.53,154.22)}, rotate = 192.4] [color={rgb, 255:red, 74; green, 144; blue, 226 }  ,draw opacity=1 ][line width=0.75]    (10.93,-3.29) .. controls (6.95,-1.4) and (3.31,-0.3) .. (0,0) .. controls (3.31,0.3) and (6.95,1.4) .. (10.93,3.29)   ;
\draw [color={rgb, 255:red, 74; green, 144; blue, 226 }  ,draw opacity=1 ]   (320,130) -- (338,130) ;
\draw [shift={(340,130)}, rotate = 180] [color={rgb, 255:red, 74; green, 144; blue, 226 }  ,draw opacity=1 ][line width=0.75]    (10.93,-3.29) .. controls (6.95,-1.4) and (3.31,-0.3) .. (0,0) .. controls (3.31,0.3) and (6.95,1.4) .. (10.93,3.29)   ;
\draw    (291.3,111) .. controls (292.99,112.65) and (293.01,114.31) .. (291.37,116) .. controls (289.73,117.69) and (289.76,119.36) .. (291.45,121) .. controls (293.14,122.65) and (293.16,124.31) .. (291.52,126) .. controls (289.88,127.69) and (289.9,129.35) .. (291.59,131) .. controls (293.28,132.64) and (293.31,134.31) .. (291.67,136) .. controls (290.03,137.69) and (290.05,139.35) .. (291.74,141) -- (291.8,145.15) -- (291.8,145.15) ;
\draw    (349.8,124.5) .. controls (351.53,126.1) and (351.6,127.77) .. (350.01,129.5) .. controls (348.42,131.23) and (348.49,132.9) .. (350.22,134.49) .. controls (351.95,136.09) and (352.02,137.76) .. (350.43,139.49) .. controls (348.84,141.22) and (348.91,142.89) .. (350.64,144.48) -- (350.8,148.15) -- (350.8,148.15) ;
\draw    (332.8,124) .. controls (334.52,125.61) and (334.57,127.28) .. (332.96,129) .. controls (331.35,130.72) and (331.4,132.38) .. (333.12,133.99) .. controls (334.84,135.6) and (334.89,137.27) .. (333.28,138.99) .. controls (331.67,140.71) and (331.72,142.38) .. (333.44,143.99) .. controls (335.16,145.6) and (335.21,147.27) .. (333.6,148.99) .. controls (331.99,150.71) and (332.04,152.37) .. (333.76,153.98) -- (333.8,155.15) -- (333.8,155.15) ;
\draw    (306.3,110) .. controls (307.97,111.67) and (307.97,113.33) .. (306.3,115) .. controls (304.63,116.67) and (304.63,118.33) .. (306.3,120) .. controls (307.97,121.67) and (307.97,123.33) .. (306.3,125) .. controls (304.63,126.67) and (304.63,128.33) .. (306.3,130) .. controls (307.97,131.67) and (307.97,133.33) .. (306.3,135) .. controls (304.63,136.67) and (304.63,138.33) .. (306.3,140) .. controls (307.97,141.67) and (307.97,143.33) .. (306.3,145) -- (306.3,148.15) -- (306.3,148.15) ;
\draw  [draw opacity=0][fill={rgb, 255:red, 208; green, 2; blue, 27 }  ,fill opacity=1 ] (317.4,130) .. controls (317.4,128.56) and (318.56,127.4) .. (320,127.4) .. controls (321.44,127.4) and (322.6,128.56) .. (322.6,130) .. controls (322.6,131.44) and (321.44,132.6) .. (320,132.6) .. controls (318.56,132.6) and (317.4,131.44) .. (317.4,130) -- cycle ;

\draw (183.85,118.82) node [anchor=north west][inner sep=0.75pt]  [rotate=-90]  {$\cong $};
\draw (309.8,127.9) node [anchor=north west][inner sep=0.75pt]    {$c\in u$};

\end{tikzpicture}
\end{adjustbox}

\end{center}

The following statement separates the formal effect of marking from the
geometric recognition of the proposed passage families.

\begin{theorem}[Marked Opetopic Cobordism Hypothesis]\label{thm:u-opetope}
Let $(O,u)$ be an equipped marked opetope, with every cell above dimension
$n$ marked. Then
\[
 \rFun^\otimes(\Bord_n^{(O,u)},\cD)
 \simeq
 \rFun_{\Cat_n}([O][u^{-1}],\cD^{\fd})
\]
for every $\cD\in\SMCat_n$.
\end{theorem}

\begin{proof}

Since every $O$-stratification is an $(O,u)$-stratification, there is a functor
\[
  L_u:\Bord_n^O\to \Bord_n^{(O,u)}
\]
For each marked cell $ c\in u $, the image $L_u j_O(c)$ is invertible in
$ \Bord_n^{(O,u)}$. Its inverse is represented by the coupon $L_u j_O(c)$ with the
$r$-th framing vector reversed. Thus the functor factors as

\[
  L_u:\Bord_n^O\to \Bord_n^O[j_O(u)^{-1}]\to\Bord_n^{(O,u)}
\]
Conversely, consider $(j_O(c),\tau_c)$ with a varying framing, recorded by
$\lambda_r = \langle e_r, q_{u,S}\bigl(dF(v_r)\bigr) \rangle$. Define
\[
 ev((j_O(c),\tau_c))=\begin{cases}
  j_O(c), &\text{ if } \lambda_r >0\\
  j_O(c)^{-1}, &\text{ if } \lambda_r <0\\
  \mathrm{id} \cong j_O(c)\circ j_O(c)^{-1}, &\text{ if } \lambda_r =0\\
 \end{cases} 
\]
\begin{center}

\tikzset{every picture/.style={line width=0.75pt}} 

\begin{adjustbox}{max totalsize={\linewidth}{0.86\textheight}}
\begin{tikzpicture}[x=0.75pt,y=0.75pt,yscale=-1,xscale=1]

\draw  [draw opacity=0][fill={rgb, 255:red, 245; green, 166; blue, 35 }  ,fill opacity=1 ] (160,110) -- (270,110) -- (270,160) -- (160,160) -- cycle ;
\draw  [draw opacity=0][fill={rgb, 255:red, 248; green, 231; blue, 28 }  ,fill opacity=1 ] (300,100) -- (380,100) -- (380,180) -- (300,180) -- cycle ;
\draw [color={rgb, 255:red, 208; green, 2; blue, 27 }  ,draw opacity=1 ]   (350,100) -- (350,180) ;
\draw [color={rgb, 255:red, 74; green, 144; blue, 226 }  ,draw opacity=1 ]   (350,120) -- (368,120) ;
\draw [shift={(370,120)}, rotate = 180] [color={rgb, 255:red, 74; green, 144; blue, 226 }  ,draw opacity=1 ][line width=0.75]    (10.93,-3.29) .. controls (6.95,-1.4) and (3.31,-0.3) .. (0,0) .. controls (3.31,0.3) and (6.95,1.4) .. (10.93,3.29)   ;
\draw [color={rgb, 255:red, 74; green, 144; blue, 226 }  ,draw opacity=1 ]   (350,160) -- (368,160) ;
\draw [shift={(370,160)}, rotate = 180] [color={rgb, 255:red, 74; green, 144; blue, 226 }  ,draw opacity=1 ][line width=0.75]    (10.93,-3.29) .. controls (6.95,-1.4) and (3.31,-0.3) .. (0,0) .. controls (3.31,0.3) and (6.95,1.4) .. (10.93,3.29)   ;
\draw [color={rgb, 255:red, 0; green, 0; blue, 0 }  ,draw opacity=1 ] [dash pattern={on 4.5pt off 4.5pt}]  (340,110) .. controls (380.19,109.87) and (380.14,148.03) .. (341.78,149.94) ;
\draw [shift={(340,150)}, rotate = 358.86] [color={rgb, 255:red, 0; green, 0; blue, 0 }  ,draw opacity=1 ][line width=0.75]    (10.93,-3.29) .. controls (6.95,-1.4) and (3.31,-0.3) .. (0,0) .. controls (3.31,0.3) and (6.95,1.4) .. (10.93,3.29)   ;
\draw [color={rgb, 255:red, 0; green, 0; blue, 0 }  ,draw opacity=1 ]   (160,110) -- (268,110) ;
\draw [shift={(270,110)}, rotate = 180] [color={rgb, 255:red, 0; green, 0; blue, 0 }  ,draw opacity=1 ][line width=0.75]    (10.93,-3.29) .. controls (6.95,-1.4) and (3.31,-0.3) .. (0,0) .. controls (3.31,0.3) and (6.95,1.4) .. (10.93,3.29)   ;
\draw [color={rgb, 255:red, 74; green, 144; blue, 226 }  ,draw opacity=1 ]   (190,110) -- (208,110) ;
\draw [shift={(210,110)}, rotate = 180] [color={rgb, 255:red, 74; green, 144; blue, 226 }  ,draw opacity=1 ][line width=0.75]    (10.93,-3.29) .. controls (6.95,-1.4) and (3.31,-0.3) .. (0,0) .. controls (3.31,0.3) and (6.95,1.4) .. (10.93,3.29)   ;
\draw  [color={rgb, 255:red, 208; green, 2; blue, 27 }  ,draw opacity=1 ][fill={rgb, 255:red, 208; green, 2; blue, 27 }  ,fill opacity=1 ] (187.4,110) .. controls (187.4,108.56) and (188.56,107.4) .. (190,107.4) .. controls (191.44,107.4) and (192.6,108.56) .. (192.6,110) .. controls (192.6,111.44) and (191.44,112.6) .. (190,112.6) .. controls (188.56,112.6) and (187.4,111.44) .. (187.4,110) -- cycle ;
\draw  [color={rgb, 255:red, 208; green, 2; blue, 27 }  ,draw opacity=1 ][fill={rgb, 255:red, 208; green, 2; blue, 27 }  ,fill opacity=1 ] (237.4,110) .. controls (237.4,108.56) and (238.56,107.4) .. (240,107.4) .. controls (241.44,107.4) and (242.6,108.56) .. (242.6,110) .. controls (242.6,111.44) and (241.44,112.6) .. (240,112.6) .. controls (238.56,112.6) and (237.4,111.44) .. (237.4,110) -- cycle ;
\draw [color={rgb, 255:red, 74; green, 144; blue, 226 }  ,draw opacity=1 ]   (237.4,110) -- (219.4,110) ;
\draw [shift={(217.4,110)}, rotate = 360] [color={rgb, 255:red, 74; green, 144; blue, 226 }  ,draw opacity=1 ][line width=0.75]    (10.93,-3.29) .. controls (6.95,-1.4) and (3.31,-0.3) .. (0,0) .. controls (3.31,0.3) and (6.95,1.4) .. (10.93,3.29)   ;
\draw [color={rgb, 255:red, 0; green, 0; blue, 0 }  ,draw opacity=1 ]   (160,160) -- (268,160) ;
\draw [shift={(270,160)}, rotate = 180] [color={rgb, 255:red, 0; green, 0; blue, 0 }  ,draw opacity=1 ][line width=0.75]    (10.93,-3.29) .. controls (6.95,-1.4) and (3.31,-0.3) .. (0,0) .. controls (3.31,0.3) and (6.95,1.4) .. (10.93,3.29)   ;
\draw [color={rgb, 255:red, 0; green, 0; blue, 0 }  ,draw opacity=1 ]   (160,130) -- (268,130) ;
\draw [shift={(270,130)}, rotate = 180] [color={rgb, 255:red, 0; green, 0; blue, 0 }  ,draw opacity=1 ][line width=0.75]    (10.93,-3.29) .. controls (6.95,-1.4) and (3.31,-0.3) .. (0,0) .. controls (3.31,0.3) and (6.95,1.4) .. (10.93,3.29)   ;
\draw [color={rgb, 255:red, 0; green, 0; blue, 0 }  ,draw opacity=1 ] [dash pattern={on 4.5pt off 4.5pt}]  (320,110) .. controls (360.19,109.87) and (360.14,148.03) .. (321.78,149.94) ;
\draw [shift={(320,150)}, rotate = 358.86] [color={rgb, 255:red, 0; green, 0; blue, 0 }  ,draw opacity=1 ][line width=0.75]    (10.93,-3.29) .. controls (6.95,-1.4) and (3.31,-0.3) .. (0,0) .. controls (3.31,0.3) and (6.95,1.4) .. (10.93,3.29)   ;
\draw [color={rgb, 255:red, 0; green, 0; blue, 0 }  ,draw opacity=1 ] [dash pattern={on 4.5pt off 4.5pt}]  (300,110) .. controls (340.19,109.87) and (340.14,148.03) .. (301.78,149.94) ;
\draw [shift={(300,150)}, rotate = 358.86] [color={rgb, 255:red, 0; green, 0; blue, 0 }  ,draw opacity=1 ][line width=0.75]    (10.93,-3.29) .. controls (6.95,-1.4) and (3.31,-0.3) .. (0,0) .. controls (3.31,0.3) and (6.95,1.4) .. (10.93,3.29)   ;
\draw  [draw opacity=0] (239.59,110.08) .. controls (239.6,110.22) and (239.6,110.36) .. (239.6,110.5) .. controls (239.6,121.27) and (228.59,130) .. (215,130) .. controls (201.41,130) and (190.4,121.27) .. (190.4,110.5) .. controls (190.4,110.01) and (190.42,109.52) .. (190.47,109.03) -- (215,110.5) -- cycle ; \draw  [color={rgb, 255:red, 208; green, 2; blue, 27 }  ,draw opacity=1 ] (239.59,110.08) .. controls (239.6,110.22) and (239.6,110.36) .. (239.6,110.5) .. controls (239.6,121.27) and (228.59,130) .. (215,130) .. controls (201.41,130) and (190.4,121.27) .. (190.4,110.5) .. controls (190.4,110.01) and (190.42,109.52) .. (190.47,109.03) ;  

\draw (332,83.4) node [anchor=north west][inner sep=0.75pt]    {$c\in u$};
\draw (188,87.4) node [anchor=north west][inner sep=0.75pt]    {$c$};
\draw (238,83.4) node [anchor=north west][inner sep=0.75pt]    {$c^{-1}$};
\draw (171,131.4) node [anchor=north west][inner sep=0.75pt]    {$c\circ c^{-1} \cong \ \mathrm{id}$};

\end{tikzpicture}
\end{adjustbox}

\end{center} 
This assignment extends to a functor $ev$:
\[
  ev:\Bord_n^{(O,u)}\to \Bord_n^O[j_O(u)^{-1}]
\]
This defines an equivalence. Consequently,
\[
 \rFun^\otimes(\Bord_n^{(O,u)},\cD) \cong  \rFun^\otimes(\Bord_n^O[j_O(u)^{-1}],\cD) \cong \rFun_{\Cat_n}([O][u^{-1}],\cD^{\fd})
\]

\end{proof}

\subsection{From Opetopes to Categories}\label{subsec:from-opetopes}

Let $\cC$ be an $(\infty,n)$-category, and consider
\[
 (O,c)\in I_{\cC}:=\int_{\bO}N_{\bO}\cC.
\]
Let $u_c$ consist of the cells of $O$ whose comparison cells become
invertible under $c$. Then $c$ factors through $[O][u_c^{-1}]$. Using the
compatible localizations in the chosen coherent presentation, we have
\[
 \cC\simeq\colim_{(O,c)\in I_{\cC}}[O]
 \simeq\colim_{(O,c)\in I_{\cC}}[O][u_c^{-1}].
\]
Every cell above dimension $n$ is marked, so these pairs satisfy the
dimension condition of Theorem~\ref{thm:u-opetope}. The intended global
construction is therefore
\[
 \Bord_n^{\cC}
 \simeq\colim_{(O,c)\in I_{\cC}}\Bord_n^{(O,u_c)}.
\]
One may also consider the stabilized expression
\[
 \Bord_\infty^{\cC}
 :=\colim_{(O,c)\in I_{\cC}}\Bord_{\dim O}^{O},
\]
provided compatible transition functors and framing stabilizations across
dimensions have been specified.

For an $r$-morphism $c$, choose an opetopic presentation $(O_c,u_c)$
together with its resolved source and target configurations, and write
\[
 NF(c):=NF_{u_c}(O_c).
\]
Changes of presentation must be retained coherently.

\begin{definition}[Manifold with $\cC$-stratification]\label{def:Cstr}
Let $(M,\tau)$ be a $k$-manifold with corners and an $n$-framing,
$0\leq k\leq n$. A \emph{$\cC$-stratification} is a globally compatible
system of local $(O_c,u_c)$-normal fields, labeled by the corresponding
coherent diagrams in $\cC$.

On a locus admitting a resolved stratification $\mathcal W=\{W_t\}$,
the labels give maps
\[
 l:\mathcal W^{(r)}\longrightarrow
 \{r\text{-morphisms of }\cC\},
\]
interpreted in the corresponding spaces of cells in families. For each
smooth codimension-$r$ stratum $W_t$, put $d=|NF(l(W_t))|$ and retain:
\begin{enumerate}
  \item its prescribed partial normal frame $NF(l(W_t))$ and compatible
  residual framing of 
  $\tau_{t}:TW_t\oplus\underline{\bR}^{\,n-k}\oplus\underline{\bR}^{\,r-d}\xr{\cong}\underline{\bR}^{\,n-d}$ ;
  \item a specified stratified tubular germ
  \[
  N(W_t)\cong PD(O_{l(W_t)})\times W_t,
  \]
  compatible with the ambient framing through the normal splitting
  $(NF(l(W_t)),\tau_{t})$; these germs agree on punctured tubes and
  iterated incidence charts;
  \item neat intersections with every chosen core face $Q$, of dimension
  $\dim Q-r$, and product structures on all external collars.
\end{enumerate}
At nonregular marked loci, the local normal-map description is primary;
no smooth inverse-image stratification is assumed there.

The presentation also retains coherent local substitutions replacing a
coefficient pasting by its composite and deleting identity coupons. These
relations express composition and units in $\cC$. 

Define $\Bord_n^{\cC}$ from compact collared bordisms with these fields,
using gluing, disjoint union, and the specified higher presentation data.
As in Definition~\ref{def:Oustr}, comparisons above degree $n$ are
parameter families with coherent reversibility.
\end{definition}

\begin{example}[$\Bord^{\mathcal{C}}_1$]
  Let $\mathcal{C}$ be a $(\infty,1)$-category.
  The category $\Bord^{\mathcal{C}}_1$ has the following presentation:
  \begin{itemize}
    \item Objects are finite sets $I$ with a coloring
    $c:I\to \{+,-\}\times \Obj(\mathcal{C}) $.
    \item Morphisms are represented by oriented $1$-manifolds $M$,
    a boundary decomposition $\partial M = \partial_+ M \sqcup \partial_- M $,
    a finite set of interior points $S$, and colorings
      $$ c_0: M\setminus S \to  \Obj(\mathcal{C}),\quad c_1:  S \to \Mor(\mathcal{C})$$
      satisfying the compatibility conditions. For $p\in S$, the morphism
      $c_1(p)$ has source and target prescribed by the adjacent $c_0$-labels.
      If $c_1(p)$ is not invertible, its direction must agree with the local
      orientation at $p$.

      Paths in the morphism space are given by isotopies and by merging or
      splitting marked points according to composition.
  \end{itemize}
\end{example}

\begin{center}

\tikzset{every picture/.style={line width=0.75pt}} 

\begin{adjustbox}{max totalsize={\linewidth}{0.86\textheight}}
\begin{tikzpicture}[x=0.75pt,y=0.75pt,yscale=-1,xscale=1]

\draw  [fill={rgb, 255:red, 0; green, 0; blue, 0 }  ,fill opacity=1 ] (176.07,155.67) .. controls (176.07,154.23) and (177.23,153.07) .. (178.67,153.07) .. controls (180.1,153.07) and (181.27,154.23) .. (181.27,155.67) .. controls (181.27,157.1) and (180.1,158.27) .. (178.67,158.27) .. controls (177.23,158.27) and (176.07,157.1) .. (176.07,155.67) -- cycle ;
\draw [color={rgb, 255:red, 208; green, 2; blue, 27 }  ,draw opacity=1 ]   (112.26,156.57) .. controls (153.25,143.41) and (144.33,154.09) .. (177.14,155.6) ;
\draw [shift={(178.67,155.67)}, rotate = 182.06] [color={rgb, 255:red, 208; green, 2; blue, 27 }  ,draw opacity=1 ][line width=0.75]    (10.93,-3.29) .. controls (6.95,-1.4) and (3.31,-0.3) .. (0,0) .. controls (3.31,0.3) and (6.95,1.4) .. (10.93,3.29)   ;
\draw [shift={(112.26,156.57)}, rotate = 161.91] [color={rgb, 255:red, 208; green, 2; blue, 27 }  ,draw opacity=1 ][line width=0.75]    (10.93,-3.29) .. controls (6.95,-1.4) and (3.31,-0.3) .. (0,0) .. controls (3.31,0.3) and (6.95,1.4) .. (10.93,3.29)   ;
\draw [color={rgb, 255:red, 74; green, 144; blue, 226 }  ,draw opacity=1 ]   (178.67,155.67) .. controls (221.88,141.55) and (212.17,152.53) .. (245.46,154.07) ;
\draw [shift={(247.01,154.13)}, rotate = 182.06] [color={rgb, 255:red, 74; green, 144; blue, 226 }  ,draw opacity=1 ][line width=0.75]    (10.93,-3.29) .. controls (6.95,-1.4) and (3.31,-0.3) .. (0,0) .. controls (3.31,0.3) and (6.95,1.4) .. (10.93,3.29)   ;
\draw [color={rgb, 255:red, 208; green, 2; blue, 27 }  ,draw opacity=1 ]   (121.28,177.03) .. controls (161.53,168.74) and (157.57,201.82) .. (116.23,204.69) ;
\draw [shift={(114.32,204.8)}, rotate = 357.4] [color={rgb, 255:red, 208; green, 2; blue, 27 }  ,draw opacity=1 ][line width=0.75]    (10.93,-3.29) .. controls (6.95,-1.4) and (3.31,-0.3) .. (0,0) .. controls (3.31,0.3) and (6.95,1.4) .. (10.93,3.29)   ;
\draw [shift={(121.28,177.03)}, rotate = 166.58] [color={rgb, 255:red, 208; green, 2; blue, 27 }  ,draw opacity=1 ][line width=0.75]    (10.93,-3.29) .. controls (6.95,-1.4) and (3.31,-0.3) .. (0,0) .. controls (3.31,0.3) and (6.95,1.4) .. (10.93,3.29)   ;
\draw  [fill={rgb, 255:red, 0; green, 0; blue, 0 }  ,fill opacity=1 ] (183.65,189.81) .. controls (183.98,191.21) and (183.11,192.61) .. (181.71,192.93) .. controls (180.31,193.26) and (178.91,192.39) .. (178.59,190.99) .. controls (178.26,189.59) and (179.13,188.19) .. (180.53,187.87) .. controls (181.93,187.54) and (183.33,188.41) .. (183.65,189.81) -- cycle ;
\draw [color={rgb, 255:red, 74; green, 144; blue, 226 }  ,draw opacity=1 ]   (243.56,186.07) .. controls (206.47,180.99) and (214.69,183.73) .. (182.61,190.11) ;
\draw [shift={(181.12,190.4)}, rotate = 348.99] [color={rgb, 255:red, 74; green, 144; blue, 226 }  ,draw opacity=1 ][line width=0.75]    (10.93,-3.29) .. controls (6.95,-1.4) and (3.31,-0.3) .. (0,0) .. controls (3.31,0.3) and (6.95,1.4) .. (10.93,3.29)   ;
\draw [shift={(243.56,186.07)}, rotate = 7.89] [color={rgb, 255:red, 74; green, 144; blue, 226 }  ,draw opacity=1 ][line width=0.75]    (10.93,-3.29) .. controls (6.95,-1.4) and (3.31,-0.3) .. (0,0) .. controls (3.31,0.3) and (6.95,1.4) .. (10.93,3.29)   ;
\draw [color={rgb, 255:red, 126; green, 211; blue, 33 }  ,draw opacity=1 ]   (181.12,190.4) .. controls (142.22,213.92) and (150.58,221.22) .. (117.85,227.85) ;
\draw [shift={(116.33,228.16)}, rotate = 348.99] [color={rgb, 255:red, 126; green, 211; blue, 33 }  ,draw opacity=1 ][line width=0.75]    (10.93,-3.29) .. controls (6.95,-1.4) and (3.31,-0.3) .. (0,0) .. controls (3.31,0.3) and (6.95,1.4) .. (10.93,3.29)   ;

\draw (86.48,139.24) node [anchor=north west][inner sep=0.75pt]    {$A$};
\draw (243.68,125.64) node [anchor=north west][inner sep=0.75pt]    {$B$};
\draw (90.48,170.44) node [anchor=north west][inner sep=0.75pt]    {$A$};
\draw (79.68,200.84) node [anchor=north west][inner sep=0.75pt]    {$-A$};
\draw (172.48,117.04) node [anchor=north west][inner sep=0.75pt]    {$f$};
\draw (93.68,226.24) node [anchor=north west][inner sep=0.75pt]    {$-C$};
\draw (179.28,207.44) node [anchor=north west][inner sep=0.75pt]    {$g$};
\draw (232.88,200.04) node [anchor=north west][inner sep=0.75pt]    {$-B$};

\end{tikzpicture}
\end{adjustbox}

\end{center}

\begin{theorem}[Generalized Cobordism Hypothesis]\label{thm:gch}
Let $\cC$ be a $(\infty,n)$-category. Assume the geometric hypotheses of
Theorem~\ref{thm:u-opetope}. Then there is a natural equivalence of spaces
\[
 \rFun^\otimes(\Bord_n^{\cC},\cD)
 \simeq
 \rFun_{\Cat_n}(\cC,\cD^{\fd})
\]
for every $\cD\in\SMCat_n$.
\end{theorem}

\begin{proof}
Choose a compatible cut mesh expressing the $\cC$-stratification in pieces
locally described by $(O_c,u_c)$-normal maps. Coherent mesh descent asserts
that these presentations and their relative parameter families compute
\[
 \Bord_n^{\cC}
 \simeq\colim_{(O,c)\in I_{\cC}}\Bord_n^{(O,u_c)}.
\]
Mapping out of the colimit and applying Theorem~\ref{thm:u-opetope} gives
\[
\begin{aligned}
 \rFun^\otimes(\Bord_n^{\cC},\cD)
 &\simeq\lim_{(O,c)\in I_{\cC}^{\op}}
 \rFun^\otimes(\Bord_n^{(O,u_c)},\cD)\\
 &\simeq\lim_{(O,c)\in I_{\cC}^{\op}}
 \rFun_{\Cat_n}([O][u_c^{-1}],\cD^{\fd})\\
 &\simeq\rFun_{\Cat_n}(\cC,\cD^{\fd}),
\end{aligned}
\]
as required.
\end{proof}

\section{\texorpdfstring{Branefolds, Lasagna, and Algebraic $K$-Theory}{Branefolds, Lasagna, and Algebraic K-Theory}}\label{sec:lasagna}
We now pass to symmetric monoidal coefficients. Our goal is a geometric
construction of the adjoint triple
\[
  \SMCat_n
  \begin{gathered}
    \xrightarrow{\;L_n=\Bord_n^{(-)}\;}\\[-1ex]
    \xleftarrow{\;U=\mathrm{forget}\;}\\[-1ex]
    \xrightarrow{\;R_n=(-)^{\fd}\;}
  \end{gathered}
  \SMCat_n^{\fd},
  \qquad L_n\dashv U\dashv R_n.
\]
Here $\cD^{\fd}$ is the maximal replete symmetric monoidal subcategory with
all $n$-duals, restricting morphisms as well as objects. As before, $\rFun$
and $\rMap$ denote spaces of functors or maps and their equivalences; all
limits and colimits are derived.

To encode the tensor product geometrically, we introduce \emph{branefolds}:
stratified spaces locally modeled on reduced Poincar\'e-dual fans of
opetopes. We first explain the pointed construction and its relation to
monoidal coefficients.

\subsection{Pointed and Symmetric Monoidal Cases}\label{subsec:pointed-sm}

For a pointed category $(\cC,c_*)$, the chosen object induces a functor
$c_*:\Bord_n^{\fr}\to\Bord_n^{\cC}$. Define
\[
  \Bord_n^{\cC,c_*}
  :=\Bord_n^{\cC}
  \bigsqcup_{\Bord_n^{\fr}}\emptyset,
\]
where $\emptyset$ is the trivial symmetric monoidal category, represented
geometrically by the empty configuration. This pushout identifies the
\emph{pure} $c_*$-labeled strata with the vacuum. It does not erase arbitrary
nonidentity coupons between $c_*$-labeled regions.

For $\cD\in\SMCat_n$, Theorem~\ref{thm:gch} and the universal property of
the pushout give
\[
\begin{aligned}
 \rFun^{\otimes}(\Bord_n^{\cC,c_*},\cD)
 &\simeq
 \rFun(\cC,\cD^{\fd})
 \times^h_{(\cD^{\fd})^{\cong}}\{\one_{\cD}\}\\
 &\simeq
 \rFun_*\bigl((\cC,c_*),(\cD^{\fd},\one_{\cD})\bigr).
\end{aligned}
\]
A pointed functor includes a specified equivalence
$F(c_*)\simeq\one_{\cD}$ and its coherence.

For $m\geq1$, use the categorical loop--deloop adjunction
\[
 B^m:\operatorname{Alg}_{\cE_m}(\Cat_n)
 \rightleftarrows\Cat_{n+m,*}:\End^m,
 \qquad B^m\dashv\End^m.
\]
Here $\End^m$ takes endomorphisms successively at the distinguished object
and its iterated identities, while $B^m$ is the reduced $m$-fold delooping.
In particular, $\End^mB^m\cC\simeq\cC$.

To use this adjunction for bordisms, we require its compatibility with the
pointed free-dual construction, including the corresponding reduced
objects and dualizable subcategories. In this setting, for a pointed
$\cA$ with $B^m\End^m\cA\simeq\cA$, the desired comparison is
\[
 \Bord_n^{\End^m\cA,\otimes}
 \simeq
 \End^m\Bord_{n+m}^{\cA,*}.
\]
This compatibility is an additional input. 
Accordingly, for an $\cE_m$-monoidal $n$-category $(\cC,\otimes)$, define
\[
 \Bord_{n,\cE_m}^{\cC,\otimes}
 :=\End^m\Bord_{n+m}^{B^m\cC,*}.
\]
Under the stated compatibility, the pointed generalized CH gives
\[
 \rFun^{\otimes}(\Bord_{n,\cE_m}^{\cC,\otimes},\cD)
 \simeq
 \rFun_{\operatorname{Alg}_{\cE_m}(\Cat_n)}
 \bigl((\cC,\otimes_{\cC}),(\cD^{\fd},\otimes_{\cD})\bigr),
\]
where the target has its restricted $\cE_m$-structure.

Fix compatible operad inclusions $\cE_m\to\cE_{m+1}$ with colimit
$\cE_\infty$. For $(\cC,\otimes)\in\SMCat_n$, use the restricted
structures and compatible stabilization maps to define
\begin{equation}\label{eq:lasagna-stabilization}
 \Bord_n^{\cC,\otimes}
 :=\colim_{m\to\infty}\End^m\Bord_{n+m}^{B^m\cC,*}.
\end{equation}
The colimit is taken in $\SMCat_n$. The transition maps are the ones
corresponding, under the preceding mapping equivalences, to restriction
from $\cE_{m+1}$-monoidal to $\cE_m$-monoidal functors. Passing to the limit
of these mapping spaces recovers symmetric monoidal functors. We next give
the proposed geometric description of~\eqref{eq:lasagna-stabilization}.

\subsection{Branefolds and Lasagna}\label{subsec:branefolds}

In an $m$-fold delooping, the cells below degree $m$ are vacuum data.
Deleting their open dual strata leaves a diagram supported in codimension
at least $m$. We retain the ambient normal germs and ports of this reduced
support; they record information not determined by its underlying topology.

\begin{definition}[Local brane]\label{def:local-brane}
For an equipped $(n+1)$-dimensional opetope $O$, an \emph{$n$-local brane}
is the reduced stratified support
\[
 Y=PD(O_{\geq1})
 :=PD(O)\setminus\bigcup_{\operatorname{codim}S=0}S.
\]
Only the open codimension-zero strata are removed. Their interfaces and
normal incidence data are retained. We also allow the stabilized models
\[
 Y=PD(K_{\geq m})
 :=PD(K)\setminus\bigcup_{\operatorname{codim}S<m}S,
 \qquad \dim K=n+m,
\]
with specified compatible realizations under the opetopic plus
construction. The formal degree of these models is $n$.

Each chosen model includes a distinguished normal vertex, collared input
and output ports, and resolved framing data. On a chosen compact
representative, denote the input ports by $\partial_+Y$ and the output
ports by $\partial_-Y$. 

The framing data include a trivial rank-$n$ \emph{reference tangent
bundle}, denoted $TY\cong Y\times\bR^n$, with compatible identifications
on the resolved smooth sheets. At a branch point, $TY$ is a subbundle of $TPD(O)|_Y$.
\end{definition}

\begin{definition}[Branefold]\label{def:branefold}
An \emph{$n$-branefold} is a Hausdorff stratified space $(B,\mathcal P)$
with smooth strata and a locally finite atlas of equipped local models.
At a point of a $k$-stratum $P_s$, a chart has the form
\[
 U\cong V\times Y,
 \qquad V\subset\bR^k,\qquad
 P_s\cap U\cong V\times\{o_Y\},
\]
where $Y$ is an $(n-k)$-local brane with distinguished vertex $o_Y$.
Chart changes are stratified homeomorphisms, smooth on every stratum and
compatible with the resolved normal germs, ports, and collars.

The charts carry rank-$n$ bundles
$T\bR^k\oplus TY$ and specified lifts $g_{ab}$ of their transition maps,
with
\[
 g_{aa}=\mathrm{id},\qquad g_{ab}g_{bc}=g_{ac}
\]
on overlaps, with pullbacks understood. These bundles glue to the
reference tangent bundle $TB$.
For $m\geq n$, an \emph{$m$-framing} of $B$ is a trivialization
\[
 \tau:TB\oplus\underline{\bR}^{\,m-n}
 \xr{\cong}\underline{\bR}^{\,m}.
\]
For branefolds with corners, the atlas, reference bundles, and framings
are required to be products on the specified categorical boundary collars.
\end{definition}

We call a branefold decorated by $(\cC,\otimes)$ a \emph{lasagna}.

\begin{definition}[$(\cC,\otimes)$-lasagna]\label{def:lasagna}
Let $(B,\mathcal P,\tau)$ be a $k$-branefold with corners and a
stable framing. A \emph{$(\cC,\otimes)$-decoration}
consists of resolved $\cC$-stratifications $\mathcal W_s$ on its strata $P_s$,
as in Definition~\ref{def:Cstr}, with the following compatibility data:
\begin{enumerate}
 \item In a chart $V\times Y$ centered on $P_s$, the incident decorations
 determine diagrams on $V\times\partial_+Y$ and $V\times\partial_-Y$.
 Their source and target data agree with the labels prescribed along $P_s$.
 All restrictions to common lower faces agree through the specified collar
 and incidence maps, with coherent compatibility on chart overlaps.
 \item At a structural codimension-one junction, the normal $1$-local brane
 is represented by a rooted tree with
 $\partial_+Y_1=I$ and $\partial_-Y_1=\{*\}$. Its incoming labels are the
 family $(l(i))_{i\in I}$, and the decoration includes an equivalence
 \[
 l(\partial_-Y_1)\simeq\bigotimes_{i\in I}l(i).
 \]
 The empty input gives $\one_{\cC}$. The input ordering, port permutations,
 and the associated symmetric coherence are retained. General coefficient
 coupons carry the corresponding cells of $\cC$ and need not be invertible.
 \item At higher-codimension loci, retain the coherent diagrams for
 composition, associativity, units, symmetry, and substitution of these
 tensor operations, as supplied by the chosen coherent presentation of
 $\cC^\otimes$. Their restrictions agree on all incident faces and corners.
\end{enumerate}

Define $\Lasa_n^{\cC,\otimes}$ from compact collared lasagna bordisms by
gluing, disjoint union, and categorical completion. Include the coherent
local relations for coefficient composition and identities, invertible
tensor junctions, and vacuum erasure. Only pure unit-labeled identity
diagrams are vacuum: a nonidentity coupon $\one_{\cC}\to\one_{\cC}$ is
retained, even when its visible support is an isolated point. Its
categorical degree is part of the presentation.

As in Definition~\ref{def:Cstr}, comparisons above degree $n$ are
parameter families with coherent reversibility. All framings, local
resolutions, and higher presentation data are retained.
\end{definition}

\begin{center}

\tikzset{every picture/.style={line width=0.75pt}} 

\begin{adjustbox}{max totalsize={\linewidth}{0.86\textheight}}
\begin{tikzpicture}[x=0.75pt,y=0.75pt,yscale=-1,xscale=1]

\draw  [draw opacity=0][fill={rgb, 255:red, 184; green, 233; blue, 134 }  ,fill opacity=1 ] (360,140) -- (408.97,140) -- (420,170) -- (371.03,170) -- cycle ;
\draw  [draw opacity=0][fill={rgb, 255:red, 144; green, 19; blue, 254 }  ,fill opacity=1 ] (381,100) -- (430,100) -- (409,140) -- (360,140) -- cycle ;
\draw  [draw opacity=0][fill={rgb, 255:red, 208; green, 2; blue, 27 }  ,fill opacity=1 ] (360.4,140) -- (409,140) -- (357.6,170) -- (309,170) -- cycle ;
\draw [color={rgb, 255:red, 74; green, 144; blue, 226 }  ,draw opacity=1 ]   (130,120) -- (180,120) ;
\draw [color={rgb, 255:red, 208; green, 2; blue, 27 }  ,draw opacity=1 ]   (80.5,120.05) -- (130,120) ;
\draw  [fill={rgb, 255:red, 0; green, 0; blue, 0 }  ,fill opacity=1 ] (131.5,120) .. controls (131.5,119.17) and (130.83,118.5) .. (130,118.5) .. controls (129.17,118.5) and (128.5,119.17) .. (128.5,120) .. controls (128.5,120.83) and (129.17,121.5) .. (130,121.5) .. controls (130.83,121.5) and (131.5,120.83) .. (131.5,120) -- cycle ;
\draw  [draw opacity=0][fill={rgb, 255:red, 208; green, 2; blue, 27 }  ,fill opacity=1 ] (230,90) -- (260,90) -- (260,120) -- (230,150) -- cycle ;
\draw  [draw opacity=0][fill={rgb, 255:red, 184; green, 233; blue, 134 }  ,fill opacity=1 ] (260,90) -- (290,90) -- (290,150) -- (260,120) -- cycle ;
\draw  [draw opacity=0][fill={rgb, 255:red, 248; green, 231; blue, 28 }  ,fill opacity=1 ] (259.6,120) -- (290,150) -- (230,150) -- cycle ;
\draw    (260,90) -- (260,120) ;
\draw    (260,120) -- (230,150) ;
\draw    (290,150) -- (260,120) ;

\draw (124.44,94.82) node [anchor=north west][inner sep=0.75pt]    {$f$};
\draw (87.11,95.51) node [anchor=north west][inner sep=0.75pt]    {$A$};
\draw (159.78,95.29) node [anchor=north west][inner sep=0.75pt]    {$B$};
\draw (100.24,134.42) node [anchor=north west][inner sep=0.75pt]    {$A\ \xrightarrow{f} B$};
\draw (235,102.4) node [anchor=north west][inner sep=0.75pt]    {$A$};
\draw (271,102.4) node [anchor=north west][inner sep=0.75pt]    {$B$};
\draw (251,132.4) node [anchor=north west][inner sep=0.75pt]    {$C$};
\draw (257,92.4) node [anchor=north west][inner sep=0.75pt]    {$f$};
\draw (232,123.4) node [anchor=north west][inner sep=0.75pt]    {$g$};
\draw (277,122.4) node [anchor=north west][inner sep=0.75pt]    {$h$};
\draw (231,149.4) node [anchor=north west][inner sep=0.75pt]    {$A\ \xrightarrow{f} B$};
\draw (247,173.47) node [anchor=north west][inner sep=0.75pt]  [rotate=-40.46]  {$\xrightarrow[g]{}$};
\draw (253.67,188.73) node [anchor=north west][inner sep=0.75pt]    {$C$};
\draw (256.44,181.05) node [anchor=north west][inner sep=0.75pt]  [rotate=-312.36]  {$\xrightarrow[h]{}$};
\draw (351,142.4) node [anchor=north west][inner sep=0.75pt]    {$A$};
\draw (391,152.4) node [anchor=north west][inner sep=0.75pt]    {$B$};
\draw (371,112.4) node [anchor=north west][inner sep=0.75pt]    {$A\otimes B$};
\draw (251.67,169.07) node [anchor=north west][inner sep=0.75pt]    {$\Uparrow $};

\end{tikzpicture}
\end{adjustbox}

\end{center}

Thickened embeddings of lasagnas motivate the following \emph{stable realization
condition}: every compact collared lasagna admits a compatible ambient
$B^m\cC$-stratified realization after sufficiently many stabilizations,
and the space of such stabilized choices is weakly contractible. Require
this statement for parameter families, relative to prescribed boundary
and parameter collars, and compatibly with gluing, disjoint union, and the
specified local presentation relations. 

\begin{remark}
  The name suggests a connection with skein lasagna modules
  \cite{MorrisonWalkerWedrich}. The lasagnas defined here are intended to be
  intrinsic (or ``transparent''): they do not depend on a chosen embedding
  into an ambient manifold, or on a choice of embedding in $\bR^\infty$.
\end{remark}

\begin{theorem}[Lasagna Cobordism Hypothesis]\label{thm:lch}
Let $(\cC,\otimes)\in\SMCat_n$. Assume the pointed generalized CH, the
loop--deloop compatibility above, and the stable realization condition.
Then there is an equivalence of symmetric monoidal $(\infty,n)$-categories
\[
 \Lasa_n^{\cC,\otimes}
 \simeq
 \colim_{m\to\infty}\End^m\Bord_{n+m}^{B^m\cC,*}.
\]
Consequently, for every $\cD\in\SMCat_n$, there is a natural equivalence
of spaces
\[
 \rFun^\otimes(\Lasa_n^{\cC,\otimes},\cD)
 \simeq
 \rFun^\otimes(\cC,\cD^{\fd}).
\]
\end{theorem}

\begin{proof}
In the $m$-fold endomorphism category, all lower boundary data are fixed
at the vacuum. Deleting the open strata of codimension less than $m$
leaves the reduced local branes and their $\cC$-decorations. This defines
compatible reduction functors to $\Lasa_n^{\cC,\otimes}$ and hence a
functor from the displayed colimit.

Conversely, the stable realization condition supplies an ambient
$B^m\cC$-stratified representative of each lasagna, compatibly on its
collars. Its relative family version supplies the corresponding lifts
of morphisms and all higher homotopies. The contractibility of the
stabilized choice spaces makes these lifts independent of choices.
Reduction and realization are therefore inverse after completion,
proving the first equivalence.

Mapping out of the colimit and using the compatible $\cE_m$-universal
properties gives
\[
\begin{aligned}
 \rFun^\otimes(\Lasa_n^{\cC,\otimes},\cD)
 &\simeq
 \lim_m\rFun_{\operatorname{Alg}_{\cE_m}(\Cat_n)}
       (\cC,\cD^{\fd})\\
 &\simeq\rFun^\otimes(\cC,\cD^{\fd}),
\end{aligned}
\]
where both coefficient structures in the limit are restricted to $\cE_m$.
\end{proof}

\subsection{\texorpdfstring{Pontryagin--Thom Isomorphism and Algebraic $K$-Theory}{Pontryagin--Thom Isomorphism and Algebraic K-Theory}}\label{subsec:pt}

Let $\cC$ be a symmetric monoidal $(\infty,n)$-category. In this section,
$K(\cC)$ denotes the connective spectrum associated with the group
completion of the $\cE_\infty$-space $\cC^{\cong}$. Thus
\[
 \Omega^\infty K(\cC)\simeq(\cC^{\cong})^{\mathrm{gp}}.
\]
This definition uses the symmetric monoidal core; it does not additionally
encode an unspecified exact or Waldhausen structure.

For an $(\infty,n)$-category $\cD$, write $|\cD|$ for its groupoid
completion, obtained by inverting all morphisms. If $\cD$ is symmetric
monoidal, this is an $\cE_\infty$-space. When it is grouplike, let
$\operatorname{sp}|\cD|$ denote the corresponding connective spectrum,
not the suspension spectrum of its underlying space.

The classical Pontryagin--Thom construction relates bordism classes to
homotopy groups of a Thom spectrum. The lasagna analogue below gives a
spectrum-level identification. Let $\Xi_d^{\cC,\otimes}$ denote the
bordism classes of closed stable framed $d$-dimensional
$(\cC^{\cong},\otimes)$-lasagnas, using the prescribed lasagna relations
and $(d+1)$-dimensional lasagna bordisms.

\begin{theorem}[Generalized Pontryagin--Thom Isomorphism]\label{thm:gPT}
Let $n\geq1$ and $\cC\in\SMCat_n$, and assume
Theorem~\ref{thm:lch}. Then there is a natural equivalence of connective
spectra
\[
 PT:K(\cC)\xrightarrow{\;\simeq\;}
 \operatorname{sp}\bigl|\Lasa_n^{\cC^{\cong},\otimes}\bigr|.
\]
Equivalently,
\[
 \Omega^\infty K(\cC)
 \simeq\bigl|\Lasa_n^{\cC^{\cong},\otimes}\bigr|.
\]
Suppose, in addition, that the completed geometric model has the following
relative representation property: based spheres in degrees $d<n$ are
represented by closed $d$-lasagnas, and their based homotopies are
represented, relative to the endpoints, by $(d+1)$-lasagna bordisms.
Then, for $0\leq d<n$,
\[
 \Xi_d^{\cC,\otimes}
 \cong\pi_d\bigl|\Lasa_n^{\cC^{\cong},\otimes}\bigr|
 \cong\pi_dK(\cC).
\]
\end{theorem}

\begin{proof}
Put $\cA=\Lasa_n^{\cC^{\cong},\otimes}$. Its object duals make $|\cA|$
grouplike: evaluation becomes a path from $a^\vee\otimes a$ to the unit,
so $[a^\vee]+[a]=0$ in $\pi_0|\cA|$. Hence
$\operatorname{sp}|\cA|$ is defined.

For a connective spectrum $E$, regard $\Omega^\infty E$ as its additive
Picard $\infty$-groupoid. Every object and morphism has the required duals,
so $(\Omega^\infty E)^{\fd}=\Omega^\infty E$. By groupoid completion,
Theorem~\ref{thm:lch}, and group completion, respectively,
\[
\begin{aligned}
 \rMap_{\mathrm{Sp}_{\geq0}}(\operatorname{sp}|\cA|,E)
 &\simeq\rFun^\otimes(\cA,\Omega^\infty E)\\
 &\simeq\rFun^\otimes(\cC^{\cong},\Omega^\infty E)\\
 &\simeq\rMap_{\operatorname{Alg}_{\cE_\infty}(\Spc)}
               (\cC^{\cong},\Omega^\infty E)\\
 &\simeq\rMap_{\mathrm{Sp}_{\geq0}}(K(\cC),E).
\end{aligned}
\]
These equivalences are natural in $E$. Yoneda gives $PT$; equivalently,
$PT$ is induced by the coefficient functor and group completion. Thus
\[
 \pi_d|\cA|\cong\pi_dK(\cC)
 \qquad\text{for every }d\geq0.
\]
In the range $0\leq d<n$, the additional relative representation property
makes the map $\Xi_d^{\cC,\otimes}\to\pi_d|\cA|$ surjective on
representatives and injective on bordism classes. This proves the last
assertion. 
\end{proof}

Thus, once the stated geometric realization properties are established,
lasagna bordism classes give geometric representatives for algebraic
$K$-groups, even when computing either side remains difficult.

\section{\texorpdfstring{Examples: Vector Spaces and the $K$-Theory of a Field}{Examples: Vector Spaces and the K-Theory of a Field}}\label{sec:examples}
Let $\Vect_F$ denote the category of finite-dimensional vector spaces over a
field $F$. The symmetric monoidal groupoid
$(\Vect_F^{\cong},\oplus,0)$ has group-completion spectrum $K(F)$
\cite[Chapter IV]{WeibelKbook}. Throughout this section we use the generalized
Pontryagin--Thom comparison of Theorem~\ref{thm:gPT}, including its stated
geometric realization hypotheses. Thus, in the stable framed model,
\[
 \mathrm{PT}_d:\Xi_d^{\Vect_F,\oplus}\xrightarrow{\cong}K_d(F).
\]
For a fixed categorical dimension $n$, the interpretation by closed
$d$-lasagnas and $(d+1)$-dimensional bordisms is used in the range $d<n$.
Our aim is to describe low-dimensional representatives and the geometric
moves suggested by the algebra. All groups of matrices below are discrete,
and $\GL(F)=\colim_r\GL_r(F)$ denotes the stable general linear group.
Whenever a proposed surgery is used as a framed bordism, its framing must
extend over the surgery trace; an oriented filling alone does not verify
this requirement.

\subsection{Lasagna and Local Systems}\label{subsec:local-systems}
A $(\Vect_F^{\cong},\oplus)$-decoration is a local system on each smooth
piece, with direct-sum identifications at the branches. We call this a
\emph{branched local system}; it need not be locally constant on the entire
singular support.

\begin{definition}[Local system on a branefold]
Let $(B,\mathcal P)$ be a resolved branefold with its specified ports and
collars. A branched $F$-local system consists of finite-rank $F$-local
systems $L_s$ on its strata, together with the following data.
\begin{enumerate}
 \item In a chart $V\times Y$ centered on $P_s$, identify the local system
 on $P_s$ with the restriction of the outgoing decoration on
 $V\times\partial_-Y$. All restrictions use the given collar germs.
 \item At a structural codimension-one junction with
 $\partial_+Y_1=I$ and $\partial_-Y_1=\{*\}$, specify an isomorphism
 \[
   L_{\partial_-Y_1}\xrightarrow{\cong}
   \bigoplus_{i\in I}L_i.
 \]
 The empty sum is the zero local system. More generally, a coefficient
 coupon with several incoming and outgoing branches is labeled by an
 isomorphism between their direct sums. Port orderings, permutations,
 and their symmetric monoidal coherence are retained.
 \item On overlaps and at higher-codimension loci, require the induced
 composites to agree with the prescribed coefficient coherence. There
 are no independent nonidentity higher coefficient morphisms in the
 ordinary groupoid $\Vect_F^{\cong}$. 
\end{enumerate}
An isomorphism of decorations is a family of stratumwise local-system
isomorphisms commuting with all these identifications.
\end{definition}

\begin{proposition}\label{prop:ex-local-systems}
The space of $(\Vect_F^{\cong},\oplus)$-lasagna decorations on a fixed
resolved framed branefold is equivalent to the groupoid of branched
$F$-local systems on it.
\end{proposition}
\begin{proof}
On a connected smooth piece, an object of
$\operatorname{Fun}(\Pi_1(-),\Vect_F^{\cong})$ is precisely a finite-rank
local system. At a junction, a symmetric monoidal coefficient diagram
specifies exactly an isomorphism between the indicated direct sums.
Taking the homotopy limit over the pieces and their incidence maps imposes
the same compatibility on both sides, including isomorphisms of
labels. Hence the two decoration groupoids agree.
\end{proof}

Write $\Xi_d^{\Loc_F}$ for the stable framed cobordism group of closed
$d$-branefolds with these decorations. Proposition~\ref{prop:ex-local-systems}
gives
\[
 \Xi_d^{\Loc_F}\cong\Xi_d^{\Vect_F,\oplus}.
\]
For example, the junction $F,F\to F^2$ is permitted. An ordinary local
system on a connected $Y$-graph would have constant rank and would not
allow this decoration.

\subsection{\texorpdfstring{Rag\`u, $K_0(F)$, Spaghetti, and $K_1(F)$}{K0(F), Spaghetti, and K1(F)}}\label{subsec:k01}

\begin{proposition}[Signed dimension]
There is an isomorphism
\[
 \dim:\Xi_0^{\Loc_F}\xrightarrow{\cong}\bZ,
 \qquad
 [(V_i,\epsilon_i)_{i\in I}]
 \longmapsto\sum_{i\in I}\epsilon_i\dim_FV_i.
\]
\end{proposition}
\begin{proof}
The direct-sum junction imposes $[V\oplus W]=[V]+[W]$, the vacuum
relation gives $[0]=0$, and a cup identifies a negative point with the
additive inverse of its positive counterpart. Since $V\cong F^{\dim V}$,
every class is an integer multiple of the positively framed point labeled
by $F$. Signed dimension preserves these relations and sends this generator
to $1$, so it is inverse to $r\mapsto r[F]$.
\end{proof}

\begin{center}

\tikzset{every picture/.style={line width=0.75pt}} 

\begin{adjustbox}{max totalsize={\linewidth}{0.86\textheight}}
\begin{tikzpicture}[x=0.75pt,y=0.75pt,yscale=-1,xscale=1]

\draw  [fill={rgb, 255:red, 0; green, 0; blue, 0 }  ,fill opacity=1 ] (96.24,76.4) .. controls (96.24,75.09) and (97.3,74.02) .. (98.61,74.02) .. controls (99.93,74.02) and (100.99,75.09) .. (100.99,76.4) .. controls (100.99,77.71) and (99.93,78.77) .. (98.61,78.77) .. controls (97.3,78.77) and (96.24,77.71) .. (96.24,76.4) -- cycle ;
\draw  [fill={rgb, 255:red, 0; green, 0; blue, 0 }  ,fill opacity=1 ] (92.25,121.38) .. controls (92.25,120.06) and (93.31,119) .. (94.63,119) .. controls (95.94,119) and (97,120.06) .. (97,121.38) .. controls (97,122.69) and (95.94,123.75) .. (94.63,123.75) .. controls (93.31,123.75) and (92.25,122.69) .. (92.25,121.38) -- cycle ;
\draw  [fill={rgb, 255:red, 0; green, 0; blue, 0 }  ,fill opacity=1 ] (95.82,107.75) .. controls (95.82,106.44) and (96.88,105.37) .. (98.19,105.37) .. controls (99.5,105.37) and (100.57,106.44) .. (100.57,107.75) .. controls (100.57,109.06) and (99.5,110.12) .. (98.19,110.12) .. controls (96.88,110.12) and (95.82,109.06) .. (95.82,107.75) -- cycle ;
\draw  [fill={rgb, 255:red, 0; green, 0; blue, 0 }  ,fill opacity=1 ] (91.75,145.38) .. controls (91.75,144.06) and (92.81,143) .. (94.13,143) .. controls (95.44,143) and (96.5,144.06) .. (96.5,145.38) .. controls (96.5,146.69) and (95.44,147.75) .. (94.13,147.75) .. controls (92.81,147.75) and (91.75,146.69) .. (91.75,145.38) -- cycle ;
\draw  [draw opacity=0] (101.08,76.4) .. controls (101.12,76.4) and (101.15,76.4) .. (101.19,76.4) .. controls (117.76,76.4) and (131.19,83.42) .. (131.19,92.07) .. controls (131.19,100.73) and (117.76,107.75) .. (101.19,107.75) .. controls (100.98,107.75) and (100.77,107.75) .. (100.57,107.75) -- (101.19,92.07) -- cycle ; \draw   (101.08,76.4) .. controls (101.12,76.4) and (101.15,76.4) .. (101.19,76.4) .. controls (117.76,76.4) and (131.19,83.42) .. (131.19,92.07) .. controls (131.19,100.73) and (117.76,107.75) .. (101.19,107.75) .. controls (100.98,107.75) and (100.77,107.75) .. (100.57,107.75) ;  
\draw    (94.63,121.38) .. controls (110.55,127.37) and (123.2,121.25) .. (137.2,113.25) .. controls (151.2,105.25) and (151.7,79.75) .. (175.7,79.25) ;
\draw    (94.13,145.38) .. controls (107.2,136.75) and (110.2,128.25) .. (137.2,113.25) ;
\draw  [fill={rgb, 255:red, 0; green, 0; blue, 0 }  ,fill opacity=1 ] (171.74,114.9) .. controls (171.74,113.59) and (172.8,112.52) .. (174.11,112.52) .. controls (175.43,112.52) and (176.49,113.59) .. (176.49,114.9) .. controls (176.49,116.21) and (175.43,117.27) .. (174.11,117.27) .. controls (172.8,117.27) and (171.74,116.21) .. (171.74,114.9) -- cycle ;
\draw  [fill={rgb, 255:red, 0; green, 0; blue, 0 }  ,fill opacity=1 ] (171.32,146.25) .. controls (171.32,144.94) and (172.38,143.87) .. (173.69,143.87) .. controls (175,143.87) and (176.07,144.94) .. (176.07,146.25) .. controls (176.07,147.56) and (175,148.62) .. (173.69,148.62) .. controls (172.38,148.62) and (171.32,147.56) .. (171.32,146.25) -- cycle ;
\draw  [draw opacity=0] (173.69,146.25) .. controls (173.66,146.25) and (173.62,146.25) .. (173.59,146.25) .. controls (157.02,146.25) and (143.59,139.23) .. (143.59,130.57) .. controls (143.59,121.91) and (157.02,114.89) .. (173.59,114.89) .. controls (173.8,114.89) and (174,114.9) .. (174.21,114.9) -- (173.59,130.57) -- cycle ; \draw   (173.69,146.25) .. controls (173.66,146.25) and (173.62,146.25) .. (173.59,146.25) .. controls (157.02,146.25) and (143.59,139.23) .. (143.59,130.57) .. controls (143.59,121.91) and (157.02,114.89) .. (173.59,114.89) .. controls (173.8,114.89) and (174,114.9) .. (174.21,114.9) ;  
\draw  [fill={rgb, 255:red, 0; green, 0; blue, 0 }  ,fill opacity=1 ] (173.32,79.25) .. controls (173.32,77.94) and (174.39,76.88) .. (175.7,76.88) .. controls (177.01,76.88) and (178.07,77.94) .. (178.07,79.25) .. controls (178.07,80.56) and (177.01,81.63) .. (175.7,81.63) .. controls (174.39,81.63) and (173.32,80.56) .. (173.32,79.25) -- cycle ;

\draw (79.5,63.4) node [anchor=north west][inner sep=0.75pt]    {$V$};
\draw (66,97.4) node [anchor=north west][inner sep=0.75pt]    {$-V$};
\draw (69,116.9) node [anchor=north west][inner sep=0.75pt]    {$W$};
\draw (73.5,145.9) node [anchor=north west][inner sep=0.75pt]    {$X$};
\draw (180.5,68.4) node [anchor=north west][inner sep=0.75pt]    {$W\oplus X$};
\draw (193,106.9) node [anchor=north west][inner sep=0.75pt]    {$Y$};
\draw (181.67,140.23) node [anchor=north west][inner sep=0.75pt]    {$-Y$};

\end{tikzpicture}
\end{adjustbox}

\end{center}

A closed $1$-branefold is generally a graph with junctions, rather than
literally a wedge of circles. The structural merge, split, and duality
moves suggest circle representatives. We normalize their framing before
using them as generators.

Fix an orientation of $S^1$. Let $\tau_0$ be its \emph{bounding} stable
framing, and let $\tau_1$ be the other stable framing. The latter represents
$\eta\in\pi_1^s(\mathbb S)\cong\bZ/2$. These are stable framing classes;
a two-dimensional unstabilized framing has an integer winding number.
For $g\in\GL_r(F)$, write
\[
 L_r(g):=[S^1,\tau_0;F^r_g]\in\Xi_1^{\Loc_F},
\]
where $F^r_g$ has monodromy $g$. Using the bounding framing ensures
$L_r(I_r)=0$.

\begin{lemma}[Framing and the self-swap]\label{lem:ex-framing}
Under the Pontryagin--Thom comparison,
\[
 [S^1,\tau_1;\underline{F^r}]
 =L_{2r}\!\left(\begin{smallmatrix}0&I_r\\I_r&0\end{smallmatrix}\right)
 =L_r(-I_r).
\]
In particular, its determinant is $(-1)^r$.
\end{lemma}
\begin{proof}
The nonbounding framed circle acts by the stable Hopf element. In additive
$K$-theory,
\[
 \eta[V]=[\operatorname{swap}_{V,V}]=[-\operatorname{id}_V]
 \quad\text{in }K_1(F)
\]
\cite[Corollaries 1.9--1.10]{MuroTonks}. Apply the stable
Pontryagin--Thom comparison and compute
$\det(\operatorname{swap}_{F^r,F^r})=(-1)^{r^2}=(-1)^r$.
\end{proof}
This identifies the framing correction as a cobordism class. A proposed
picture of the correction must retain the corresponding stable framing;
the underlying unframed disk cannot fill both framing classes.

\begin{center}

\tikzset{every picture/.style={line width=0.75pt}} 

\begin{adjustbox}{max totalsize={\linewidth}{0.86\textheight}}
\begin{tikzpicture}[x=0.75pt,y=0.75pt,yscale=-1,xscale=1]

\draw   (41.18,110.59) .. controls (41.18,88.83) and (58.83,71.18) .. (80.59,71.18) .. controls (102.36,71.18) and (120,88.83) .. (120,110.59) .. controls (120,132.36) and (102.36,150) .. (80.59,150) .. controls (58.83,150) and (41.18,132.36) .. (41.18,110.59) -- cycle ;
\draw   (170,110.04) .. controls (170,87.94) and (187.91,70.04) .. (210,70.04) .. controls (232.09,70.04) and (250,87.94) .. (250,110.04) .. controls (250,132.13) and (232.09,150.04) .. (210,150.04) .. controls (187.91,150.04) and (170,132.13) .. (170,110.04) -- cycle ;
\draw    (80,70) -- (98,70) ;
\draw [shift={(100,70)}, rotate = 180] [color={rgb, 255:red, 0; green, 0; blue, 0 }  ][line width=0.75]    (10.93,-3.29) .. controls (6.95,-1.4) and (3.31,-0.3) .. (0,0) .. controls (3.31,0.3) and (6.95,1.4) .. (10.93,3.29)   ;
\draw    (80,70) -- (80,50) ;
\draw [shift={(80,50)}, rotate = 90] [color={rgb, 255:red, 0; green, 0; blue, 0 }  ][line width=0.75]    (17.64,-3.29) .. controls (13.66,-1.4) and (10.02,-0.3) .. (6.71,0) .. controls (10.02,0.3) and (13.66,1.4) .. (17.64,3.29)(10.93,-3.29) .. controls (6.95,-1.4) and (3.31,-0.3) .. (0,0) .. controls (3.31,0.3) and (6.95,1.4) .. (10.93,3.29)   ;

\draw    (120,110) -- (138,110) ;
\draw [shift={(140,110)}, rotate = 180] [color={rgb, 255:red, 0; green, 0; blue, 0 }  ][line width=0.75]    (10.93,-3.29) .. controls (6.95,-1.4) and (3.31,-0.3) .. (0,0) .. controls (3.31,0.3) and (6.95,1.4) .. (10.93,3.29)   ;
\draw    (120,110) -- (120,90) ;
\draw [shift={(120,90)}, rotate = 90] [color={rgb, 255:red, 0; green, 0; blue, 0 }  ][line width=0.75]    (17.64,-3.29) .. controls (13.66,-1.4) and (10.02,-0.3) .. (6.71,0) .. controls (10.02,0.3) and (13.66,1.4) .. (17.64,3.29)(10.93,-3.29) .. controls (6.95,-1.4) and (3.31,-0.3) .. (0,0) .. controls (3.31,0.3) and (6.95,1.4) .. (10.93,3.29)   ;

\draw    (80,150) -- (98,150) ;
\draw [shift={(100,150)}, rotate = 180] [color={rgb, 255:red, 0; green, 0; blue, 0 }  ][line width=0.75]    (10.93,-3.29) .. controls (6.95,-1.4) and (3.31,-0.3) .. (0,0) .. controls (3.31,0.3) and (6.95,1.4) .. (10.93,3.29)   ;
\draw    (80,150) -- (80,130) ;
\draw [shift={(80,130)}, rotate = 90] [color={rgb, 255:red, 0; green, 0; blue, 0 }  ][line width=0.75]    (17.64,-3.29) .. controls (13.66,-1.4) and (10.02,-0.3) .. (6.71,0) .. controls (10.02,0.3) and (13.66,1.4) .. (17.64,3.29)(10.93,-3.29) .. controls (6.95,-1.4) and (3.31,-0.3) .. (0,0) .. controls (3.31,0.3) and (6.95,1.4) .. (10.93,3.29)   ;

\draw    (40,110) -- (58,110) ;
\draw [shift={(60,110)}, rotate = 180] [color={rgb, 255:red, 0; green, 0; blue, 0 }  ][line width=0.75]    (10.93,-3.29) .. controls (6.95,-1.4) and (3.31,-0.3) .. (0,0) .. controls (3.31,0.3) and (6.95,1.4) .. (10.93,3.29)   ;
\draw    (40,110) -- (40,90) ;
\draw [shift={(40,90)}, rotate = 90] [color={rgb, 255:red, 0; green, 0; blue, 0 }  ][line width=0.75]    (17.64,-3.29) .. controls (13.66,-1.4) and (10.02,-0.3) .. (6.71,0) .. controls (10.02,0.3) and (13.66,1.4) .. (17.64,3.29)(10.93,-3.29) .. controls (6.95,-1.4) and (3.31,-0.3) .. (0,0) .. controls (3.31,0.3) and (6.95,1.4) .. (10.93,3.29)   ;

\draw    (210,70.04) -- (228,70.04) ;
\draw [shift={(230,70.04)}, rotate = 180] [color={rgb, 255:red, 0; green, 0; blue, 0 }  ][line width=0.75]    (10.93,-3.29) .. controls (6.95,-1.4) and (3.31,-0.3) .. (0,0) .. controls (3.31,0.3) and (6.95,1.4) .. (10.93,3.29)   ;
\draw    (210,70.04) -- (210,50.04) ;
\draw [shift={(210,50.04)}, rotate = 90] [color={rgb, 255:red, 0; green, 0; blue, 0 }  ][line width=0.75]    (17.64,-3.29) .. controls (13.66,-1.4) and (10.02,-0.3) .. (6.71,0) .. controls (10.02,0.3) and (13.66,1.4) .. (17.64,3.29)(10.93,-3.29) .. controls (6.95,-1.4) and (3.31,-0.3) .. (0,0) .. controls (3.31,0.3) and (6.95,1.4) .. (10.93,3.29)   ;

\draw    (250,110.04) -- (250,128.04) ;
\draw [shift={(250,130.04)}, rotate = 270] [color={rgb, 255:red, 0; green, 0; blue, 0 }  ][line width=0.75]    (10.93,-3.29) .. controls (6.95,-1.4) and (3.31,-0.3) .. (0,0) .. controls (3.31,0.3) and (6.95,1.4) .. (10.93,3.29)   ;
\draw    (250,110.04) -- (270,110.04) ;
\draw [shift={(270,110.04)}, rotate = 180] [color={rgb, 255:red, 0; green, 0; blue, 0 }  ][line width=0.75]    (17.64,-3.29) .. controls (13.66,-1.4) and (10.02,-0.3) .. (6.71,0) .. controls (10.02,0.3) and (13.66,1.4) .. (17.64,3.29)(10.93,-3.29) .. controls (6.95,-1.4) and (3.31,-0.3) .. (0,0) .. controls (3.31,0.3) and (6.95,1.4) .. (10.93,3.29)   ;

\draw    (210,150.04) -- (192,150.04) ;
\draw [shift={(190,150.04)}, rotate = 360] [color={rgb, 255:red, 0; green, 0; blue, 0 }  ][line width=0.75]    (10.93,-3.29) .. controls (6.95,-1.4) and (3.31,-0.3) .. (0,0) .. controls (3.31,0.3) and (6.95,1.4) .. (10.93,3.29)   ;
\draw    (210,150.04) -- (210,170.04) ;
\draw [shift={(210,170.04)}, rotate = 270] [color={rgb, 255:red, 0; green, 0; blue, 0 }  ][line width=0.75]    (17.64,-3.29) .. controls (13.66,-1.4) and (10.02,-0.3) .. (6.71,0) .. controls (10.02,0.3) and (13.66,1.4) .. (17.64,3.29)(10.93,-3.29) .. controls (6.95,-1.4) and (3.31,-0.3) .. (0,0) .. controls (3.31,0.3) and (6.95,1.4) .. (10.93,3.29)   ;

\draw    (170,110.04) -- (170,92.04) ;
\draw [shift={(170,90.04)}, rotate = 90] [color={rgb, 255:red, 0; green, 0; blue, 0 }  ][line width=0.75]    (10.93,-3.29) .. controls (6.95,-1.4) and (3.31,-0.3) .. (0,0) .. controls (3.31,0.3) and (6.95,1.4) .. (10.93,3.29)   ;
\draw    (170,110.04) -- (150,110.04) ;
\draw [shift={(150,110.04)}, rotate = 360] [color={rgb, 255:red, 0; green, 0; blue, 0 }  ][line width=0.75]    (17.64,-3.29) .. controls (13.66,-1.4) and (10.02,-0.3) .. (6.71,0) .. controls (10.02,0.3) and (13.66,1.4) .. (17.64,3.29)(10.93,-3.29) .. controls (6.95,-1.4) and (3.31,-0.3) .. (0,0) .. controls (3.31,0.3) and (6.95,1.4) .. (10.93,3.29)   ;

\draw    (288.11,100) .. controls (329.31,98.27) and (327.98,132.27) .. (368.11,130) ;
\draw    (288.11,130) .. controls (325.31,128.93) and (331.31,98.27) .. (368.11,100) ;
\draw    (288.11,100) -- (306.11,100) ;
\draw [shift={(308.11,100)}, rotate = 180] [color={rgb, 255:red, 0; green, 0; blue, 0 }  ][line width=0.75]    (10.93,-3.29) .. controls (6.95,-1.4) and (3.31,-0.3) .. (0,0) .. controls (3.31,0.3) and (6.95,1.4) .. (10.93,3.29)   ;
\draw    (288.11,100) -- (288.11,80) ;
\draw [shift={(288.11,80)}, rotate = 90] [color={rgb, 255:red, 0; green, 0; blue, 0 }  ][line width=0.75]    (17.64,-3.29) .. controls (13.66,-1.4) and (10.02,-0.3) .. (6.71,0) .. controls (10.02,0.3) and (13.66,1.4) .. (17.64,3.29)(10.93,-3.29) .. controls (6.95,-1.4) and (3.31,-0.3) .. (0,0) .. controls (3.31,0.3) and (6.95,1.4) .. (10.93,3.29)   ;

\draw    (288.11,130) -- (306.11,130) ;
\draw [shift={(308.11,130)}, rotate = 180] [color={rgb, 255:red, 0; green, 0; blue, 0 }  ][line width=0.75]    (10.93,-3.29) .. controls (6.95,-1.4) and (3.31,-0.3) .. (0,0) .. controls (3.31,0.3) and (6.95,1.4) .. (10.93,3.29)   ;
\draw    (288.11,130) -- (288.11,110) ;
\draw [shift={(288.11,110)}, rotate = 90] [color={rgb, 255:red, 0; green, 0; blue, 0 }  ][line width=0.75]    (17.64,-3.29) .. controls (13.66,-1.4) and (10.02,-0.3) .. (6.71,0) .. controls (10.02,0.3) and (13.66,1.4) .. (17.64,3.29)(10.93,-3.29) .. controls (6.95,-1.4) and (3.31,-0.3) .. (0,0) .. controls (3.31,0.3) and (6.95,1.4) .. (10.93,3.29)   ;

\draw    (358.11,100) -- (376.11,100) ;
\draw [shift={(378.11,100)}, rotate = 180] [color={rgb, 255:red, 0; green, 0; blue, 0 }  ][line width=0.75]    (10.93,-3.29) .. controls (6.95,-1.4) and (3.31,-0.3) .. (0,0) .. controls (3.31,0.3) and (6.95,1.4) .. (10.93,3.29)   ;
\draw    (358.11,100) -- (358.11,80) ;
\draw [shift={(358.11,80)}, rotate = 90] [color={rgb, 255:red, 0; green, 0; blue, 0 }  ][line width=0.75]    (17.64,-3.29) .. controls (13.66,-1.4) and (10.02,-0.3) .. (6.71,0) .. controls (10.02,0.3) and (13.66,1.4) .. (17.64,3.29)(10.93,-3.29) .. controls (6.95,-1.4) and (3.31,-0.3) .. (0,0) .. controls (3.31,0.3) and (6.95,1.4) .. (10.93,3.29)   ;

\draw    (358.11,130) -- (376.11,130) ;
\draw [shift={(378.11,130)}, rotate = 180] [color={rgb, 255:red, 0; green, 0; blue, 0 }  ][line width=0.75]    (10.93,-3.29) .. controls (6.95,-1.4) and (3.31,-0.3) .. (0,0) .. controls (3.31,0.3) and (6.95,1.4) .. (10.93,3.29)   ;
\draw    (358.11,130) -- (358.11,110) ;
\draw [shift={(358.11,110)}, rotate = 90] [color={rgb, 255:red, 0; green, 0; blue, 0 }  ][line width=0.75]    (17.64,-3.29) .. controls (13.66,-1.4) and (10.02,-0.3) .. (6.71,0) .. controls (10.02,0.3) and (13.66,1.4) .. (17.64,3.29)(10.93,-3.29) .. controls (6.95,-1.4) and (3.31,-0.3) .. (0,0) .. controls (3.31,0.3) and (6.95,1.4) .. (10.93,3.29)   ;

\draw    (288.11,100) .. controls (244.64,88.27) and (288.02,71.13) .. (328.11,70) .. controls (368.2,68.87) and (410.64,93.6) .. (358.11,100) ;
\draw    (298.11,130) .. controls (255.98,137.6) and (288.02,161.13) .. (328.11,160) .. controls (368.2,158.87) and (404.64,139.6) .. (368.11,130) ;
\draw    (328.11,70) -- (346.11,70) ;
\draw [shift={(348.11,70)}, rotate = 180] [color={rgb, 255:red, 0; green, 0; blue, 0 }  ][line width=0.75]    (10.93,-3.29) .. controls (6.95,-1.4) and (3.31,-0.3) .. (0,0) .. controls (3.31,0.3) and (6.95,1.4) .. (10.93,3.29)   ;
\draw    (328.11,70) -- (328.11,50) ;
\draw [shift={(328.11,50)}, rotate = 90] [color={rgb, 255:red, 0; green, 0; blue, 0 }  ][line width=0.75]    (17.64,-3.29) .. controls (13.66,-1.4) and (10.02,-0.3) .. (6.71,0) .. controls (10.02,0.3) and (13.66,1.4) .. (17.64,3.29)(10.93,-3.29) .. controls (6.95,-1.4) and (3.31,-0.3) .. (0,0) .. controls (3.31,0.3) and (6.95,1.4) .. (10.93,3.29)   ;

\draw    (328.11,160) -- (346.11,160) ;
\draw [shift={(348.11,160)}, rotate = 180] [color={rgb, 255:red, 0; green, 0; blue, 0 }  ][line width=0.75]    (10.93,-3.29) .. controls (6.95,-1.4) and (3.31,-0.3) .. (0,0) .. controls (3.31,0.3) and (6.95,1.4) .. (10.93,3.29)   ;
\draw    (328.11,160) -- (328.11,140) ;
\draw [shift={(328.11,140)}, rotate = 90] [color={rgb, 255:red, 0; green, 0; blue, 0 }  ][line width=0.75]    (17.64,-3.29) .. controls (13.66,-1.4) and (10.02,-0.3) .. (6.71,0) .. controls (10.02,0.3) and (13.66,1.4) .. (17.64,3.29)(10.93,-3.29) .. controls (6.95,-1.4) and (3.31,-0.3) .. (0,0) .. controls (3.31,0.3) and (6.95,1.4) .. (10.93,3.29)   ;

\draw    (390,90) -- (408,90) ;
\draw [shift={(410,90)}, rotate = 180] [color={rgb, 255:red, 0; green, 0; blue, 0 }  ][line width=0.75]    (10.93,-3.29) .. controls (6.95,-1.4) and (3.31,-0.3) .. (0,0) .. controls (3.31,0.3) and (6.95,1.4) .. (10.93,3.29)   ;
\draw    (390,90) -- (390,70) ;
\draw [shift={(390,70)}, rotate = 90] [color={rgb, 255:red, 0; green, 0; blue, 0 }  ][line width=0.75]    (17.64,-3.29) .. controls (13.66,-1.4) and (10.02,-0.3) .. (6.71,0) .. controls (10.02,0.3) and (13.66,1.4) .. (17.64,3.29)(10.93,-3.29) .. controls (6.95,-1.4) and (3.31,-0.3) .. (0,0) .. controls (3.31,0.3) and (6.95,1.4) .. (10.93,3.29)   ;

\draw    (390,140.04) -- (408,140.04) ;
\draw [shift={(410,140.04)}, rotate = 180] [color={rgb, 255:red, 0; green, 0; blue, 0 }  ][line width=0.75]    (10.93,-3.29) .. controls (6.95,-1.4) and (3.31,-0.3) .. (0,0) .. controls (3.31,0.3) and (6.95,1.4) .. (10.93,3.29)   ;
\draw    (390,140.04) -- (390,120.04) ;
\draw [shift={(390,120.04)}, rotate = 90] [color={rgb, 255:red, 0; green, 0; blue, 0 }  ][line width=0.75]    (17.64,-3.29) .. controls (13.66,-1.4) and (10.02,-0.3) .. (6.71,0) .. controls (10.02,0.3) and (13.66,1.4) .. (17.64,3.29)(10.93,-3.29) .. controls (6.95,-1.4) and (3.31,-0.3) .. (0,0) .. controls (3.31,0.3) and (6.95,1.4) .. (10.93,3.29)   ;

\draw    (500,90) -- (518,90) ;
\draw [shift={(520,90)}, rotate = 180] [color={rgb, 255:red, 0; green, 0; blue, 0 }  ][line width=0.75]    (10.93,-3.29) .. controls (6.95,-1.4) and (3.31,-0.3) .. (0,0) .. controls (3.31,0.3) and (6.95,1.4) .. (10.93,3.29)   ;
\draw    (500,90) -- (500,70) ;
\draw [shift={(500,70)}, rotate = 90] [color={rgb, 255:red, 0; green, 0; blue, 0 }  ][line width=0.75]    (17.64,-3.29) .. controls (13.66,-1.4) and (10.02,-0.3) .. (6.71,0) .. controls (10.02,0.3) and (13.66,1.4) .. (17.64,3.29)(10.93,-3.29) .. controls (6.95,-1.4) and (3.31,-0.3) .. (0,0) .. controls (3.31,0.3) and (6.95,1.4) .. (10.93,3.29)   ;

\draw    (500,140) -- (518,140) ;
\draw [shift={(520,140)}, rotate = 180] [color={rgb, 255:red, 0; green, 0; blue, 0 }  ][line width=0.75]    (10.93,-3.29) .. controls (6.95,-1.4) and (3.31,-0.3) .. (0,0) .. controls (3.31,0.3) and (6.95,1.4) .. (10.93,3.29)   ;
\draw    (500,140) -- (500,120) ;
\draw [shift={(500,120)}, rotate = 90] [color={rgb, 255:red, 0; green, 0; blue, 0 }  ][line width=0.75]    (17.64,-3.29) .. controls (13.66,-1.4) and (10.02,-0.3) .. (6.71,0) .. controls (10.02,0.3) and (13.66,1.4) .. (17.64,3.29)(10.93,-3.29) .. controls (6.95,-1.4) and (3.31,-0.3) .. (0,0) .. controls (3.31,0.3) and (6.95,1.4) .. (10.93,3.29)   ;

\draw    (410,110.04) .. controls (363.87,88.3) and (409.91,71.13) .. (450,70) .. controls (490.09,68.87) and (521.2,91.64) .. (480,110.04) ;
\draw    (410,120.04) .. controls (365.87,139.64) and (404.85,161.13) .. (444.94,160) .. controls (485.03,158.87) and (525.2,139.64) .. (480,120.04) ;
\draw    (444.94,70) -- (462.94,70) ;
\draw [shift={(464.94,70)}, rotate = 180] [color={rgb, 255:red, 0; green, 0; blue, 0 }  ][line width=0.75]    (10.93,-3.29) .. controls (6.95,-1.4) and (3.31,-0.3) .. (0,0) .. controls (3.31,0.3) and (6.95,1.4) .. (10.93,3.29)   ;
\draw    (444.94,70) -- (444.94,50) ;
\draw [shift={(444.94,50)}, rotate = 90] [color={rgb, 255:red, 0; green, 0; blue, 0 }  ][line width=0.75]    (17.64,-3.29) .. controls (13.66,-1.4) and (10.02,-0.3) .. (6.71,0) .. controls (10.02,0.3) and (13.66,1.4) .. (17.64,3.29)(10.93,-3.29) .. controls (6.95,-1.4) and (3.31,-0.3) .. (0,0) .. controls (3.31,0.3) and (6.95,1.4) .. (10.93,3.29)   ;

\draw    (444.94,160) -- (462.94,160) ;
\draw [shift={(464.94,160)}, rotate = 180] [color={rgb, 255:red, 0; green, 0; blue, 0 }  ][line width=0.75]    (10.93,-3.29) .. controls (6.95,-1.4) and (3.31,-0.3) .. (0,0) .. controls (3.31,0.3) and (6.95,1.4) .. (10.93,3.29)   ;
\draw    (444.94,160) -- (444.94,140) ;
\draw [shift={(444.94,140)}, rotate = 90] [color={rgb, 255:red, 0; green, 0; blue, 0 }  ][line width=0.75]    (17.64,-3.29) .. controls (13.66,-1.4) and (10.02,-0.3) .. (6.71,0) .. controls (10.02,0.3) and (13.66,1.4) .. (17.64,3.29)(10.93,-3.29) .. controls (6.95,-1.4) and (3.31,-0.3) .. (0,0) .. controls (3.31,0.3) and (6.95,1.4) .. (10.93,3.29)   ;

\draw   (410,100) -- (480,100) -- (480,130) -- (410,130) -- cycle ;

\draw (72,162.4) node [anchor=north west][inner sep=0.75pt]    {$\tau _{0}$};
\draw (252,152.44) node [anchor=north west][inner sep=0.75pt]    {$\tau _{1}$};
\draw (267,103.4) node [anchor=north west][inner sep=0.75pt]    {$=$};
\draw (371,103.44) node [anchor=north west][inner sep=0.75pt]    {$=$};
\draw (421,103.4) node [anchor=north west][inner sep=0.75pt]    {$\mathrm{swap}$};

\end{tikzpicture}
\end{adjustbox}

\end{center}

Direct-sum junctions, product collars, and framed pairs of pants give
\begin{equation}\label{eq:ex-circle-relations}
\begin{aligned}
 L_{r+s}(g\oplus h)&=L_r(g)+L_s(h),&
 L_{r+1}(g\oplus1)&=L_r(g),\\
 L_r(gh)&=L_r(g)+L_r(h),&
 L_r(kgk^{-1})&=L_r(g).
\end{aligned}
\end{equation}
For the multiplication relation, the three boundary holonomies of a pair
of pants are $g,h,(gh)^{-1}$, with their boundary orientations. Choose the
stable framing inducing the bounding framing on these boundary circles.
The block-sum relation uses a direct-sum junction transported around a
circle. Thus stabilization defines a homomorphism
\[
 L=L_\infty:\GL(F)\longrightarrow\Xi_1^{\Loc_F}.
\]

\begin{center}

\tikzset{every picture/.style={line width=0.75pt}} 

\begin{adjustbox}{max totalsize={\linewidth}{0.86\textheight}}
\begin{tikzpicture}[x=0.75pt,y=0.75pt,yscale=-1,xscale=1]

\draw  [draw opacity=0][dash pattern={on 4.5pt off 4.5pt}] (184.19,96.87) .. controls (184.02,96.9) and (183.84,96.91) .. (183.67,96.91) .. controls (178.56,96.91) and (174.42,83.48) .. (174.42,66.91) .. controls (174.42,50.34) and (178.56,36.91) .. (183.67,36.91) .. controls (183.93,36.91) and (184.19,36.95) .. (184.45,37.02) -- (183.67,66.91) -- cycle ; \draw  [dash pattern={on 4.5pt off 4.5pt}] (184.19,96.87) .. controls (184.02,96.9) and (183.84,96.91) .. (183.67,96.91) .. controls (178.56,96.91) and (174.42,83.48) .. (174.42,66.91) .. controls (174.42,50.34) and (178.56,36.91) .. (183.67,36.91) .. controls (183.93,36.91) and (184.19,36.95) .. (184.45,37.02) ;  
\draw   (22.67,66.67) .. controls (22.67,52.86) and (33.86,41.67) .. (47.67,41.67) .. controls (61.47,41.67) and (72.67,52.86) .. (72.67,66.67) .. controls (72.67,80.47) and (61.47,91.67) .. (47.67,91.67) .. controls (33.86,91.67) and (22.67,80.47) .. (22.67,66.67) -- cycle ;
\draw  [fill={rgb, 255:red, 0; green, 0; blue, 0 }  ,fill opacity=1 ] (32.55,44.87) .. controls (32.55,43.36) and (33.76,42.15) .. (35.27,42.15) .. controls (36.77,42.15) and (37.99,43.36) .. (37.99,44.87) .. controls (37.99,46.37) and (36.77,47.59) .. (35.27,47.59) .. controls (33.76,47.59) and (32.55,46.37) .. (32.55,44.87) -- cycle ;
\draw   (132.19,36.87) -- (184.19,36.87) .. controls (189.16,36.87) and (193.19,50.3) .. (193.19,66.87) .. controls (193.19,83.44) and (189.16,96.87) .. (184.19,96.87) -- (132.19,96.87) .. controls (127.22,96.87) and (123.19,83.44) .. (123.19,66.87) .. controls (123.19,50.3) and (127.22,36.87) .. (132.19,36.87) .. controls (137.16,36.87) and (141.19,50.3) .. (141.19,66.87) .. controls (141.19,83.44) and (137.16,96.87) .. (132.19,96.87) ;
\draw  [fill={rgb, 255:red, 0; green, 0; blue, 0 }  ,fill opacity=1 ] (44.95,41.67) .. controls (44.95,40.16) and (46.16,38.95) .. (47.67,38.95) .. controls (49.17,38.95) and (50.39,40.16) .. (50.39,41.67) .. controls (50.39,43.17) and (49.17,44.39) .. (47.67,44.39) .. controls (46.16,44.39) and (44.95,43.17) .. (44.95,41.67) -- cycle ;
\draw    (139.23,50.75) .. controls (150.27,47.83) and (166.03,53.55) .. (170.43,63.55) ;
\draw    (141.63,75.55) .. controls (151.63,76.11) and (167.23,71.31) .. (170.43,63.55) ;
\draw    (170.43,63.55) -- (193.63,63.31) ;
\draw  [fill={rgb, 255:red, 0; green, 0; blue, 0 }  ,fill opacity=1 ] (58.15,45.27) .. controls (58.15,43.76) and (59.36,42.55) .. (60.87,42.55) .. controls (62.37,42.55) and (63.59,43.76) .. (63.59,45.27) .. controls (63.59,46.77) and (62.37,47.99) .. (60.87,47.99) .. controls (59.36,47.99) and (58.15,46.77) .. (58.15,45.27) -- cycle ;
\draw   (238.4,40.55) -- (277.95,40.55) .. controls (283.14,40.55) and (287.35,53.98) .. (287.35,70.55) .. controls (287.35,87.12) and (283.14,100.55) .. (277.95,100.55) -- (238.4,100.55) .. controls (233.21,100.55) and (229,87.12) .. (229,70.55) .. controls (229,53.98) and (233.21,40.55) .. (238.4,40.55) .. controls (243.59,40.55) and (247.8,53.98) .. (247.8,70.55) .. controls (247.8,87.12) and (243.59,100.55) .. (238.4,100.55) ;
\draw    (247.8,65.2) -- (288.08,65.22) ;
\draw    (279.24,40.55) -- (347.91,31.85) ;
\draw    (279.24,100.55) -- (345.08,112.13) ;
\draw  [dash pattern={on 2.25pt off 1.5pt}]  (279.24,40.55) -- (310.53,55.1) ;
\draw  [dash pattern={on 2.25pt off 1.5pt}]  (279.24,100.55) -- (309.2,88.1) ;
\draw  [draw opacity=0][dash pattern={on 4.5pt off 4.5pt}] (310.53,55.1) .. controls (312.72,56.31) and (314.4,63.23) .. (314.4,71.58) .. controls (314.4,80.78) and (312.36,88.25) .. (309.85,88.25) .. controls (307.33,88.25) and (305.29,80.78) .. (305.29,71.58) .. controls (305.29,62.37) and (307.33,54.91) .. (309.85,54.91) .. controls (309.93,54.91) and (310,54.92) .. (310.08,54.93) -- (309.85,71.58) -- cycle ; \draw  [dash pattern={on 4.5pt off 4.5pt}] (310.53,55.1) .. controls (312.72,56.31) and (314.4,63.23) .. (314.4,71.58) .. controls (314.4,80.78) and (312.36,88.25) .. (309.85,88.25) .. controls (307.33,88.25) and (305.29,80.78) .. (305.29,71.58) .. controls (305.29,62.37) and (307.33,54.91) .. (309.85,54.91) .. controls (309.93,54.91) and (310,54.92) .. (310.08,54.93) ;  
\draw  [draw opacity=0][dash pattern={on 4.5pt off 4.5pt}] (310.08,54.93) .. controls (318.98,55.28) and (326.08,62.38) .. (326.08,71.09) .. controls (326.08,80.02) and (318.62,87.26) .. (309.41,87.26) .. controls (309.23,87.26) and (309.06,87.26) .. (308.88,87.25) -- (309.41,71.09) -- cycle ; \draw  [dash pattern={on 4.5pt off 4.5pt}] (310.08,54.93) .. controls (318.98,55.28) and (326.08,62.38) .. (326.08,71.09) .. controls (326.08,80.02) and (318.62,87.26) .. (309.41,87.26) .. controls (309.23,87.26) and (309.06,87.26) .. (308.88,87.25) ;  
\draw  [draw opacity=0] (347.91,31.85) .. controls (355.51,33.18) and (361.53,50.7) .. (361.53,72.12) .. controls (361.53,94.4) and (355.01,112.47) .. (346.97,112.47) .. controls (346.33,112.47) and (345.7,112.35) .. (345.08,112.13) -- (346.97,72.12) -- cycle ; \draw   (347.91,31.85) .. controls (355.51,33.18) and (361.53,50.7) .. (361.53,72.12) .. controls (361.53,94.4) and (355.01,112.47) .. (346.97,112.47) .. controls (346.33,112.47) and (345.7,112.35) .. (345.08,112.13) ;  
\draw    (288.08,65.22) -- (361.4,58.4) ;
\draw  [draw opacity=0][dash pattern={on 4.5pt off 4.5pt}] (345.08,112.13) .. controls (338,109.39) and (332.55,92.54) .. (332.55,72.17) .. controls (332.55,49.88) and (339.07,31.82) .. (347.11,31.82) .. controls (347.54,31.82) and (347.96,31.87) .. (348.38,31.97) -- (347.11,72.17) -- cycle ; \draw  [dash pattern={on 4.5pt off 4.5pt}] (345.08,112.13) .. controls (338,109.39) and (332.55,92.54) .. (332.55,72.17) .. controls (332.55,49.88) and (339.07,31.82) .. (347.11,31.82) .. controls (347.54,31.82) and (347.96,31.87) .. (348.38,31.97) ;  
\draw  [draw opacity=0][dash pattern={on 4.5pt off 4.5pt}] (279.24,100.55) .. controls (279.07,100.58) and (278.9,100.6) .. (278.72,100.6) .. controls (273.62,100.6) and (269.47,87.16) .. (269.47,70.6) .. controls (269.47,54.03) and (273.62,40.6) .. (278.72,40.6) .. controls (278.99,40.6) and (279.25,40.63) .. (279.5,40.7) -- (278.72,70.6) -- cycle ; \draw  [dash pattern={on 4.5pt off 4.5pt}] (279.24,100.55) .. controls (279.07,100.58) and (278.9,100.6) .. (278.72,100.6) .. controls (273.62,100.6) and (269.47,87.16) .. (269.47,70.6) .. controls (269.47,54.03) and (273.62,40.6) .. (278.72,40.6) .. controls (278.99,40.6) and (279.25,40.63) .. (279.5,40.7) ;  

\draw (39.52,92) node [anchor=north west][inner sep=0.75pt]    {$V$};
\draw (31.92,19.4) node [anchor=north west][inner sep=0.75pt]    {$f$};
\draw (47.92,15.8) node [anchor=north west][inner sep=0.75pt]    {$g$};
\draw (126.29,43.12) node [anchor=north west][inner sep=0.75pt]    {$f$};
\draw (58.72,31.4) node [anchor=north west][inner sep=0.75pt]    {$\cdots $};
\draw (127.79,64.07) node [anchor=north west][inner sep=0.75pt]    {$g$};
\draw (193.79,54.07) node [anchor=north west][inner sep=0.75pt]    {$fg$};
\draw (230.47,48.8) node [anchor=north west][inner sep=0.75pt]  [font=\scriptsize]  {$f\oplus \mathrm{id}$};
\draw (246.87,83.13) node [anchor=north west][inner sep=0.75pt]  [font=\scriptsize]  {$V\oplus W$};
\draw (291.47,66.47) node [anchor=north west][inner sep=0.75pt]  [font=\scriptsize]  {$W$};
\draw (322.13,97.47) node [anchor=north west][inner sep=0.75pt]  [font=\scriptsize]  {$V$};
\draw (370.47,47.8) node [anchor=north west][inner sep=0.75pt]  [font=\scriptsize]  {$f$};
\draw (15.2,110.8) node [anchor=north west][inner sep=0.75pt]  [font=\footnotesize]  {$L_{V}( f,g,\dotsc )$};
\draw (104.4,108.8) node [anchor=north west][inner sep=0.75pt]  [font=\footnotesize]  {$L_{V}( f,g) \sim L_{V}( fg)$};
\draw (239.2,122) node [anchor=north west][inner sep=0.75pt]  [font=\footnotesize]  {$L_{V\oplus W}( f\oplus \mathrm{id}) \sim L_{V}( f)$};

\end{tikzpicture}
\end{adjustbox}

\end{center}

The punctured-torus relation is a geometric expression of the fact that
$L([g,h])=0$. Equivalently, a representation of a once-punctured torus with
generator holonomies $g,h$ has boundary holonomy $[g,h]$. Its compatible
framed filling gives the same relation. Hence $L$ factors through
$\GL(F)^{\mathrm{ab}}$.

\begin{center}

\tikzset{every picture/.style={line width=0.75pt}} 

\begin{adjustbox}{max totalsize={\linewidth}{0.86\textheight}}
\begin{tikzpicture}[x=0.75pt,y=0.75pt,yscale=-1,xscale=1]

\draw [color={rgb, 255:red, 208; green, 2; blue, 27 }  ,draw opacity=1 ] [dash pattern={on 4.5pt off 4.5pt}]  (131.1,152.57) .. controls (141.83,153.46) and (159.73,123.81) .. (159.38,113.54) ;
\draw [color={rgb, 255:red, 208; green, 2; blue, 27 }  ,draw opacity=1 ]   (117.97,94.1) .. controls (136.03,92.77) and (156.35,98.53) .. (159.38,113.54) ;
\draw  [draw opacity=0] (101.26,73.56) .. controls (117.45,68.03) and (137.51,64.76) .. (159.22,64.76) .. controls (213,64.76) and (256.6,84.82) .. (256.6,109.57) .. controls (256.6,134.32) and (213,154.38) .. (159.22,154.38) .. controls (138.42,154.38) and (119.15,151.38) .. (103.33,146.27) -- (159.22,109.57) -- cycle ; \draw   (101.26,73.56) .. controls (117.45,68.03) and (137.51,64.76) .. (159.22,64.76) .. controls (213,64.76) and (256.6,84.82) .. (256.6,109.57) .. controls (256.6,134.32) and (213,154.38) .. (159.22,154.38) .. controls (138.42,154.38) and (119.15,151.38) .. (103.33,146.27) ;  
\draw  [draw opacity=0] (208.69,107.97) .. controls (201.16,114.54) and (190.98,118.57) .. (179.78,118.57) .. controls (169.05,118.57) and (159.26,114.87) .. (151.84,108.79) -- (179.78,79.99) -- cycle ; \draw   (208.69,107.97) .. controls (201.16,114.54) and (190.98,118.57) .. (179.78,118.57) .. controls (169.05,118.57) and (159.26,114.87) .. (151.84,108.79) ;  
\draw  [draw opacity=0] (162.61,111.95) .. controls (167.6,109.78) and (173.29,108.55) .. (179.33,108.55) .. controls (185.92,108.55) and (192.1,110.01) .. (197.39,112.57) -- (179.33,137.55) -- cycle ; \draw   (162.61,111.95) .. controls (167.6,109.78) and (173.29,108.55) .. (179.33,108.55) .. controls (185.92,108.55) and (192.1,110.01) .. (197.39,112.57) ;  
\draw [color={rgb, 255:red, 208; green, 2; blue, 27 }  ,draw opacity=1 ]   (111.92,139.85) .. controls (116.43,146.36) and (122.64,152.13) .. (131.1,152.57) ;
\draw  [draw opacity=0] (110.7,77.02) .. controls (123.53,74.09) and (139.12,72.37) .. (155.93,72.37) .. controls (199.8,72.37) and (235.35,84.08) .. (235.35,98.53) .. controls (235.35,112.99) and (199.8,124.7) .. (155.93,124.7) .. controls (142.52,124.7) and (129.89,123.61) .. (118.81,121.67) -- (155.93,98.53) -- cycle ; \draw  [color={rgb, 255:red, 126; green, 211; blue, 33 }  ,draw opacity=1 ] (110.7,77.02) .. controls (123.53,74.09) and (139.12,72.37) .. (155.93,72.37) .. controls (199.8,72.37) and (235.35,84.08) .. (235.35,98.53) .. controls (235.35,112.99) and (199.8,124.7) .. (155.93,124.7) .. controls (142.52,124.7) and (129.89,123.61) .. (118.81,121.67) ;  
\draw   (83.55,109.86) .. controls (83.55,90.03) and (91.73,73.95) .. (101.81,73.95) .. controls (111.9,73.95) and (120.07,90.03) .. (120.07,109.86) .. controls (120.07,129.69) and (111.9,145.77) .. (101.81,145.77) .. controls (91.73,145.77) and (83.55,129.69) .. (83.55,109.86) -- cycle ;

\draw (96.41,77.3) node [anchor=north west][inner sep=0.75pt]  [font=\footnotesize]  {$f$};
\draw (104.88,89.72) node [anchor=north west][inner sep=0.75pt]  [font=\footnotesize]  {$g$};
\draw (97.98,110.3) node [anchor=north west][inner sep=0.75pt]  [font=\footnotesize]  {$f^{-1}$};
\draw (95.41,127.58) node [anchor=north west][inner sep=0.75pt]  [font=\footnotesize]  {$g^{-1}$};
\draw (127,168.8) node [anchor=north west][inner sep=0.75pt]    {$L([ f,g]) \sim 0$};

\end{tikzpicture}
\end{adjustbox}

\end{center}

\begin{proposition}\label{prop:ex-K1}
With the preceding normalization,
\[
 L:\GL(F)^{\mathrm{ab}}\xrightarrow{\cong}\Xi_1^{\Loc_F},
 \qquad
 \Xi_1^{\Loc_F}\xrightarrow{\cong}F^\times,
 \quad L(g)\longmapsto\det(g).
\]
\end{proposition}
\begin{proof}
Under $\mathrm{PT}_1$, the normalized circle $L(g)$ represents the loop
associated with the automorphism $g$ in the group-completion model. Thus
$\mathrm{PT}_1\circ L$ is the standard isomorphism
$\GL(F)^{\mathrm{ab}}\cong K_1(F)$. For a field, elementary row operations
give
\[
 E(F)=[\GL(F),\GL(F)]=\Ker(\det),
 \qquad K_1(F)\cong F^\times
\]
\cite[III.1]{WeibelKbook}. Since $\mathrm{PT}_1$ is an isomorphism, so is
$L$. In particular, every class has a circle representative. This proof
uses the preceding Pontryagin--Thom theorem; the pictures explain the
representatives and relations rather than provide an independent proof
of that theorem.
\end{proof}

\subsection{\texorpdfstring{Lasagna, $K_2(F)$, Steinberg Relations, and Ghys Manifolds}{Lasagna, K2(F), Steinberg Relations, and Ghys Manifolds}}\label{subsec:k2}

The proposed geometric strategy is to simplify a closed $2$-lasagna by
surgeries, merge compatible layers, and obtain surface representatives.
To use this as a normal-form argument, each move must preserve the full
junction data and extend the chosen stable framing. We first record the
local moves that can be justified directly.

\begin{lemma}[Unraveling lemma]\label{lem:urvl}
Fix a rank and a boundary local system on $\partial D^2$, together with its
identifications. The space of extensions to an ordinary local system on
$D^2$ is either empty or contractible. In particular, any two such
extensions are homotopic relative to the specified boundary data.
\end{lemma}
\begin{proof}
A rank-$r$ local system is a map to $B\GL_r(F)$, a $K(\GL_r(F),1)$.
A boundary local system extends exactly when its monodromy is trivial.
After choosing one extension, the relative extension space has the
homotopy type of $\Omega^2B\GL_r(F)$, which is contractible.
\end{proof}
The underlying disk and its framing are fixed in this lemma. It does not
assert that arbitrary branched lasagna fillings with the same boundary
are homotopic; their difference can detect $K_2(F)$.

\begin{center}

\tikzset{every picture/.style={line width=0.75pt}} 

\begin{adjustbox}{max totalsize={\linewidth}{0.86\textheight}}
\begin{tikzpicture}[x=0.75pt,y=0.75pt,yscale=-1,xscale=1]

\draw  [fill={rgb, 255:red, 248; green, 231; blue, 28 }  ,fill opacity=1 ] (54,113.85) .. controls (54,83.94) and (78.24,59.7) .. (108.15,59.7) .. controls (138.06,59.7) and (162.3,83.94) .. (162.3,113.85) .. controls (162.3,143.76) and (138.06,168) .. (108.15,168) .. controls (78.24,168) and (54,143.76) .. (54,113.85) -- cycle ;
\draw    (60.67,88) .. controls (140,60.67) and (67.33,170.67) .. (157.33,135.33) ;
\draw    (59.33,139.33) .. controls (99.33,109.33) and (114.67,115.33) .. (154.67,85.33) ;
\draw   (77.33,96.67) .. controls (77.33,82.86) and (88.53,71.67) .. (102.33,71.67) .. controls (116.14,71.67) and (127.33,82.86) .. (127.33,96.67) .. controls (127.33,110.47) and (116.14,121.67) .. (102.33,121.67) .. controls (88.53,121.67) and (77.33,110.47) .. (77.33,96.67) -- cycle ;
\draw  [fill={rgb, 255:red, 248; green, 231; blue, 28 }  ,fill opacity=1 ] (212.67,112.52) .. controls (212.67,82.61) and (236.91,58.37) .. (266.82,58.37) .. controls (296.72,58.37) and (320.97,82.61) .. (320.97,112.52) .. controls (320.97,142.42) and (296.72,166.67) .. (266.82,166.67) .. controls (236.91,166.67) and (212.67,142.42) .. (212.67,112.52) -- cycle ;
\draw    (221.33,83) -- (279.33,120) ;
\draw    (225.33,145.67) -- (279.33,120) ;
\draw    (248,121.33) -- (310,144.67) ;
\draw    (248,121.33) -- (311.33,81.33) ;
\draw    (263.33,110) -- (264,126.67) ;
\draw  [fill={rgb, 255:red, 248; green, 231; blue, 28 }  ,fill opacity=1 ] (75.77,228.27) .. controls (75.77,200.2) and (98.53,177.45) .. (126.6,177.45) .. controls (154.67,177.45) and (177.43,200.2) .. (177.43,228.27) .. controls (177.43,256.34) and (154.67,279.1) .. (126.6,279.1) .. controls (98.53,279.1) and (75.77,256.34) .. (75.77,228.27) -- cycle ;
\draw  [fill={rgb, 255:red, 248; green, 231; blue, 28 }  ,fill opacity=1 ] (208.16,226.99) .. controls (208.16,198.92) and (230.92,176.17) .. (258.99,176.17) .. controls (287.06,176.17) and (309.82,198.92) .. (309.82,226.99) .. controls (309.82,255.06) and (287.06,277.82) .. (258.99,277.82) .. controls (230.92,277.82) and (208.16,255.06) .. (208.16,226.99) -- cycle ;
\draw [color={rgb, 255:red, 74; green, 144; blue, 226 }  ,draw opacity=1 ]   (218.09,197.3) -- (300.25,197.77) ;
\draw [shift={(302.25,197.78)}, rotate = 180.32] [color={rgb, 255:red, 74; green, 144; blue, 226 }  ,draw opacity=1 ][line width=0.75]    (10.93,-3.29) .. controls (6.95,-1.4) and (3.31,-0.3) .. (0,0) .. controls (3.31,0.3) and (6.95,1.4) .. (10.93,3.29)   ;
\draw [color={rgb, 255:red, 208; green, 2; blue, 27 }  ,draw opacity=1 ]   (210.52,217.16) -- (304.98,218.55) ;
\draw [shift={(306.98,218.58)}, rotate = 180.84] [color={rgb, 255:red, 208; green, 2; blue, 27 }  ,draw opacity=1 ][line width=0.75]    (10.93,-3.29) .. controls (6.95,-1.4) and (3.31,-0.3) .. (0,0) .. controls (3.31,0.3) and (6.95,1.4) .. (10.93,3.29)   ;
\draw [color={rgb, 255:red, 74; green, 144; blue, 226 }  ,draw opacity=1 ]   (211.67,239.04) -- (307.88,239.33) ;
\draw [shift={(209.67,239.03)}, rotate = 0.17] [color={rgb, 255:red, 74; green, 144; blue, 226 }  ,draw opacity=1 ][line width=0.75]    (10.93,-3.29) .. controls (6.95,-1.4) and (3.31,-0.3) .. (0,0) .. controls (3.31,0.3) and (6.95,1.4) .. (10.93,3.29)   ;
\draw [color={rgb, 255:red, 208; green, 2; blue, 27 }  ,draw opacity=1 ]   (220.09,256.91) -- (298.47,258.3) ;
\draw [shift={(218.09,256.88)}, rotate = 1.01] [color={rgb, 255:red, 208; green, 2; blue, 27 }  ,draw opacity=1 ][line width=0.75]    (10.93,-3.29) .. controls (6.95,-1.4) and (3.31,-0.3) .. (0,0) .. controls (3.31,0.3) and (6.95,1.4) .. (10.93,3.29)   ;
\draw [color={rgb, 255:red, 74; green, 144; blue, 226 }  ,draw opacity=1 ]   (80.67,206.33) .. controls (124.89,206.37) and (123.7,231.2) .. (79.03,239.13) ;
\draw [shift={(77.67,239.37)}, rotate = 350.54] [color={rgb, 255:red, 74; green, 144; blue, 226 }  ,draw opacity=1 ][line width=0.75]    (10.93,-3.29) .. controls (6.95,-1.4) and (3.31,-0.3) .. (0,0) .. controls (3.31,0.3) and (6.95,1.4) .. (10.93,3.29)   ;
\draw [color={rgb, 255:red, 74; green, 144; blue, 226 }  ,draw opacity=1 ]   (175.67,236.37) .. controls (131.48,234.57) and (124.68,210.62) .. (169.44,204.49) ;
\draw [shift={(170.82,204.31)}, rotate = 172.83] [color={rgb, 255:red, 74; green, 144; blue, 226 }  ,draw opacity=1 ][line width=0.75]    (10.93,-3.29) .. controls (6.95,-1.4) and (3.31,-0.3) .. (0,0) .. controls (3.31,0.3) and (6.95,1.4) .. (10.93,3.29)   ;
\draw [color={rgb, 255:red, 208; green, 2; blue, 27 }  ,draw opacity=1 ]   (75.77,228.27) .. controls (119.99,228.31) and (131.44,250.58) .. (87.03,258.47) ;
\draw [shift={(85.67,258.7)}, rotate = 350.54] [color={rgb, 255:red, 208; green, 2; blue, 27 }  ,draw opacity=1 ][line width=0.75]    (10.93,-3.29) .. controls (6.95,-1.4) and (3.31,-0.3) .. (0,0) .. controls (3.31,0.3) and (6.95,1.4) .. (10.93,3.29)   ;
\draw [color={rgb, 255:red, 208; green, 2; blue, 27 }  ,draw opacity=1 ]   (170,255.7) .. controls (125.82,253.9) and (131.1,230.93) .. (176.11,224.82) ;
\draw [shift={(177.48,224.64)}, rotate = 172.83] [color={rgb, 255:red, 208; green, 2; blue, 27 }  ,draw opacity=1 ][line width=0.75]    (10.93,-3.29) .. controls (6.95,-1.4) and (3.31,-0.3) .. (0,0) .. controls (3.31,0.3) and (6.95,1.4) .. (10.93,3.29)   ;

\draw (40.67,69.07) node [anchor=north west][inner sep=0.75pt]    {$a_{1}$};
\draw (160.67,69.73) node [anchor=north west][inner sep=0.75pt]    {$a_{2}$};
\draw (163.33,134.4) node [anchor=north west][inner sep=0.75pt]    {$a_{3}$};
\draw (41.33,141.73) node [anchor=north west][inner sep=0.75pt]    {$a_{4}$};
\draw (175.33,101.07) node [anchor=north west][inner sep=0.75pt]    {$\sim $};
\draw (199.33,67.73) node [anchor=north west][inner sep=0.75pt]    {$a_{1}$};
\draw (319.33,68.4) node [anchor=north west][inner sep=0.75pt]    {$a_{2}$};
\draw (322,133.07) node [anchor=north west][inner sep=0.75pt]    {$a_{3}$};
\draw (200,140.4) node [anchor=north west][inner sep=0.75pt]    {$a_{4}$};
\draw (182.77,218.09) node [anchor=north west][inner sep=0.75pt]    {$\sim $};

\end{tikzpicture}
\end{adjustbox}

\end{center}

\begin{lemma}[Image Candy \footnote{also called Cut Rock Candy} Surgery]\label{lem:imgS}
Let $A$ be a framed decorated surface, and let $D_1,D_2$ be disjoint disk
regions on which the decoration is an ordinary local system. Choose an
identification of their disk decorations and an attaching framing that
extends over a $3$-dimensional $1$-handle. Replacing the two disks by the
annulus $D^1\times S^1$ gives a decorated surface $B$ with $A\sim B$.
In a string presentation, oppositely oriented attaching circles have
words $(a_1,\ldots,a_q)$ and $(a_q^{-1},\ldots,a_1^{-1})$.
\end{lemma}
\begin{proof}
Attach $D^1\times D^2$ to $A\times I$ along $D_1,D_2$ in its upper
boundary. The chosen identification extends the local system over the
handle, since the attaching disks and the handle are contractible.
Lemma~\ref{lem:urvl} supplies the relative compatibility for alternative
disk presentations. The incoming and outgoing boundaries of this surgery
trace are $A$ and $B$; the chosen attaching framing makes it a framed
bordism. 
\end{proof}

\begin{center}

\tikzset{every picture/.style={line width=0.75pt}} 

\begin{adjustbox}{max totalsize={\linewidth}{0.86\textheight}}
\begin{tikzpicture}[x=0.75pt,y=0.75pt,yscale=-1,xscale=1]

\draw  [fill={rgb, 255:red, 248; green, 231; blue, 28 }  ,fill opacity=1 ] (69,120.84) .. controls (69,86.49) and (113.84,58.65) .. (169.15,58.65) .. controls (224.46,58.65) and (269.3,86.49) .. (269.3,120.84) .. controls (269.3,155.19) and (224.46,183.03) .. (169.15,183.03) .. controls (113.84,183.03) and (69,155.19) .. (69,120.84) -- cycle ;
\draw  [dash pattern={on 0.84pt off 2.51pt}] (97.8,119.65) .. controls (97.8,108.05) and (107.2,98.65) .. (118.8,98.65) .. controls (130.4,98.65) and (139.8,108.05) .. (139.8,119.65) .. controls (139.8,131.25) and (130.4,140.65) .. (118.8,140.65) .. controls (107.2,140.65) and (97.8,131.25) .. (97.8,119.65) -- cycle ;
\draw  [dash pattern={on 0.84pt off 2.51pt}] (190.3,116.65) .. controls (190.3,105.05) and (199.7,95.65) .. (211.3,95.65) .. controls (222.9,95.65) and (232.3,105.05) .. (232.3,116.65) .. controls (232.3,128.25) and (222.9,137.65) .. (211.3,137.65) .. controls (199.7,137.65) and (190.3,128.25) .. (190.3,116.65) -- cycle ;
\draw [color={rgb, 255:red, 144; green, 19; blue, 254 }  ,draw opacity=1 ]   (118.8,119.65) -- (129.73,82.57) ;
\draw [shift={(130.3,80.65)}, rotate = 106.43] [color={rgb, 255:red, 144; green, 19; blue, 254 }  ,draw opacity=1 ][line width=0.75]    (10.93,-3.29) .. controls (6.95,-1.4) and (3.31,-0.3) .. (0,0) .. controls (3.31,0.3) and (6.95,1.4) .. (10.93,3.29)   ;
\draw [color={rgb, 255:red, 208; green, 2; blue, 27 }  ,draw opacity=1 ]   (86.3,136.15) -- (117.02,120.56) ;
\draw [shift={(118.8,119.65)}, rotate = 153.08] [color={rgb, 255:red, 208; green, 2; blue, 27 }  ,draw opacity=1 ][line width=0.75]    (10.93,-3.29) .. controls (6.95,-1.4) and (3.31,-0.3) .. (0,0) .. controls (3.31,0.3) and (6.95,1.4) .. (10.93,3.29)   ;
\draw [color={rgb, 255:red, 74; green, 144; blue, 226 }  ,draw opacity=1 ]   (144.8,141.15) -- (120.34,120.92) ;
\draw [shift={(118.8,119.65)}, rotate = 39.59] [color={rgb, 255:red, 74; green, 144; blue, 226 }  ,draw opacity=1 ][line width=0.75]    (10.93,-3.29) .. controls (6.95,-1.4) and (3.31,-0.3) .. (0,0) .. controls (3.31,0.3) and (6.95,1.4) .. (10.93,3.29)   ;
\draw [color={rgb, 255:red, 144; green, 19; blue, 254 }  ,draw opacity=1 ]   (222.3,79.65) -- (211.87,114.73) ;
\draw [shift={(211.3,116.65)}, rotate = 286.56] [color={rgb, 255:red, 144; green, 19; blue, 254 }  ,draw opacity=1 ][line width=0.75]    (10.93,-3.29) .. controls (6.95,-1.4) and (3.31,-0.3) .. (0,0) .. controls (3.31,0.3) and (6.95,1.4) .. (10.93,3.29)   ;
\draw [color={rgb, 255:red, 208; green, 2; blue, 27 }  ,draw opacity=1 ]   (211.3,116.65) -- (232.89,138.24) ;
\draw [shift={(234.3,139.65)}, rotate = 225] [color={rgb, 255:red, 208; green, 2; blue, 27 }  ,draw opacity=1 ][line width=0.75]    (10.93,-3.29) .. controls (6.95,-1.4) and (3.31,-0.3) .. (0,0) .. controls (3.31,0.3) and (6.95,1.4) .. (10.93,3.29)   ;
\draw [color={rgb, 255:red, 74; green, 144; blue, 226 }  ,draw opacity=1 ]   (212.8,116.15) -- (188.36,135.89) ;
\draw [shift={(186.8,137.15)}, rotate = 321.07] [color={rgb, 255:red, 74; green, 144; blue, 226 }  ,draw opacity=1 ][line width=0.75]    (10.93,-3.29) .. controls (6.95,-1.4) and (3.31,-0.3) .. (0,0) .. controls (3.31,0.3) and (6.95,1.4) .. (10.93,3.29)   ;
\draw  [fill={rgb, 255:red, 248; green, 231; blue, 28 }  ,fill opacity=1 ] (294,116.84) .. controls (294,79.73) and (338.84,49.65) .. (394.15,49.65) .. controls (449.46,49.65) and (494.3,79.73) .. (494.3,116.84) .. controls (494.3,153.95) and (449.46,184.03) .. (394.15,184.03) .. controls (338.84,184.03) and (294,153.95) .. (294,116.84) -- cycle ;
\draw  [draw opacity=0][dash pattern={on 0.84pt off 2.51pt}] (322.83,121.75) .. controls (322.81,121.39) and (322.8,121.02) .. (322.8,120.65) .. controls (322.8,109.05) and (332.2,99.65) .. (343.8,99.65) .. controls (354.83,99.65) and (363.87,108.15) .. (364.73,118.95) -- (343.8,120.65) -- cycle ; \draw  [dash pattern={on 0.84pt off 2.51pt}] (322.83,121.75) .. controls (322.81,121.39) and (322.8,121.02) .. (322.8,120.65) .. controls (322.8,109.05) and (332.2,99.65) .. (343.8,99.65) .. controls (354.83,99.65) and (363.87,108.15) .. (364.73,118.95) ;  
\draw  [draw opacity=0][dash pattern={on 0.84pt off 2.51pt}] (415.31,117.17) .. controls (415.56,105.8) and (424.86,96.65) .. (436.3,96.65) .. controls (447.9,96.65) and (457.3,106.05) .. (457.3,117.65) .. controls (457.3,118.6) and (457.24,119.54) .. (457.11,120.46) -- (436.3,117.65) -- cycle ; \draw  [dash pattern={on 0.84pt off 2.51pt}] (415.31,117.17) .. controls (415.56,105.8) and (424.86,96.65) .. (436.3,96.65) .. controls (447.9,96.65) and (457.3,106.05) .. (457.3,117.65) .. controls (457.3,118.6) and (457.24,119.54) .. (457.11,120.46) ;  
\draw [color={rgb, 255:red, 144; green, 19; blue, 254 }  ,draw opacity=1 ] [dash pattern={on 0.84pt off 2.51pt}]  (345.8,98.15) -- (353.1,78.52) ;
\draw [shift={(353.8,76.65)}, rotate = 110.41] [color={rgb, 255:red, 144; green, 19; blue, 254 }  ,draw opacity=1 ][line width=0.75]    (10.93,-3.29) .. controls (6.95,-1.4) and (3.31,-0.3) .. (0,0) .. controls (3.31,0.3) and (6.95,1.4) .. (10.93,3.29)   ;
\draw [color={rgb, 255:red, 208; green, 2; blue, 27 }  ,draw opacity=1 ]   (307.8,142.65) -- (324.11,132.23) ;
\draw [shift={(325.8,131.15)}, rotate = 147.43] [color={rgb, 255:red, 208; green, 2; blue, 27 }  ,draw opacity=1 ][line width=0.75]    (10.93,-3.29) .. controls (6.95,-1.4) and (3.31,-0.3) .. (0,0) .. controls (3.31,0.3) and (6.95,1.4) .. (10.93,3.29)   ;
\draw [color={rgb, 255:red, 74; green, 144; blue, 226 }  ,draw opacity=1 ]   (379.8,147.15) -- (362.81,132.46) ;
\draw [shift={(361.3,131.15)}, rotate = 40.86] [color={rgb, 255:red, 74; green, 144; blue, 226 }  ,draw opacity=1 ][line width=0.75]    (10.93,-3.29) .. controls (6.95,-1.4) and (3.31,-0.3) .. (0,0) .. controls (3.31,0.3) and (6.95,1.4) .. (10.93,3.29)   ;
\draw [color={rgb, 255:red, 144; green, 19; blue, 254 }  ,draw opacity=1 ] [dash pattern={on 0.84pt off 2.51pt}]  (434.8,81.15) -- (432.62,94.68) ;
\draw [shift={(432.3,96.65)}, rotate = 279.16] [color={rgb, 255:red, 144; green, 19; blue, 254 }  ,draw opacity=1 ][line width=0.75]    (10.93,-3.29) .. controls (6.95,-1.4) and (3.31,-0.3) .. (0,0) .. controls (3.31,0.3) and (6.95,1.4) .. (10.93,3.29)   ;
\draw [color={rgb, 255:red, 208; green, 2; blue, 27 }  ,draw opacity=1 ]   (455.8,127.65) -- (475.05,138.19) ;
\draw [shift={(476.8,139.15)}, rotate = 208.71] [color={rgb, 255:red, 208; green, 2; blue, 27 }  ,draw opacity=1 ][line width=0.75]    (10.93,-3.29) .. controls (6.95,-1.4) and (3.31,-0.3) .. (0,0) .. controls (3.31,0.3) and (6.95,1.4) .. (10.93,3.29)   ;
\draw [color={rgb, 255:red, 74; green, 144; blue, 226 }  ,draw opacity=1 ]   (422.3,134.15) -- (408.76,146.79) ;
\draw [shift={(407.3,148.15)}, rotate = 316.97] [color={rgb, 255:red, 74; green, 144; blue, 226 }  ,draw opacity=1 ][line width=0.75]    (10.93,-3.29) .. controls (6.95,-1.4) and (3.31,-0.3) .. (0,0) .. controls (3.31,0.3) and (6.95,1.4) .. (10.93,3.29)   ;
\draw  [draw opacity=0] (322.35,121.54) .. controls (321.01,118.3) and (320.3,114.91) .. (320.3,111.41) .. controls (320.3,88.35) and (351.3,69.65) .. (389.55,69.65) .. controls (427.8,69.65) and (458.8,88.35) .. (458.8,111.41) .. controls (458.8,114.05) and (458.39,116.63) .. (457.62,119.13) -- (389.55,111.41) -- cycle ; \draw   (322.35,121.54) .. controls (321.01,118.3) and (320.3,114.91) .. (320.3,111.41) .. controls (320.3,88.35) and (351.3,69.65) .. (389.55,69.65) .. controls (427.8,69.65) and (458.8,88.35) .. (458.8,111.41) .. controls (458.8,114.05) and (458.39,116.63) .. (457.62,119.13) ;  
\draw  [draw opacity=0] (364.32,118.71) .. controls (364.33,119.04) and (364.34,119.37) .. (364.34,119.71) .. controls (364.34,131.3) and (354.94,140.71) .. (343.34,140.71) .. controls (332.08,140.71) and (322.9,131.85) .. (322.36,120.73) -- (343.34,119.71) -- cycle ; \draw   (364.32,118.71) .. controls (364.33,119.04) and (364.34,119.37) .. (364.34,119.71) .. controls (364.34,131.3) and (354.94,140.71) .. (343.34,140.71) .. controls (332.08,140.71) and (322.9,131.85) .. (322.36,120.73) ;  
\draw  [draw opacity=0] (457.28,116.65) .. controls (457.29,116.98) and (457.3,117.31) .. (457.3,117.65) .. controls (457.3,129.25) and (447.9,138.65) .. (436.3,138.65) .. controls (425.04,138.65) and (415.86,129.8) .. (415.32,118.67) -- (436.3,117.65) -- cycle ; \draw   (457.28,116.65) .. controls (457.29,116.98) and (457.3,117.31) .. (457.3,117.65) .. controls (457.3,129.25) and (447.9,138.65) .. (436.3,138.65) .. controls (425.04,138.65) and (415.86,129.8) .. (415.32,118.67) ;  
\draw  [draw opacity=0] (364.66,120.52) .. controls (366.51,110.72) and (377.25,103.19) .. (390.22,103.19) .. controls (402.36,103.19) and (412.55,109.79) .. (415.32,118.67) -- (390.22,123.55) -- cycle ; \draw   (364.66,120.52) .. controls (366.51,110.72) and (377.25,103.19) .. (390.22,103.19) .. controls (402.36,103.19) and (412.55,109.79) .. (415.32,118.67) ;  
\draw [color={rgb, 255:red, 208; green, 2; blue, 27 }  ,draw opacity=1 ]   (325.8,131.15) .. controls (344.3,117.28) and (359.9,94.83) .. (379.3,88.15) .. controls (398.41,81.57) and (433.72,89.89) .. (454.84,125.98) ;
\draw [shift={(455.8,127.65)}, rotate = 240.75] [color={rgb, 255:red, 208; green, 2; blue, 27 }  ,draw opacity=1 ][line width=0.75]    (10.93,-3.29) .. controls (6.95,-1.4) and (3.31,-0.3) .. (0,0) .. controls (3.31,0.3) and (6.95,1.4) .. (10.93,3.29)   ;
\draw [color={rgb, 255:red, 74; green, 144; blue, 226 }  ,draw opacity=1 ]   (361.3,131.15) .. controls (354.33,117.21) and (369.8,93.65) .. (390.8,94.65) .. controls (411.27,95.62) and (425.1,108.01) .. (422.53,132.26) ;
\draw [shift={(422.3,134.15)}, rotate = 277.82] [color={rgb, 255:red, 74; green, 144; blue, 226 }  ,draw opacity=1 ][line width=0.75]    (10.93,-3.29) .. controls (6.95,-1.4) and (3.31,-0.3) .. (0,0) .. controls (3.31,0.3) and (6.95,1.4) .. (10.93,3.29)   ;
\draw [color={rgb, 255:red, 144; green, 19; blue, 254 }  ,draw opacity=1 ] [dash pattern={on 0.84pt off 2.51pt}]  (432.3,96.65) .. controls (420.91,79.82) and (379.64,67.89) .. (346.79,97.25) ;
\draw [shift={(345.8,98.15)}, rotate = 317.25] [color={rgb, 255:red, 144; green, 19; blue, 254 }  ,draw opacity=1 ][line width=0.75]    (10.93,-3.29) .. controls (6.95,-1.4) and (3.31,-0.3) .. (0,0) .. controls (3.31,0.3) and (6.95,1.4) .. (10.93,3.29)   ;

\draw (108.67,81.07) node [anchor=north west][inner sep=0.75pt]    {$a_{1}$};
\draw (139.17,138.73) node [anchor=north west][inner sep=0.75pt]    {$a_{2}$};
\draw (90.83,136.4) node [anchor=north west][inner sep=0.75pt]    {$a_{3}$};
\draw (196.17,74.07) node [anchor=north west][inner sep=0.75pt]    {$a_{1}^{-1}$};
\draw (184.67,137.23) node [anchor=north west][inner sep=0.75pt]    {$a_{2}^{-1}$};
\draw (228.83,138.4) node [anchor=north west][inner sep=0.75pt]    {$a_{3}^{-1}$};
\draw (342.17,59.07) node [anchor=north west][inner sep=0.75pt]    {$a_{1}$};
\draw (364.17,139.73) node [anchor=north west][inner sep=0.75pt]    {$a_{2}$};
\draw (315.83,137.4) node [anchor=north west][inner sep=0.75pt]    {$a_{3}$};
\draw (426.67,59.07) node [anchor=north west][inner sep=0.75pt]    {$a_{1}^{-1}$};
\draw (410.67,146.23) node [anchor=north west][inner sep=0.75pt]    {$a_{2}^{-1}$};
\draw (453.83,139.4) node [anchor=north west][inner sep=0.75pt]    {$a_{3}^{-1}$};
\draw (275.3,111.9) node [anchor=north west][inner sep=0.75pt]    {$\sim $};

\end{tikzpicture}
\end{adjustbox}

\end{center}

\paragraph{Scalar local systems and determinant normalization.}
An $F^\times$-local system on a closed oriented surface $S$ is a class in
$H^1(S;F^\times)$.

\begin{proposition}
Every rank-one local system on $S$ can be represented by finitely many
oriented curves $C_i$, labeled by $u_i\in F^\times$, such that
\[
 \rho(\gamma)=\prod_i u_i^{\langle\gamma,C_i\rangle}.
\]
The curves can be chosen in general position, but need not be pairwise
disjoint.

Furthermore, each such local system is bordant to an $E(F)$-local system.
\end{proposition}
\begin{proof}
Choose an integral basis of $H_1(S;\bZ)$ and curves representing its
intersection-dual basis. The character $H_1(S;\bZ)\to F^\times$ is
specified by its values on that basis, which gives the formula. At a
crossing, the two transition functions commute, so the local cocycle
condition holds.

For a single oriented curve $C_i$ labeled by $u_i$, the corresponding
local system $L_i$ extends over a handlebody by
Lemma~\ref{lem:ex-handlebody}. Hence

\[
  L\sim L \oplus \bigoplus L_i \in \Loc_{E(F)}
\]
\end{proof}
Commutativity allows consistent crossings; it does not show that all
crossings can be removed through local-system homotopies. In particular,
it does not imply that every scalar-decorated surface bounds.

\begin{proposition}[Determinant normalization]\label{lem:EF}
For a local system $L$ of rank $r$ on $S$, the local system
\[
 L^0:=L\oplus(\det L)^{-1}
\]
has canonically trivial determinant and monodromy in
$\mathrm{SL}_{r+1}(F)\subset E(F)$. In the lasagna bordism group,
\begin{equation}\label{eq:ex-det-correction}
 [S,L^0]=[S,L]+[S,(\det L)^{-1}].
\end{equation}
\end{proposition}
\begin{proof}
For each loop $\gamma$, the determinant of the block monodromy is
$\det\rho(\gamma)\det\rho(\gamma)^{-1}=1$. The canonical trivialization
is the evaluation
$\det L\otimes(\det L)^{-1}\cong\underline F$. Equation
\eqref{eq:ex-det-correction} is the direct-sum relation.
\end{proof}
Since $\det L$ is also bordant to a local system in $\Loc_{E(F)}$, this
is the algebraic step in the proposed reduction to $E(F)$-local systems.
Thus $ \Xi_2^{\Loc_{E(F)}}\cong \Xi_2^{\Loc_{F}}$.

\paragraph{Punctured surfaces and the Steinberg group.}
Let $S$ have one parameterized boundary circle with a base point $b$.
Choose a basis in the fiber over $b$, and stabilize these bases compatibly.
The boundary holonomy is then an element of $\GL(F)$, rather than merely
a conjugacy class. For an oriented surface of genus $g$, a compatible
choice of generators gives
\[
 \rho(\partial S)=\prod_{i=1}^g[\rho(a_i),\rho(b_i)]\in E(F),
 \qquad [x,y]:=xyx^{-1}y^{-1}.
\]
The proposed relative group $\Xi_2^{\Loc_F,S^1,b}$ uses based
boundary-connected sum, with the parameterizations, basis, and collar
framing retained. Boundary holonomy is intended to give
\[
 \varphi:\Xi_2^{\Loc_F,S^1,b}\longrightarrow E(F).
\]
Write $\Xi_2^{\Loc_{E(F)},S^1,b}$ for the corresponding presentation
restricted to $E(F)$-local systems. By the preceding argument, every local
system is bordant to an $E(F)$-local system, giving
$\Xi_2^{\Loc_{E(F)},S^1,b}\cong \Xi_2^{\Loc_{F},S^1,b}$. The reduced model
is intended to encode the following standard algebraic group.

\begin{proposition}[The Hopf--Steinberg model]\label{prop:ex-Hopf}
Let $E(F)=P/R$, where $P$ is free and $R$ is the kernel. Then
\[
 \St(F)\cong\frac{[P,P]}{[R,P]},
 \qquad
 K_2(F)\cong\frac{R\cap[P,P]}{[R,P]}.
\]
The first quotient is the universal central extension of $E(F)$.
\end{proposition}
\begin{proof}
In $V=P/[R,P]$, the image $\overline R$ is central. Since $E(F)$ is
perfect, $P=R[P,P]$, so $V=\overline R\,V'$. Hence
$V'=[V',V']$ is perfect. It maps onto $E(F)$ with central kernel
$(R\cap[P,P])/[R,P]$. Every central extension of $E(F)$ admits a lift
from the free group $P$; this lift kills $[R,P]$. Its restriction to the
perfect group $V'$ is unique, since two lifts differ by a homomorphism
to an abelian central kernel. Thus $V'$ is the universal central
extension. The stable Steinberg extension has this universal property
\cite[III.5.3--5.5]{WeibelKbook}.
\end{proof}

\begin{proposition}
We have
\[
 \Xi_2^{\Loc_{F},S^1,b}\cong\Xi_2^{\Loc_{E(F)},S^1,b}\cong\St(F),
\]
and boundary holonomy is the canonical map $\varphi:\St(F)\to E(F)$.
\end{proposition}
\begin{proof}
By the preceding argument, it suffices to consider $E(F)$-local systems.
We must show that the decorated-surface presentation has the following
properties:
\begin{enumerate}
 \item cutting to a polygon and reading its strings gives a word in
 $[P,P]$, compatibly with boundary-connected sum, and every such word is
 realized; 
 \item its allowed relative bordisms identify exactly the words whose
 quotient lies in $[R,P]$.
\end{enumerate}
These two properties identify the geometric presentation with
$[P,P]/[R,P]$. Apply Proposition~\ref{prop:ex-Hopf}.

For the first property, move the crossings into the boundary collar; the
remaining crossing-free strings correspond to $ [P,P]$. For the second,
the relations in $[R,P]$ are realized by Image Candy Surgery. This gives
the surjection
$ \Xi_2^{\Loc_{E(F)},S^1,b} \twoheadrightarrow \St(F)$.
Injectivity follows from the generalized Pontryagin--Thom theorem
(Theorem~\ref{thm:gPT}).
\end{proof}

Forgetting the boundary base point gives conjugation in the word model.
At the level of orbit \emph{sets}, conjugation by $\St(F)$ and by $E(F)$
has the same orbits, since $K_2(F)$ is central. Therefore
$\Xi_2^{\Loc_{F},S^1}\cong\Xi_2^{\Loc_{E(F)},S^1}$. Closing a boundary
with holonomy $1$ also requires a chosen boundary trivialization and a
framing that extends over the capping disk. This gives
$ \Xi_2^{\Loc_{F}}\cong \Xi_2^{\Loc_{E(F)},S^1}\cong \Ker \varphi $.

\paragraph{Commuting tori and Steinberg symbols.}
For commuting $f,g\in E(F)$, choose lifts $\widetilde f,\widetilde g$ in
$\St(F)$ and define
\begin{equation}\label{eq:ex-commuting-torus}
 C(f,g):=[\widetilde f,\widetilde g]\in K_2(F).
\end{equation}
Centrality of the kernel makes this independent of the lifts. It is the
image of the oriented torus class in $H_2(E(F);\bZ)$ under
$H_2(E(F);\bZ)\cong K_2(F)$ \cite[III.5]{WeibelKbook}. Geometrically, use
the torus with holonomies $f,g$, product bounding circle framings, and the usual normalized Pontryagin--Thom comparison.

For arbitrary commuting $f,g\in\GL(F)$, there is again a decorated torus.
By the preceding reduction, it has an $E(F)$-valued representative of the
same bordism class.

\begin{center}

\tikzset{every picture/.style={line width=0.75pt}} 

\begin{adjustbox}{max totalsize={\linewidth}{0.86\textheight}}
\begin{tikzpicture}[x=0.75pt,y=0.75pt,yscale=-1,xscale=1]

\draw  [draw opacity=0] (199.81,157.95) .. controls (193.54,166.8) and (184.07,172.43) .. (173.48,172.43) .. controls (163.27,172.43) and (154.1,167.19) .. (147.84,158.89) -- (173.48,132.46) -- cycle ; \draw   (199.81,157.95) .. controls (193.54,166.8) and (184.07,172.43) .. (173.48,172.43) .. controls (163.27,172.43) and (154.1,167.19) .. (147.84,158.89) ;  
\draw  [draw opacity=0] (156.87,167.22) .. controls (161.5,163.95) and (167.09,162.05) .. (173.11,162.05) .. controls (179.63,162.05) and (185.64,164.29) .. (190.48,168.06) -- (173.11,192.09) -- cycle ; \draw   (156.87,167.22) .. controls (161.5,163.95) and (167.09,162.05) .. (173.11,162.05) .. controls (179.63,162.05) and (185.64,164.29) .. (190.48,168.06) ;  
\draw   (93.42,166.15) .. controls (93.42,140.39) and (128.39,119.52) .. (171.52,119.52) .. controls (214.66,119.52) and (249.62,140.39) .. (249.62,166.15) .. controls (249.62,191.9) and (214.66,212.78) .. (171.52,212.78) .. controls (128.39,212.78) and (93.42,191.9) .. (93.42,166.15) -- cycle ;
\draw  [color={rgb, 255:red, 184; green, 233; blue, 134 }  ,draw opacity=1 ] (114.43,161.76) .. controls (114.43,147.23) and (140.45,135.46) .. (172.55,135.46) .. controls (204.65,135.46) and (230.66,147.23) .. (230.66,161.76) .. controls (230.66,176.28) and (204.65,188.06) .. (172.55,188.06) .. controls (140.45,188.06) and (114.43,176.28) .. (114.43,161.76) -- cycle ;
\draw  [draw opacity=0] (152.3,209.95) .. controls (146.23,208.2) and (141.61,199.29) .. (141.61,188.55) .. controls (141.61,176.56) and (147.37,166.84) .. (154.47,166.84) .. controls (155.29,166.84) and (156.1,166.97) .. (156.87,167.22) -- (154.47,188.55) -- cycle ; \draw  [color={rgb, 255:red, 208; green, 2; blue, 27 }  ,draw opacity=1 ] (152.3,209.95) .. controls (146.23,208.2) and (141.61,199.29) .. (141.61,188.55) .. controls (141.61,176.56) and (147.37,166.84) .. (154.47,166.84) .. controls (155.29,166.84) and (156.1,166.97) .. (156.87,167.22) ;  
\draw  [draw opacity=0][dash pattern={on 4.5pt off 4.5pt}] (163.47,173.03) .. controls (165.86,176.97) and (167.34,182.47) .. (167.34,188.55) .. controls (167.34,199.71) and (162.35,208.9) .. (155.93,210.12) -- (154.47,188.55) -- cycle ; \draw  [color={rgb, 255:red, 208; green, 2; blue, 27 }  ,draw opacity=1 ][dash pattern={on 4.5pt off 4.5pt}] (163.47,173.03) .. controls (165.86,176.97) and (167.34,182.47) .. (167.34,188.55) .. controls (167.34,199.71) and (162.35,208.9) .. (155.93,210.12) ;  
\draw    (155.57,193.65) .. controls (194.23,221.39) and (241.59,210.71) .. (283.69,183.6) ;
\draw [shift={(284.97,182.78)}, rotate = 146.84] [color={rgb, 255:red, 0; green, 0; blue, 0 }  ][line width=0.75]    (10.93,-3.29) .. controls (6.95,-1.4) and (3.31,-0.3) .. (0,0) .. controls (3.31,0.3) and (6.95,1.4) .. (10.93,3.29)   ;
\draw    (295.44,149.7) -- (296.29,175.27) ;
\draw [shift={(296.36,177.26)}, rotate = 268.09] [color={rgb, 255:red, 0; green, 0; blue, 0 }  ][line width=0.75]    (4.37,-1.96) .. controls (2.78,-0.92) and (1.32,-0.27) .. (0,0) .. controls (1.32,0.27) and (2.78,0.92) .. (4.37,1.96)   ;
\draw    (350.57,149.7) -- (350.57,175.26) ;
\draw [shift={(350.57,177.26)}, rotate = 270] [color={rgb, 255:red, 0; green, 0; blue, 0 }  ][line width=0.75]    (4.37,-1.96) .. controls (2.78,-0.92) and (1.32,-0.27) .. (0,0) .. controls (1.32,0.27) and (2.78,0.92) .. (4.37,1.96)   ;
\draw    (309.22,191.05) -- (334.79,191.05) ;
\draw [shift={(336.79,191.05)}, rotate = 180] [color={rgb, 255:red, 0; green, 0; blue, 0 }  ][line width=0.75]    (4.37,-1.96) .. controls (2.78,-0.92) and (1.32,-0.27) .. (0,0) .. controls (1.32,0.27) and (2.78,0.92) .. (4.37,1.96)   ;
\draw    (309.22,135.92) -- (334.79,135.92) ;
\draw [shift={(336.79,135.92)}, rotate = 180] [color={rgb, 255:red, 0; green, 0; blue, 0 }  ][line width=0.75]    (4.37,-1.96) .. controls (2.78,-0.92) and (1.32,-0.27) .. (0,0) .. controls (1.32,0.27) and (2.78,0.92) .. (4.37,1.96)   ;
\draw    (330.1,156.66) -- (316.32,172.28)(327.85,154.68) -- (314.07,170.3) ;
\draw  [dash pattern={on 0.84pt off 2.51pt}] (128.83,183.93) .. controls (128.83,175.85) and (135.38,169.3) .. (143.45,169.3) .. controls (151.53,169.3) and (158.08,175.85) .. (158.08,183.93) .. controls (158.08,192) and (151.53,198.55) .. (143.45,198.55) .. controls (135.38,198.55) and (128.83,192) .. (128.83,183.93) -- cycle ;

\draw (103.39,149.45) node [anchor=north west][inner sep=0.75pt]  [font=\normalsize]  {$f$};
\draw (134.32,205.89) node [anchor=north west][inner sep=0.75pt]  [font=\normalsize]  {$g$};
\draw (316.67,117.11) node [anchor=north west][inner sep=0.75pt]  [font=\small]  {$f$};
\draw (315.51,193.37) node [anchor=north west][inner sep=0.75pt]  [font=\small]  {$f$};
\draw (290.42,131.65) node [anchor=north west][inner sep=0.75pt]  [font=\small]  {$V$};
\draw (342.18,130.43) node [anchor=north west][inner sep=0.75pt]  [font=\small]  {$V$};
\draw (291.34,184.95) node [anchor=north west][inner sep=0.75pt]  [font=\small]  {$V$};
\draw (344.33,186.78) node [anchor=north west][inner sep=0.75pt]  [font=\small]  {$V$};
\draw (280.38,154.69) node [anchor=north west][inner sep=0.75pt]  [font=\small]  {$g$};
\draw (357.25,155.91) node [anchor=north west][inner sep=0.75pt]  [font=\small]  {$g$};
\draw (37,172.4) node [anchor=north west][inner sep=0.75pt]    {$C( f,g)$};

\end{tikzpicture}
\end{adjustbox}

\end{center}

For units $u,v\in F^\times$, retain the notation $C(u,v)$ for the standard
Steinberg symbol, realized by the two commuting determinant-one matrices
\begin{equation}\label{eq:ex-symbol}
\begin{aligned}
 A(u)&=\operatorname{diag}(u,u^{-1},1),&
 B(v)&=\operatorname{diag}(v,1,v^{-1}),\\
 C(u,v)&:=C(A(u),B(v))=\{u,v\}.
\end{aligned}
\end{equation}
Matsumoto's theorem says that these symbols generate $K_2(F)$, subject to
bilinearity and
\[
 C(u,1-u)=0\qquad(u\neq0,1)
\]
\cite[III.5.10 and III.6.1]{WeibelKbook}. Together with the assumed
Pontryagin--Thom comparison, this gives an existence proof that every
closed $2$-lasagna class has a representative which is a disjoint union
of $E(F)$-decorated tori, with orientations reversed for negative
coefficients. 

We seek explicit $3$-lasagnas realizing these relations. For $i\neq j$,
put $e_{ij}(a)=I+aE_{ij}$. With the commutator convention above,
\[
[e_{ij}(a),e_{kl}(b)]=
\begin{cases}
 I,&i\neq l,\ j\neq k,\\
 e_{il}(ab),&j=k,\ i\neq l,\\
 e_{kj}(-ab),&i=l,\ j\neq k,\\
 \begin{pmatrix}
 1+ab+a^2b^2&-a^2b\\
 ab^2&1-ab
 \end{pmatrix}_{ij},&i=l,\ j=k.
\end{cases}
\]
The subscript in the last line means that this is the $(i,j)$-block and
all other coordinates are unchanged. The first three identities lift to
the Steinberg generators $x_{ij}(a)$, together with
$x_{ij}(a)x_{ij}(b)=x_{ij}(a+b)$.

A commuting pair is \emph{overcommuting} if its lifts commute in every
central extension. Equivalently, its torus class in $H_2(G;\bZ)$ vanishes,
or it bounds a compact oriented $3$-manifold with an extending
$G$-local system \cite[Proposition 2.1]{overcommuting}. This equivalence is
an oriented-bordism statement; a prescribed stable boundary framing must
be checked separately.

\begin{lemma}[Handlebody extension criterion]\label{lem:ex-handlebody}
Let $S=\partial H$, where $H$ is a handlebody with a specified stable
framing, and equip $S$ with the induced framing. A $G$-local system on
$S$ extends over $H$ if and only if its monodromy is trivial on the
meridians of a complete compressing-disk system.
\end{lemma}
\begin{proof}
The map $\pi_1(S)\to\pi_1(H)$ is surjective, with kernel normally
generated by those meridians. Thus the representation factors through
$\pi_1(H)$ exactly under the stated condition. A map to the aspherical
space $BG$ realizing that representation gives the extension.
\end{proof}
A crossing-free string picture gives such a filling when its curves are
matched to the chosen disk system and its framing is the induced one.

\begin{proposition}[An explicit OCT filling]\label{prop:ex-OCT}
For every $a,b\in F$, the torus with holonomies
$e_{12}(a),e_{12}(b)$ bounds a local-system-decorated $3$-manifold. One
may take
\[
 W=\Sigma_{1,1}\times S^1,
\]
where $\Sigma_{1,1}$ is a once-punctured torus. Compatible product stable
framings make this a framed filling of the induced boundary torus.
\end{proposition}
\begin{proof}
Write
$\pi_1(W)=\langle x,y,z\mid[x,z]=[y,z]=1\rangle$, with boundary
coordinates $[x,y],z$. Define
\[
 \rho(x)=e_{13}(a),\qquad
 \rho(y)=e_{32}(1),\qquad
 \rho(z)=e_{12}(b).
\]
The first two matrices commute with $\rho(z)$, and
$[e_{13}(a),e_{32}(1)]=e_{12}(a)$, so this is a representation with the
required boundary holonomies. Choose a stable framing of $\Sigma_{1,1}$
inducing the bounding framing on its boundary and the bounding stable
framing on $S^1$; their product gives the claimed boundary framing.
\end{proof}
This supplies an explicit overcommuting-translation (OCT) block for the
stable elementary group. The Image Candy Surgery presentation below is
another proposed way to display the same elementary relation. We orient
cooriented strings so that crossing them in the positive normal direction
applies their displayed labels; reversing that direction inverts them.

\begin{center}

\tikzset{every picture/.style={line width=0.75pt}} 

\begin{adjustbox}{max totalsize={\linewidth}{0.86\textheight}}
\begin{tikzpicture}[x=0.75pt,y=0.75pt,yscale=-1,xscale=1]

\draw  [fill={rgb, 255:red, 248; green, 231; blue, 28 }  ,fill opacity=1 ] (99.8,160.7) .. controls (99.8,134.47) and (121.07,113.2) .. (147.3,113.2) .. controls (173.53,113.2) and (194.8,134.47) .. (194.8,160.7) .. controls (194.8,186.93) and (173.53,208.2) .. (147.3,208.2) .. controls (121.07,208.2) and (99.8,186.93) .. (99.8,160.7) -- cycle ;
\draw    (138.3,138) -- (138.3,178) ;
\draw    (158.3,138) -- (158.3,178) ;

\draw  [fill={rgb, 255:red, 248; green, 231; blue, 28 }  ,fill opacity=1 ] (110.8,55.63) .. controls (110.8,33.6) and (128.65,15.75) .. (150.67,15.75) .. controls (172.7,15.75) and (190.55,33.6) .. (190.55,55.63) .. controls (190.55,77.65) and (172.7,95.5) .. (150.67,95.5) .. controls (128.65,95.5) and (110.8,77.65) .. (110.8,55.63) -- cycle ;
\draw    (140.3,36) -- (140.3,76) ;
\draw    (160.3,36) -- (160.3,76) ;

\draw [color={rgb, 255:red, 208; green, 2; blue, 27 }  ,draw opacity=1 ]   (148.8,58.2) .. controls (148.2,95.2) and (176.8,86.2) .. (176.4,58.2) .. controls (176,30.48) and (151.69,11.38) .. (148.88,56.8) ;
\draw [shift={(148.8,58.2)}, rotate = 272.91] [color={rgb, 255:red, 208; green, 2; blue, 27 }  ,draw opacity=1 ][line width=0.75]    (10.93,-3.29) .. controls (6.95,-1.4) and (3.31,-0.3) .. (0,0) .. controls (3.31,0.3) and (6.95,1.4) .. (10.93,3.29)   ;
\draw [color={rgb, 255:red, 65; green, 117; blue, 5 }  ,draw opacity=1 ]   (140.3,56) .. controls (123.1,55.6) and (126.4,87.8) .. (152.8,88.2) .. controls (178.8,88.59) and (199.76,58.33) .. (162.06,56.08) ;
\draw [shift={(160.3,56)}, rotate = 2] [color={rgb, 255:red, 65; green, 117; blue, 5 }  ,draw opacity=1 ][line width=0.75]    (10.93,-3.29) .. controls (6.95,-1.4) and (3.31,-0.3) .. (0,0) .. controls (3.31,0.3) and (6.95,1.4) .. (10.93,3.29)   ;
\draw [color={rgb, 255:red, 208; green, 2; blue, 27 }  ,draw opacity=1 ]   (147.3,160.7) .. controls (146.7,197.7) and (175.3,188.7) .. (174.9,160.7) .. controls (174.5,132.98) and (150.19,113.88) .. (147.38,159.3) ;
\draw [shift={(147.3,160.7)}, rotate = 272.91] [color={rgb, 255:red, 208; green, 2; blue, 27 }  ,draw opacity=1 ][line width=0.75]    (10.93,-3.29) .. controls (6.95,-1.4) and (3.31,-0.3) .. (0,0) .. controls (3.31,0.3) and (6.95,1.4) .. (10.93,3.29)   ;
\draw [color={rgb, 255:red, 74; green, 144; blue, 226 }  ,draw opacity=1 ]   (138.8,158.5) .. controls (121.6,158.1) and (122.1,197.35) .. (148.5,197.75) .. controls (174.5,198.14) and (198.18,161.04) .. (160.56,158.59) ;
\draw [shift={(158.8,158.5)}, rotate = 2] [color={rgb, 255:red, 74; green, 144; blue, 226 }  ,draw opacity=1 ][line width=0.75]    (10.93,-3.29) .. controls (6.95,-1.4) and (3.31,-0.3) .. (0,0) .. controls (3.31,0.3) and (6.95,1.4) .. (10.93,3.29)   ;
\draw [color={rgb, 255:red, 126; green, 211; blue, 33 }  ,draw opacity=1 ]   (139.3,151) .. controls (100.5,149.75) and (119.5,204.25) .. (150.5,205.25) .. controls (181.19,206.24) and (205.51,154.3) .. (160.68,151.08) ;
\draw [shift={(159.3,151)}, rotate = 2.76] [color={rgb, 255:red, 126; green, 211; blue, 33 }  ,draw opacity=1 ][line width=0.75]    (10.93,-3.29) .. controls (6.95,-1.4) and (3.31,-0.3) .. (0,0) .. controls (3.31,0.3) and (6.95,1.4) .. (10.93,3.29)   ;
\draw [color={rgb, 255:red, 126; green, 211; blue, 33 }  ,draw opacity=1 ]   (158.5,169.25) .. controls (175.5,172.75) and (158,194.25) .. (147.5,192.25) .. controls (137.37,190.32) and (129.1,185.13) .. (136.62,170.39) ;
\draw [shift={(137.5,168.75)}, rotate = 119.36] [color={rgb, 255:red, 126; green, 211; blue, 33 }  ,draw opacity=1 ][line width=0.75]    (10.93,-4.9) .. controls (6.95,-2.3) and (3.31,-0.67) .. (0,0) .. controls (3.31,0.67) and (6.95,2.3) .. (10.93,4.9)   ;
\draw [color={rgb, 255:red, 74; green, 144; blue, 226 }  ,draw opacity=1 ]   (159.5,141.75) .. controls (182.5,138.25) and (180.5,118.75) .. (149,121.75) .. controls (117.97,124.71) and (99.56,149) .. (135.33,143.28) ;
\draw [shift={(137,143)}, rotate = 169.93] [color={rgb, 255:red, 74; green, 144; blue, 226 }  ,draw opacity=1 ][line width=0.75]    (10.93,-3.29) .. controls (6.95,-1.4) and (3.31,-0.3) .. (0,0) .. controls (3.31,0.3) and (6.95,1.4) .. (10.93,3.29)   ;
\draw    (122.5,90.25) -- (122.95,109.75) ;
\draw [shift={(123,111.75)}, rotate = 268.67] [color={rgb, 255:red, 0; green, 0; blue, 0 }  ][line width=0.75]    (7.65,-3.43) .. controls (4.86,-1.61) and (2.31,-0.47) .. (0,0) .. controls (2.31,0.47) and (4.86,1.61) .. (7.65,3.43)   ;
\draw  [fill={rgb, 255:red, 248; green, 231; blue, 28 }  ,fill opacity=1 ] (406,183.08) .. controls (367.74,205.17) and (318.81,192.06) .. (296.72,153.8) .. controls (274.63,115.54) and (287.74,66.61) .. (326,44.52) .. controls (364.26,22.43) and (413.19,35.54) .. (435.28,73.8) .. controls (457.37,112.06) and (444.26,160.99) .. (406,183.08) -- cycle ;
\draw    (319.04,152.46) -- (353.68,132.46) ;
\draw    (309.04,135.14) -- (343.68,115.14) ;

\draw    (422.96,135.14) -- (388.32,115.14) ;
\draw    (412.96,152.46) -- (378.32,132.46) ;

\draw    (356,53.8) -- (356,93.8) ;
\draw    (376,53.8) -- (376,93.8) ;

\draw [color={rgb, 255:red, 74; green, 144; blue, 226 }  ,draw opacity=1 ]   (376.14,59.84) .. controls (397.26,62.8) and (405.77,95.12) .. (393.63,116.05) ;
\draw [shift={(392.65,117.64)}, rotate = 303.03] [color={rgb, 255:red, 74; green, 144; blue, 226 }  ,draw opacity=1 ][line width=0.75]    (10.93,-3.29) .. controls (6.95,-1.4) and (3.31,-0.3) .. (0,0) .. controls (3.31,0.3) and (6.95,1.4) .. (10.93,3.29)   ;
\draw [color={rgb, 255:red, 74; green, 144; blue, 226 }  ,draw opacity=1 ]   (382.5,136.03) .. controls (370.58,149.47) and (379.25,118.45) .. (383.84,101.7) .. controls (388.15,85.96) and (385.34,86.09) .. (377.59,80.36) ;
\draw [shift={(376.05,79.19)}, rotate = 38.13] [color={rgb, 255:red, 74; green, 144; blue, 226 }  ,draw opacity=1 ][line width=0.75]    (10.93,-3.29) .. controls (6.95,-1.4) and (3.31,-0.3) .. (0,0) .. controls (3.31,0.3) and (6.95,1.4) .. (10.93,3.29)   ;
\draw [color={rgb, 255:red, 74; green, 144; blue, 226 }  ,draw opacity=1 ]   (355.64,77.34) .. controls (346.94,78.29) and (323.98,90.12) .. (311.86,106.03) .. controls (299.74,121.94) and (301.82,134.24) .. (306.13,139.51) .. controls (310.44,144.78) and (344.28,115.81) .. (352.82,125.82) .. controls (361.36,135.83) and (368.28,145.8) .. (380.36,149) .. controls (391.59,151.98) and (392.12,150.86) .. (398.4,145.52) ;
\draw [shift={(399.92,144.24)}, rotate = 140.45] [color={rgb, 255:red, 74; green, 144; blue, 226 }  ,draw opacity=1 ][line width=0.75]    (10.93,-3.29) .. controls (6.95,-1.4) and (3.31,-0.3) .. (0,0) .. controls (3.31,0.3) and (6.95,1.4) .. (10.93,3.29)   ;
\draw [color={rgb, 255:red, 74; green, 144; blue, 226 }  ,draw opacity=1 ]   (410.07,127.47) .. controls (416.58,113.53) and (417.72,81.16) .. (403.12,61.1) .. controls (388.53,41.04) and (368.08,41.8) .. (351.2,45.44) .. controls (335.07,48.92) and (340.38,54.67) .. (353.5,58.44) ;
\draw [shift={(355.4,58.95)}, rotate = 194.48] [color={rgb, 255:red, 74; green, 144; blue, 226 }  ,draw opacity=1 ][line width=0.75]    (10.93,-3.29) .. controls (6.95,-1.4) and (3.31,-0.3) .. (0,0) .. controls (3.31,0.3) and (6.95,1.4) .. (10.93,3.29)   ;
\draw [color={rgb, 255:red, 126; green, 211; blue, 33 }  ,draw opacity=1 ]   (315,130.8) .. controls (305.27,109.15) and (332.39,95.12) .. (353.58,88.87) ;
\draw [shift={(355.2,88.4)}, rotate = 164.34] [color={rgb, 255:red, 126; green, 211; blue, 33 }  ,draw opacity=1 ][line width=0.75]    (10.93,-3.29) .. controls (6.95,-1.4) and (3.31,-0.3) .. (0,0) .. controls (3.31,0.3) and (6.95,1.4) .. (10.93,3.29)   ;
\draw [color={rgb, 255:red, 126; green, 211; blue, 33 }  ,draw opacity=1 ]   (376.2,88) .. controls (393.8,91.6) and (362.6,99.6) .. (345.8,104) .. controls (330.01,108.14) and (331.53,110.51) .. (330.45,120.08) ;
\draw [shift={(330.2,122)}, rotate = 278.13] [color={rgb, 255:red, 126; green, 211; blue, 33 }  ,draw opacity=1 ][line width=0.75]    (10.93,-3.29) .. controls (6.95,-1.4) and (3.31,-0.3) .. (0,0) .. controls (3.31,0.3) and (6.95,1.4) .. (10.93,3.29)   ;
\draw [color={rgb, 255:red, 126; green, 211; blue, 33 }  ,draw opacity=1 ]   (338.8,140.6) .. controls (343.56,147.09) and (364.45,159.81) .. (383.01,162.87) .. controls (401.58,165.93) and (414.73,159.45) .. (417.4,152.39) .. controls (420.07,145.34) and (380.6,131.2) .. (385,118.8) .. controls (389.4,106.4) and (394.67,92.46) .. (391.4,80.4) .. controls (388.51,69.73) and (383.93,69.9) .. (378.96,69.44) ;
\draw [shift={(377,69.2)}, rotate = 9.25] [color={rgb, 255:red, 126; green, 211; blue, 33 }  ,draw opacity=1 ][line width=0.75]    (10.93,-3.29) .. controls (6.95,-1.4) and (3.31,-0.3) .. (0,0) .. controls (3.31,0.3) and (6.95,1.4) .. (10.93,3.29)   ;
\draw [color={rgb, 255:red, 126; green, 211; blue, 33 }  ,draw opacity=1 ]   (355,68.4) .. controls (339.68,69.73) and (311.07,84.93) .. (301,107.6) .. controls (290.93,130.27) and (301.8,147.6) .. (313.4,160.4) .. controls (324.48,172.63) and (326.8,165.15) .. (323.5,151.91) ;
\draw [shift={(323,150)}, rotate = 74.48] [color={rgb, 255:red, 126; green, 211; blue, 33 }  ,draw opacity=1 ][line width=0.75]    (10.93,-3.29) .. controls (6.95,-1.4) and (3.31,-0.3) .. (0,0) .. controls (3.31,0.3) and (6.95,1.4) .. (10.93,3.29)   ;
\draw [color={rgb, 255:red, 208; green, 2; blue, 27 }  ,draw opacity=1 ]   (407.32,149.76) .. controls (396.3,164.23) and (366.44,151.37) .. (350.22,136.32) ;
\draw [shift={(349,135.16)}, rotate = 44.34] [color={rgb, 255:red, 208; green, 2; blue, 27 }  ,draw opacity=1 ][line width=0.75]    (10.93,-3.29) .. controls (6.95,-1.4) and (3.31,-0.3) .. (0,0) .. controls (3.31,0.3) and (6.95,1.4) .. (10.93,3.29)   ;
\draw [color={rgb, 255:red, 208; green, 2; blue, 27 }  ,draw opacity=1 ]   (338.16,117.17) .. controls (332.47,100.13) and (355,123.15) .. (367.21,135.5) .. controls (378.69,147.11) and (379.98,144.6) .. (388.81,140.76) ;
\draw [shift={(390.6,140.01)}, rotate = 158.13] [color={rgb, 255:red, 208; green, 2; blue, 27 }  ,draw opacity=1 ][line width=0.75]    (10.93,-3.29) .. controls (6.95,-1.4) and (3.31,-0.3) .. (0,0) .. controls (3.31,0.3) and (6.95,1.4) .. (10.93,3.29)   ;
\draw [color={rgb, 255:red, 208; green, 2; blue, 27 }  ,draw opacity=1 ]   (402.41,123.26) .. controls (404.48,118.56) and (412.34,96.06) .. (405.01,78.57) .. controls (397.69,61.09) and (383.72,47.61) .. (373.32,49.3) .. controls (362.93,50.99) and (374.77,94.19) .. (361.83,96.58) .. controls (348.89,98.97) and (334.46,100.37) .. (326.71,108.54) .. controls (319.42,116.22) and (320.25,116.77) .. (321.99,126.23) ;
\draw [shift={(322.33,128.16)}, rotate = 260.32] [color={rgb, 255:red, 208; green, 2; blue, 27 }  ,draw opacity=1 ][line width=0.75]    (10.93,-3.29) .. controls (6.95,-1.4) and (3.31,-0.3) .. (0,0) .. controls (3.31,0.3) and (6.95,1.4) .. (10.93,3.29)   ;
\draw [color={rgb, 255:red, 208; green, 2; blue, 27 }  ,draw opacity=1 ]   (331.78,145.33) .. controls (340.6,157.94) and (368.06,175.11) .. (392.73,172.5) .. controls (417.4,169.89) and (426.97,151.81) .. (432.26,135.36) .. controls (437.3,119.66) and (429.67,121.38) .. (419.85,130.86) ;
\draw [shift={(418.45,132.25)}, rotate = 314.48] [color={rgb, 255:red, 208; green, 2; blue, 27 }  ,draw opacity=1 ][line width=0.75]    (10.93,-3.29) .. controls (6.95,-1.4) and (3.31,-0.3) .. (0,0) .. controls (3.31,0.3) and (6.95,1.4) .. (10.93,3.29)   ;
\draw    (202.7,159.5) -- (271.43,147.1) ;
\draw [shift={(273.4,146.75)}, rotate = 169.78] [color={rgb, 255:red, 0; green, 0; blue, 0 }  ][line width=0.75]    (10.93,-3.29) .. controls (6.95,-1.4) and (3.31,-0.3) .. (0,0) .. controls (3.31,0.3) and (6.95,1.4) .. (10.93,3.29)   ;
\draw  [dash pattern={on 0.84pt off 2.51pt}]  (170,138) .. controls (194.4,128.25) and (195.9,164.75) .. (175.4,163.75) ;
\draw  [dash pattern={on 0.84pt off 2.51pt}]  (175,154) .. controls (157.4,150.75) and (147.9,173.25) .. (161.4,185.75) ;

\draw (181,16.4) node [anchor=north west][inner sep=0.75pt]    {$=C( e_{12}( a) ,e_{12}( b))$};
\draw (193.3,118.7) node [anchor=north west][inner sep=0.75pt]   [align=left] {Image Candy};

\end{tikzpicture}
\end{adjustbox}

\end{center}

\paragraph{Ghys-type fillings and the factor of two.}
Let
\[
 h(u)=\operatorname{diag}(u,u^{-1}).
\]
The diagonal commutator formula gives
\begin{equation}\label{eq:ex-Ghys-double}
 C(h(u),h(v))=2\{u,v\}=2C(u,v).
\end{equation}
Equivalently, when the multiplicative geometric comparison of the next
subsection is used, the two diagonal summands contribute
$C(u,v)+C(u^{-1},v^{-1})=2C(u,v)$.
For $v=1-u$, this class vanishes. A filling of this diagonal torus will
be called a \emph{Ghys lasagna} $G(u)$. It fills the doubled symbol;
there is no division by $2$ in this argument.

Liechti--March\'e construct oriented Ghys manifolds and relate their
Heegaard complexity to overcommutation length \cite{overcommuting}.
Their unstable $\mathrm{SL}_2$ constructions have the field hypotheses
stated there.  The stable
$K_2$-vanishing above is valid for every field and every $u\neq0,1$.

Following the proposed surgery calculation, put
\[
 x(a)=e_{12}(a),\qquad y(a)=e_{21}(a),\qquad
 w(u)=x(u)y(-u^{-1})x(u),\qquad h(u)=w(u)w(-1).
\]
Matrix multiplication verifies the displayed formula for $h(u)$.
Expanding these words suggests resolving the diagonal crossings by
Image Candy Surgery and elementary OCT blocks.

\begin{center}

\tikzset{every picture/.style={line width=0.75pt}} 

\begin{adjustbox}{max totalsize={\linewidth}{0.86\textheight}}

\end{adjustbox}

\end{center}

The proposed picture is intended to use four OCT resolutions and produce

More generally, a finite derivation of a Steinberg-group identity can
suggest a sequence of local replacements. Translating it into a
geometric nullbordism requires a framed block for each relation and
compatible gluings of their entire boundaries.

Finally, fix a representative of the single symbol $C(u,1-u)$, for
example the torus from \eqref{eq:ex-symbol}. Its vanishing and the
assumed Pontryagin--Thom theorem give a $3$-lasagna $SR(u)$ with this
boundary. An explicit construction of $SR(u)$, rather than only a
filling of $2C(u,1-u)$, remains a separate geometric goal. One can obtain
such a construction by writing the proof in terms of Steinberg-group
relations and resolving the corresponding crossings. This produces a
large crossing-free surface, as in the method of
\cite[Section~5.4]{overcommuting}.

Two choices of $SR(u)$ differ, after gluing along their identified boundaries, by a closed $3$-lasagna. This ambiguity is important for the next subsection.

\begin{center}

\tikzset{every picture/.style={line width=0.75pt}} 

\begin{adjustbox}{max totalsize={\linewidth}{0.86\textheight}}

\end{adjustbox}

\end{center}

\subsection{\texorpdfstring{Tiramisu, $K_3(F)$, the Bloch Group, and Higher Degrees}{Tiramisu, K3(F), the Bloch Group, and Higher Degrees}}\label{subsec:k3}

\paragraph{Products and Milnor $K$-theory.}
For decorated closed branefolds, define
\[
 (M,L_1)\boxtimes(N,L_2)
 :=\bigl(M\times N,L_1\boxtimes L_2\bigr),
 \qquad
 L_1\boxtimes L_2
 =\operatorname{pr}_1^*L_1\otimes_F\operatorname{pr}_2^*L_2.
\]
Use the product stable framing and the product local diagrams. Tensor
product distributes over direct sum, so the branch identifications
remain compatible. Products of bordisms show that this descends to
\[
 \boxtimes:\Xi_p^{\Loc_F}\otimes\Xi_q^{\Loc_F}
 \longrightarrow\Xi_{p+q}^{\Loc_F}.
\]

\begin{proposition}[Milnor symbols as decorated tori]
With this multiplicative comparison, the formula
\[
 T:K_d^M(F)\longrightarrow\Xi_d^{\Loc_F},
 \qquad
 \{a_1,\ldots,a_d\}
 \longmapsto
 L_\infty(a_1)\boxtimes\cdots\boxtimes L_\infty(a_d)
\]
is well defined and corresponds to the standard product map
$K_d^M(F)\to K_d(F)$.
\end{proposition}
\begin{proof}
The circle relations make the formula multilinear in the multiplicative
groups of units. The degree-two Steinberg relation is represented by a
nullbordism $SR(u)$; multiplying that bordism by the remaining factors
kills the corresponding relation in every degree. If a different torus
representative of $C(u,1-u)$ is used, choose the connecting bordism
supplied by the same comparison. These are exactly the defining Milnor
relations. The product of the circles is the product-framed torus
$\mathbb T^d$ with a rank-one local system whose monodromies are
$a_1,\ldots,a_d$.
\end{proof}

In particular, scalar local systems on surfaces need not be nullbordant.
Over $F=\mathbb R$, the product-framed torus with monodromies $-1,-1$
represents $\{-1,-1\}\neq0$; the real Hilbert symbol detects this class
\cite[III.6.2.1]{WeibelKbook}. This explains why the scalar correction
in \eqref{eq:ex-det-correction} cannot simply be discarded.

Define the indecomposable quotient by
\[
 K_d^{\mathrm{ind}}(F)
 :=\operatorname{coker}\bigl(K_d^M(F)\longrightarrow K_d(F)\bigr).
\]
The canonical map from Milnor $K$-theory is not assumed injective, so this
notation does not regard $K_d^M(F)$ as a specified subgroup.

\paragraph{The Bloch group and choices of Steinberg fillings.}
For the remainder of this subsection, assume that $F$ is infinite. Define
$PB(F)$ to be the quotient of $\bZ[F\setminus\{0,1\}]$ by the five-term
relations
\[
 [x]-[y]+\left[\frac yx\right]
 -\left[\frac{1-x^{-1}}{1-y^{-1}}\right]
 +\left[\frac{1-x}{1-y}\right]=0,
 \qquad x,y\neq0,1,\quad x\neq y.
\]
For an integral convention retaining the relevant $2$-torsion, set
\[
 A(F)=
 \frac{F^\times\otimes_{\bZ}F^\times}
 {\langle a\otimes b+b\otimes a\rangle},
 \qquad
 \delta:PB(F)\longrightarrow A(F),\quad
 [u]\longmapsto u\otimes(1-u),
\]
and put $B(F)=\ker\delta$. The five-term expressions map to zero in
$A(F)$; replacing $A(F)$ by an exterior square without adjusting
conventions can lose $2$-torsion \cite[\S3]{MirzaiiMokari}.

Suslin's Bloch--Wigner exact sequence gives
\begin{equation}\label{eq:ex-Suslin}
 0\longrightarrow
 \widetilde{\operatorname{Tor}}{}_1^{\bZ}
       (\mu(F),\mu(F))
 \longrightarrow K_3^{\mathrm{ind}}(F)
 \xrightarrow{\rho}B(F)\longrightarrow0,
\end{equation}
where $\mu(F)$ is the group of roots of unity and the tilde denotes the
standard modified torsion term, with its usual characteristic-dependent
convention \cite{SuslinBloch,MirzaiiMokari}. This specifies that term
independently of $\ker\rho$.

It is natural to look for geometric representatives of the torsion term
using local systems on lens spaces, but no such identification is proved
here. For the Bloch group, the proposed interpretation is that $[u]$
records a chosen witness $SR(u)$ of the Steinberg relation.
Since $SR(u)$ has boundary, $\rho(SR(u))$ is not defined by
\eqref{eq:ex-Suslin}: one needs a relative or chain-level refinement of
$\rho$ with the formal rule
\[
 SR(u)\longleftrightarrow[u].
\]

For a combination $\sum_i n_i[u_i]\in\ker\delta$, the intended geometric
construction glues the signed fillings $n_iSR(u_i)$ using specified
bordisms implementing the tensor, symmetry, and bilinearity relations
among their boundary diagrams. The equality $\sum_i n_i\delta[u_i]=0$
is algebraic; the actual collar identifications and those gluing bordisms
are additional choices. Different choices can change the resulting
closed $3$-lasagna by a $K_3(F)$-class.

\begin{conjecture}[Coherent five-term fillings]\label{conj:ex-five-term}
The representatives $SR(u)$ and their boundary-gluing data can be chosen
coherently, up to $K^M_3(F)$ and
$\widetilde{\operatorname{Tor}}{}_1^{\bZ} (\mu(F),\mu(F))$, so that,
for $x,y\in F\setminus\{0,1\}$ with $x\neq y$, the
closed $3$-lasagna obtained from
\[
\begin{aligned}
 SR(x)\sqcup -SR(y)\sqcup SR\!\left(\frac yx\right)
 \sqcup -SR\!\left(\frac{1-x^{-1}}{1-y^{-1}}\right)
 \sqcup SR\!\left(\frac{1-x}{1-y}\right)
\end{aligned}
\]
by the specified local tensor and bilinearity moves bounds a framed
$4$-lasagna. There is an analogous conjecture for the doubled-symbol
fillings $G(u)$, with their compatible boundary conventions.
\end{conjecture}
The conjecture concerns one coherent system of local choices, not every
possible filling or closure. Even if a closed comparison has zero image
in $B(F)$, \eqref{eq:ex-Suslin} leaves a torsion ambiguity in
$K_3^{\mathrm{ind}}(F)$; decomposable $K_3$-classes give an additional
ambiguity before passing to that quotient.

\section{Further Directions}\label{sec:further}

The constructions above suggest several questions beyond the examples
considered here. The following directions concern both the geometry of
lasagnas and the coefficient structures they can carry.

\paragraph{Three-dimensional visualization and animation.}
A useful next step is to produce three-dimensional models of the surgery
blocks and animations of their boundary diagrams. Such models should show
not only the underlying surfaces but also their framings, local-system
labels, and attaching maps. This would make the OCT and Ghys-type
constructions easier to inspect and help verify the proposed handle counts.
An interactive presentation could also record the sequence of local moves
associated with a group-theoretic derivation.

\paragraph{Beyond the framed case.}
The framework should be compared with bordisms carrying other tangential
structures. The proposed starting point is a category $\cC$ with an
$O$-action $\cC\to BO$. One would need to specify how the coefficient
structure interacts with changes of tangential data, and which framing
relations survive after passing to oriented, spin, or other structured
bordisms. The ordinary cobordism hypothesis with tangential structures
provides the guiding comparison \cite{Lurie}.

\paragraph{Waldhausen categories and the $S_\bullet$-construction.}
For a Waldhausen category, the coefficient data should include
cofibrations, weak equivalences, and cofiber sequences. A geometric model
should represent the simplices of the $S_\bullet$-construction and their
face and degeneracy operations, rather than encode direct sum alone.
This would extend the present group-completion picture toward the
additivity relations of Waldhausen $K$-theory \cite{Waldhausen}.

\paragraph{Other coefficient categories.}
It would be useful to work out further examples, including the algebraic
$K$-theory of spaces and categories arising from coherent sheaves.
Waldhausen's construction provides a starting point for the former
\cite{Waldhausen}. For the latter, the geometric decorations should retain
the relevant exact or derived structure. Comparing these examples with
vector-space decorations may clarify which features are consequences of
additive group completion and which require additional coefficient data.

\paragraph{Trace maps to $\mathrm{THH}$ and $\mathrm{TC}$.}
The trace should admit a description in terms of cyclic diagrams of
coefficient morphisms. One would like to send a lasagna to a coherent family
of such diagrams and identify the resulting map with the Dennis trace.
For topological cyclic homology, the model must additionally retain the
circle action and its cyclotomic structure; identifying the underlying
cyclic diagrams alone would not suffice. The geometric construction of the
cyclotomic trace and the modern formulation of $\mathrm{TC}$ offer useful
comparisons \cite{AMGRTrace,NikolausScholze}.

\paragraph{Quiver representations.}
Quiver representations give concrete coefficient categories in which
objects, morphisms, and relations can be displayed by diagrams. It would be
interesting to compare these diagrams with lasagna junctions and determine
how changes of representation or relations in the quiver appear as local
geometric moves. Small examples may also provide tests of the dependence
on the chosen categorical presentation.

\paragraph{Sheaf-valued and motivic versions.}
A further goal is a relative construction for sheaves of higher categories,
with a corresponding geometric or motivic cobordism hypothesis. Such a
construction would need a specified base, topology, and descent theorem,
as well as compatibility of the local defect structures with base change.
The locality methods of geometric field theory provide a model for this
problem \cite{GPGeometric,GPLocality}; a motivic formulation would require
additional geometric input rather than a formal change of terminology.

\bibliographystyle{alpha}
\bibliography{references}
\end{document}